\documentclass{article}
\usepackage{amsmath}
\usepackage{amsfonts}
\usepackage{amsthm}
\usepackage{amssymb}
\usepackage{mathabx}
\usepackage{enumerate}

\newtheorem{thm}{Theorem}[section]
\newtheorem{prop}[thm]{Proposition}
\newtheorem{lem}[thm]{Lemma}
\newtheorem{cor}[thm]{Corollary}

\newtheorem*{thmnn}{Theorem}

\theoremstyle{definition}
\newtheorem{defn}[thm]{Definition}
\newtheorem{rmk}[thm]{Remark}
\newtheorem{notn}[thm]{Notation}
\newtheorem{ex}[thm]{Example}

\begin{document}

\title{The Antipodes of $q$-Quasi-Symmetric Functions and Non-Commutative Quasi-Symmetric Functions}

\author{Shaul Zemel}

\maketitle

\begin{abstract}

\end{abstract}

\section*{Introduction}

It is a classical result that the ring $\mathtt{Sym}$ of symmetric functions is a commutative and co-commutative Hopf algebra, which is also self-dual. This Hopf algebra was generalized in many directions. An important Hopf algebra that contains $\mathtt{Sym}$ is the algebra $\mathtt{QSym}$ of quasi-symmetric functions, defined initially by \cite{[Ge]} in relation to P-partitions. The Hopf algebra that is dual to $\mathtt{QSym}$ is the algebra $\mathtt{NSym}$ of non-commutative symmetric functions, first defined in \cite{[GKLLRT]}. The latter is identified with the Hopf algebra of the direct sum of the Solomon descent algebras, as described in the seminal work \cite{[MR]}. This paper also constructs two additional Hopf algebras which are dual to one another, one containing $\mathtt{NSym}$ and the other projecting onto $\mathtt{QSym}$. We also mention the related Hopf algebras from \cite{[PR]}, admitting bases that are parameterized by Young tableaux.

Recall that the shuffle algebra, as defined initially in \cite{[R]}, is the dual to the natural Hopf algebra structure on the free algebra on a vector space. The algebra $\mathtt{QSym}$ is an example of a the more general concept of a quasi-shuffle algebra, sometimes called a stuffle algebra, as in \cite{[Ho1]} (see also \cite{[Ho2]} and others for the relations between such structures and multiple zeta values). A more general quasi-shuffle algebra is its commutative $q$-deformation, considered in \cite{[BDM]} and others, which we shall denote here by $\mathtt{QSym}_{q}$.

There is another, non-commutative type of $q$-deformation, which is sometimes referred to with the adjective ``quantum''. It was defined in \cite{[DKKT]} for the shuffle algebra, and \cite{[TU]} considered the non-commutative $q$-deformation $\mathtt{QSym}^{(q)}$ of the quasi-shuffle algebra $\mathtt{QSym}$. We remark that the latter reference used $\mathtt{QSym}_{q}$ for that algebra, but as we will write bases for it whose indices will be in subscripts, the indication that they come from this algebra will be with the superscript $(q)$, whence our choice of notation for the algebra. Note that this algebra is not a Hopf algebra in the usual sense, but rather a deformed structure called a $q$-Hopf algebra (see Remark \ref{qHopf} below), which is related to braided monoidal categories (see, e.g., \cite{[AgM]} for the precise relation and its properties). For more on quasi-shuffle algebras and related objects see, e.g., \cite{[DEMT]}.

\medskip

In another direction, the paper \cite{[RS]} constructed a Hopf algebra $\mathtt{NCSym}$ of another type of non-commutative symmetric functions, which contains $\mathtt{NSym}$ and is based on some constructions from the much older references \cite{[W]} and \cite{[BC]} (see \cite{[BRRZ]} for some connections between $\mathtt{NSym}$ and $\mathtt{NCSym}$). It was generalized in \cite{[BZ]} to the algebra $\mathtt{NCQSym}$ of non-commutative quasi-symmetric functions, also known as the algebra of word quasi-symmetric functions and denoted by $\mathtt{WQSym}$. Recall that \cite{[MR]} defined two algebras, one containing $\mathtt{NSym}$ and the other admitting $\mathtt{QSym}$ as a quotient. The algebra $\mathtt{NCQSym}$ plays both roles at once, as it contains $\mathtt{NCSym}$ and hence also $\mathtt{NSym}$, and dividing it by the relation making the variables commute produces $\mathtt{QSym}$.

For the relations between this larger Hopf algebra and iterated integrals, mould calculus, and polyhedral cones see, \cite{[MNT]}. For more on many of these Hopf algebras and these properties, see the book \cite{[LMvW]}, the lecture notes \cite{[GR]}, or the bachelor thesis \cite{[D]} (though this list of references is far from being comprehensive). 

\medskip

All of these algebras are connected graded bi-algebras, from which the Hopf algebra property readily follows (they are also combinatorial Hopf algebras, in the terminology of \cite{[ABS]}, \cite{[BHRZ]}, and others). However, an explicit formula for the antipode in a basis for these algebras is not always easy to describe. There are simple formulae for the algebras $\mathtt{Sym}$, $\mathtt{NSym}$, and $\mathtt{QSym}$, but a general one is not known for many combinatorial Hopf algebras. For some other algebras there are expressions in the literature, but they are harder to construct, like those from \cite{[AgS]} for the algebra from \cite{[MR]}. There is a general mechanism for obtaining antipode formulae as described in \cite{[T]}, but applying this formula typically involves a great deal of cancelation. The paper \cite{[BS]} gathers some of the antipode formulae for combinatorial Hopf algebras, and establishes them using a cancelation-free version of the formula from \cite{[T]} (see also the book \cite{[AgM]} for related constructions).

As another type of generalization of several of these algebras, we also mention the K-theoretical version of $\mathtt{Sym}$, $\mathtt{NSym}$, $\mathtt{QSym}$, and the algebras from \cite{[MR]} that were defined in \cite{[LP]}, for which formulae for the antipode are given in \cite{[P]}.

The goal of this paper is to apply another way of determining antipode formulae for connected graded Hopf algebras for establishing such formulae for bases of $\mathtt{QSym}_{q}$ and $\mathtt{QSym}^{(q)}$, as well as for a part of a basis for $\mathtt{NCQSym}$ (or equivalently $\mathtt{WQSym}$). This gives a proof, which as far as I know is new, for the known antipode formula for $\mathtt{QSym}$ (as the case $q=1$ of the former two algebras). We also work over a general commutative ring, so we included the basic fact determining the co-unit in a connected graded co-algebra over any such ring in Lemma \ref{propgr} below.

\medskip

We now describe the construction and the resulting formulae in more detail. Consider the ring of power series of bounded degree in infinitely many variables, say over $\mathbb{Q}$, inside which the ring of symmetric functions is the invariants under the natural action of the infinite symmetric group. There is another action of that group on the former ring, which is described and investigated in Section 3 of \cite{[Hi2]}, and whose invariant subspace there is the ring of quasi-symmetric functions.

Like symmetric functions, the ring of quasi-symmetric functions admits a monomial basis $\{M_{\alpha}\}_{\alpha}$, here indexed by all compositions of all integers. We denote the assertion that another composition $\beta$ of the same number $n$ is a refinement of $\alpha$ by the symbol $\beta\preceq\alpha$. Then the fundamental quasi-symmetric function $F_{\alpha}$, considered first in \cite{[Ge]}, is defined to be $\sum_{\beta\preceq\alpha}M_{\beta}$.

The structure of bi-algebra on that ring, which produces the Hopf algebra $\mathtt{QSym}$ (and contains $\mathtt{Sym}$ as a Hopf subalgebra), is described in \cite{[MR]}. We will give the formula for the comultiplication below (the counit is simple for every connected graded bi-algebra), but for relating to the main result of this paper, we state here the formula for the antipode. We recall that for a composition $\alpha$ of $n$, reversing it and then taking the composition associated with the complement of the corresponding set produces the transpose composition $\alpha^{t}$ of $n$. Then the following result of \cite{[MR]}, proven independently in \cite{[E]} (and then later as Theorem 5.1 of \cite{[BS]}, and see also \cite{[Gr]}), states the following.
\begin{thmnn}
The antipode on $\mathtt{QSym}$ sends the fundamental quasi-symmetric function $F_{\alpha}$ associated with a composition $\alpha$ of $n$ to $(-1)^{n}F_{\alpha^{t}}$.
\end{thmnn}
This describes the antipode on all of $\mathtt{QSym}$, since $\{F_{\alpha}\}_{\alpha}$ is also a basis for this Hopf algebra. One easily deduces the formula for the antipode of $M_{\alpha}$ as $(-1)^{\ell(\alpha)}\sum_{\beta\succeq\alpha^{r}}M_{\beta}$, where $\alpha^{r}$ and $\ell(\alpha)$ are the reversion and length of $\alpha$ respectively---see, e.g., Theorem 4.1 of \cite{[BS]}.

\medskip

Our first result, which is very simple, is that the commutative $q$-deformation $\mathtt{QSym}_{q}$ admits, in a similar basis, the exact same antipode formula. This is not surprising, since after inverting $q$, this Hopf algebra becomes isomorphic to $\mathtt{QSym}$, under an isomorphism taking such a basis to the corresponding one.

We thus turn to the non-commutative (or quantum) $q$-deformation $\mathtt{QSym}^{(q)}$ of $\mathtt{QSym}$. Consider power series, again of bounded degree, in infinitely many variables, but which no longer commute but rather satisfy $x_{j}x_{i}=qx_{i}x_{j}$ wherever $j>i$. The action from \cite{[Hi2]} again defines the invariant subspaces, which are closed under multiplication and spanned by the basis of monomial $q$-quasi-symmetric functions $\{M_{\alpha}^{(q)}\}_{\alpha}$, indexed by compositions of integers. The multiplication there, which is described in \cite{[TU]}, with the same co-multiplication as in $\mathtt{QSym}$, does not yield a Hopf algebra per se, but rather a more general structure, described in Remark \ref{qHopf} below. Such objects are referred to as a $q$-Hopf algebra in \cite{[EK]}, \cite{[L]}, and others (which is a special case of the $\mathbb{Z}$-graded colored Hopf algebras from \cite{[Mo]}, and another categorical terminology in \cite{[AgM]}), and are related to quantum quasi-shuffles via \cite{[JRZ]}. We are interested in the formula for the antipode.

We define the fundamental $q$-quasi-symmetric function $F_{\alpha}^{(q)}:=\sum_{\beta\preceq\alpha}M_{\beta}^{(q)}$ in a similar manner to those from $\mathtt{QSym}$, and $\{F_{\alpha}^{(q)}\}_{\alpha}$ is another basis for $\mathtt{QSym}^{(q)}$. Moreover, every composition $\alpha$ has an inversion number $\operatorname{inv}(\alpha)$ (see Definition \ref{fundQSymq}  below). This number satisfies $\operatorname{inv}(\beta)\geq\operatorname{inv}(\alpha)$ wherever $\beta\preceq\alpha$, and we set $\widetilde{F}_{\alpha}^{(q)}:=\sum_{\beta\preceq\alpha}q^{\operatorname{inv}(\beta)-\operatorname{inv}(\alpha)}M_{\beta}^{(q)}$, so that $\{F_{\alpha}^{(q)}\}_{\alpha}$ is also a basis for $\mathtt{QSym}^{(q)}$ and we have $q^{\operatorname{inv}(\alpha)}\widetilde{F}_{\alpha}^{(q)}:=\sum_{\beta\preceq\alpha}q^{\operatorname{inv}(\beta)}M_{\beta}^{(q)}$. Our antipode formula here, which again describes the antipode for the full Hopf algebra $\mathtt{QSym}^{(q)}$, is as follows.
\begin{thmnn}
If $\alpha$ is a composition of $n$ then the antipode on $\mathtt{QSym}^{(q)}$ takes the fundamental $q$-quasi-symmetric function $F_{\alpha}^{(q)}$ to $(-1)^{n}q^{\operatorname{inv}(\alpha^{t})}\widetilde{F}_{\alpha^{t}}^{(q)}$.
\end{thmnn}
This is given in Theorem \ref{SQSymqF} below, and Proposition \ref{SQSymqM} below determines the image of the monomial $q$-quasi-symmetric function $M_{\alpha}^{(q)}$ as the expression $(-1)^{\ell(\alpha)}\sum_{\beta\succeq\alpha^{r}}q^{\operatorname{inv}(\beta)}M_{\beta}^{(q)}$. Note that unlike the antipode on $\mathtt{QSym}$, which is an involution (as is the case for any commutative, as well as co-commutative, Hopf algebra), the antipode on $\mathtt{QSym}^{(q)}$ is not an involution, but rather has infinite order (for generic $q$)---see Corollary \ref{deg2QSymq} below.

\medskip

There is a larger Hopf algebra, which is a real one (rather than a quantum one as in Remark \ref{qHopf}), and we now define (following \cite{[BZ]} and others). Some authors call it the algebra of non-commutative quasi-symmetric functions, denoted by $\mathtt{NCQSym}$, while others, like \cite{[Ml1]} and \cite{[Ml2]}, use the terminology of word quasi-symmetric functions, and use the notation $\mathtt{WQSym}$ (and Chapter 6 of \cite{[Hi1]} refers to it as the algebra of free quasi-symmetric functions, denoted $\mathtt{FQSym}$, and viewed as a Hopf sub-algebra of a larger Hopf algebra called the matrix quasi-symmetric functions $\mathtt{FQSym}$). This paper will use the former notation and terminology.

The algebra $\mathtt{NCQSym}$ was recognized in \cite{[NT]} to be a bidendriform algebra, which implies that it isomorphic to its graded dual. An explicit isomorphism was later described in \cite{[Ml2]}, after the prequel \cite{[Ml1]} investigated the primitive elements there. The main explicit reference for this algebra remains \cite{[BZ]}, which considers several bases for it and their properties, and also produces set of algebraically independent generators for it and for its dual. In fact, it does the same for the subalgebra $\mathtt{NCSym}$, whose primitive elements are determined in \cite{[LM]}, a reference which also establishes finds the antipode formula for that algebra.

Now, the bases for $\mathtt{NCQSym}$ are indexed by set compositions, namely the monomial basis for the part of degree $n$ is $\mathrm{M}_{\mathbf{A}}$, where $\mathbf{A}$ runs over the set compositions of the set $\mathbb{N}_{n}$ of integers between 1 and $n$ (for $\mathtt{NCSym}$ one works with set partitions). We define another set of linear generators for $\mathtt{NCQSym}$, which life the fundamental quasi-symmetric functions from \cite{[Ge]}. The elements of this set, for degree $n$, are denoted by $\mathrm{F}_{\alpha}^{\rho}$, where $\alpha$ is a composition of $n$ and $\rho$ is a permutation on $n$ elements. For a set composition $\mathbf{A}$ of $\mathbb{N}_{n}$, there are explicit $\alpha$ and $\rho$ such that the associated fundamental non-commutative quasi-symmetric functions $\mathrm{F}_{\mathbf{A}}$ is $\mathrm{F}_{\alpha}^{\rho}$ (see Definition \ref{fundNCQSym} below). This set of fundamental non-commutative quasi-symmetric functions constitute a new basis for the degree $n$ part of $\mathtt{NCQSym}$ (see Proposition \ref{basdim} below).

The elements $\mathrm{F}_{\alpha}^{\operatorname{Id}_{n}}$, for which $\rho$ is the trivial permutation, all belong to that basis. Our main result here, Theorem \ref{SNCQSymF} below, is a partial antipode formula for $\mathtt{NCQSym}$, that considers only this part of the fundamental basis there. Denoting the longest permutation on $n$ letters by $w^{0}_{n}$, it reads as follows.
\begin{thmnn}
For any given composition $\alpha$ of $n$, the antipode of $\mathtt{NCQSym}$ maps $\mathrm{F}_{\alpha}^{\operatorname{Id}_{n}}$ to $(-1)^{n}\mathrm{F}_{\alpha^{t}}^{w^{0}_{n}}$.
\end{thmnn}
Using this result, one can obtain, in the usual manner, the antipode formula for some of the monomial elements $\mathrm{M}_{\mathbf{A}}$, namely those in which the elements of the sets inside $\mathbf{A}$ increase---this is proved in Proposition \ref{SNCQSymM} below. In fact, for the two extreme compositions, we establish in Theorem \ref{SFalphaw0} below also the antipode images of $\mathrm{F}_{n}^{w^{0}_{n}}$ and $\mathrm{F}_{1^{n}}^{w^{0}_{n}}$ as $(-1)^{n}\mathrm{F}_{1^{n}}^{\operatorname{Id}_{n}}$ and $(-1)^{n}\mathrm{F}_{n}^{\operatorname{Id}_{n}}$ respectively.

Note that the square of the antipode of $\mathtt{NCQSym}$ is trivial in degree 2, where our theorems determine it entirely. However, we evaluate it completely in degree 3 in Proposition \ref{antipode3} and Corollary \ref{antiM3} below, and show in Corollary \ref{deg2QSymq} below that the matrix representing this square in degree 3 has some non-trivial Jordan blocks, so that it is again of infinite order (at least in characteristic 0).

\medskip

The idea of the proof is as follows. The proof that every connected graded bi-algebra is a Hopf algebra is based on the fact that the antipode can be constructed recursively (this known result is included as Lemma \ref{grconS} below). We rephrase this result as the statement that under an assumption on how the co-multiplication acts on a linearly independent set of elements, the images of this set of elements under the antipode can be determined via the vanishing of an appropriate combination (see Corollary \ref{detS} and Remark \ref{firstS} below).

To state it more precisely, assume that we are given a linearly independent set $\{g_{\alpha}\}_{\alpha}$ of homogeneous elements in a connected graded Hopf algebra, such that the co-multiplication takes each $g_{\alpha}$ to a sum of tensors $g_{\beta_{k}^{l}(\alpha)} \otimes g_{\beta_{n-k}^{r}(\alpha)}$ of elements of the same sort. Assume further that there are elements $\{\tilde{g}_{\alpha}\}_{\alpha}$ for which we have, away from the identity element, the equality $\sum_{k=0}^{n(\alpha)}\tilde{g}_{\beta_{k}^{l}(\alpha)}g_{\beta_{n-k}^{r}(\alpha)}=0$, or alternatively $\sum_{k=0}^{n(\alpha)}g_{\beta_{k}^{l}(\alpha)}\tilde{g}_{\beta_{n-k}^{r}(\alpha)}=0$. This determines the image of each $g_{\alpha}$ under the antipode as $\tilde{g}_{\alpha}$.

Now, the required co-multiplication rule holds for all the fundamental quasi-symmetric functions. This is well-known for $\mathtt{QSym}$ (and cited as Proposition \ref{coprodQSym} below), holds in a similar manner in $\mathtt{QSym}^{(q)}$ (see Proposition \ref{coprodQSymq} below), and we prove it also for $\mathtt{NCQSym}$ (in general) in Proposition \ref{coprodNCQSym} below. We divert from the classical formulation of this property, and phrase it using parameters that are more convenient for our calculations, using Definition \ref{cutatr} and Lemma \ref{concexp} below.

\medskip

The vanishing of the desired combinations is the technical heart of the proof. The issue is, there is a degree of freedom in writing the expression for the product of two fundamental quasi-symmetric functions (see, e.g., Corollary 5.12 of \cite{[AsS]}). While for the classical case of $\mathtt{QSym}$ one can obtain the vanishing also using the ``trivial'' way to do it (in the sense of, e.g., Remark \ref{trivtau} below), for the non-commutative Hopf algebras it is easier to take advantage of this degree of freedom. We establish the similar formulae also for $\mathtt{NCQSym}$ in Propositions \ref{prodF} and \ref{prodFtau} below.

Now, there are homomorphisms of algebras from $\mathtt{NCQSym}$ onto $\mathtt{QSym}^{(q)}$ and onto $\mathtt{QSym}$, with only the latter being a homomorphism of Hopf algebras. Moreover, these homomorphisms takes $\mathrm{F}_{\alpha}^{\operatorname{Id}_{n}}$ to $F_{\alpha}^{(q)}$ and $F_{\alpha}$ and $\mathrm{F}_{\alpha}^{w^{0}_{n}}$ to $q^{\operatorname{inv}(\alpha)}\widetilde{F}_{\alpha}^{(q)}$ and to $F_{\alpha}$, so that verifying the vanishing for $\mathtt{NCQSym}$ produces it also for the other two algebras. As the co-multiplication property is already established, one needs only to consider $\mathtt{NCQSym}$.

We will obtain this vanishing by partitioning the terms appearing in the total expression into pairs of identical terms with opposite signs. This can be seen, in the spirit of \cite{[BS]}, as the determination of an involution on the terms showing up, where two terms that are related by this involution are identical and have opposite signs. However, the details of our involution are very different from those carried out in that reference. Finally, the equalities arising from the compositions $1^{n}$ and $n$ turn out to produce the same argument for the antipodes of $\mathrm{F}_{n}^{w^{0}_{n}}$ and $\mathrm{F}_{1^{n}}^{w^{0}_{n}}$ respectively.

We conclude by evaluating the antipode on all the part of degree 3 of $\mathtt{NCQSym}$, exemplifying why our theorems cannot be extended to any other part of the fundamental basis for that algebra. It would be interesting to see whether some more general results can be obtained, perhaps using a modification of our basis.

\medskip

The paper is divided into 5 sections. Section \ref{Hopf} provides the basic definitions involving Hopf algebras, including the perspective that will be used later. Section \ref{QSym} introduces our notation for working with compositions of integers, as well as the known results about $\mathtt{QSym}$. Section \ref{qDefs} considers the two $q$-deformations, the commutative $\mathtt{QSym}_{q}$ and the non-commutative $\mathtt{QSym}^{(q)}$. Section \ref{NCQSym} defines the notions involving set compositions, and the algebra $\mathtt{NCQSym}$ with its relevant bases. Finally, Section \ref{Proofs} establishes the product formula for fundamental non-commutative quasi-symmetric functions, and proves the main results.

\medskip

I am grateful to D. Grinberg for suggesting to search for generalizations of the original proof.

\textbf{Data Availability Statement:} Data sharing not applicable to this article as no datasets were generated or analysed during the current study.

\section{Hopf Algebras \label{Hopf}}

In this section we recall the definitions required for considering Hopf algebras, and present some of the results in a way that will be most appropriate for our applications. In addition to the classical books \cite{[S]} and \cite{[Mo]} on the subject, or the more modern book \cite{[U]}, and the lecture notes \cite{[GR]}, we note that Section 2 of \cite{[P]} and the first few sections of \cite{[D]} contain this material in a manner that is convenient to follow (see also Section 3.1 of \cite{[LMvW]}).

\medskip

Let $R$ be a commutative ring, over which all the different types of algebras will be. We will express all the notions using maps of $R$-modules, and it will be useful for us to have notations for maps that are typically identifications.
\begin{notn}
For this we write $m_{R}$ for the natural isomorphism $R\otimes_{R}R \to R$ (since it represents the multiplication in $R$), and given an $R$-module $M$ we denote by $l_{M}$ the identification $R\otimes_{R}M \to M$ (left multiplication by scalars from $R$), and similarly $r_{M}:M\otimes_{R}R \to M$ (for the right multiplication). For any $R$-module $M$ we write $\tau_{M}$ for the interchanging map $M\otimes_{R}M \to M\otimes_{R}M$ which is determined by $\tau_{M}(x \otimes y)=y \otimes x$ for every $x$ and $y$ in $M$. \label{notiso}
\end{notn}
We recall that an algebra (over $R$) is defined, in the terminology using maps and arrows, as follows.
\begin{defn}
Let $A$ be an $R$-module.
\begin{enumerate}[$(i)$]
\item We say that $A$ is an \emph{$R$-algebra} if it is endowed with a multiplication map $m:A\otimes_{R}A \to A$, which is associative, and a unit map $\eta:R \to A$ of $R$-modules, such that multiplication with its image, from both sides, corresponds to the action of $R$. In terms of the maps, these properties correspond to the equalities \[m\circ(m\otimes\operatorname{Id}_{A})=m\circ(\operatorname{Id}_{A} \otimes m),\quad m\circ(\eta\otimes\operatorname{Id}_{A})=l_{A},\quad m\circ(\operatorname{Id}_{A}\otimes\eta)=r_{A}\] as maps $A\otimes_{R}A\otimes_{R}A \to A$, $R\otimes_{R}A \to A$, and $A\otimes_{R}R \to A$, using the isomorphisms from Notation \ref{notiso}.
\item The algebra is commutative if $m$ commutes with interchanging map $\tau_{A}$, namely we have the equality $m\circ\tau_{A}=m$ as maps $A\otimes_{R}A \to A$.
\item A homomorphism of algebras from $A$ into another $R$-algebra $B$, with multiplication $\mu:B\otimes_{R}B \to B$ and unit $u:R \to B$, is a map $f:A \to B$ of $R$-modules that respects multiplications and units, in the sense that $f \circ m=\mu\circ(f \otimes f)$ as maps $A\otimes_{R}A \to B$ and $f\circ\eta=u$ as maps $R \to B$.
\item For such $A$ and $B$, the module $A \otimes B$ is an algebra via the multiplication \[(m\otimes\mu)\circ(\operatorname{Id}_{A}\otimes\tau_{23}\otimes\operatorname{Id}_{B}):A\otimes_{R}B\otimes_{R}A\otimes_{R}B \to A\otimes_{R}B\] in which $\tau_{23}:B\otimes_{R}A \to A\otimes_{R}B$ is the natural map interchanging the intermediate $A$ and $B$ in the quadruple tensor product, and the unit defined by $(\eta \otimes u) \circ m_{R}^{-1}:R \to A \otimes B$.
\end{enumerate} \label{alg}
\end{defn}
Co-algebras (over $R$) are obtained by dualizing Definition \ref{alg}.
\begin{defn}
Take some $R$-module $C$.
\begin{enumerate}[$(i)$]
\item The module $C$ becomes an \emph{co-algebra} by defining a co-multiplication map $\Delta:C \to C\otimes_{R}C$, which has to be co-associative, and a co-unit map $\varepsilon:C \to R$, such that multiplication with its image, from both sides, corresponds to the action of $R$. The equalities here are \[(\Delta\otimes\operatorname{Id}_{C})\circ\Delta=(\operatorname{Id}_{C}\otimes\Delta)\circ\Delta,\quad l_{C}\circ(\varepsilon\otimes\operatorname{Id}_{C})\circ\Delta=r_{C}\circ(\operatorname{Id}_{C}\otimes\varepsilon)\circ\Delta=\operatorname{Id}_{C},\] the first as maps $C \to C\otimes_{R}C\otimes_{R}C$, and the other two going $C \to C$, where we again used the isomorphisms from Notation \ref{notiso}.
\item The algebra $C$ is co-commutative if $\Delta$ commutes with $\tau_{C}$ in the sense that the equality $\tau_{C}\circ\Delta=\Delta$ as maps $C \to C\otimes_{R}C$.
\item Let $D$ be another co-algebra, with co-multiplication $\Upsilon:D \to D\otimes_{R}D$ and co-unit $\epsilon:D \to R$. Then a map $f:C \to D$ is a homomorphism of co-algebras if respects the operations, namely the equality $(f \otimes f)\circ\Delta=\Upsilon \circ f$ of maps $C \to D\otimes_{R}D$ and $\varepsilon=\epsilon \circ f$ of maps $C \to R$ hold.
\item Let such $C$ and $D$ be given. Then with $\tau_{23}:C\otimes_{R}D \to D\otimes_{R}C$ inside the quadruple tensor product, the map \[(\operatorname{Id}_{C}\otimes\tau_{23}\otimes\operatorname{Id}_{D})\circ(\Delta\otimes\Upsilon):C\otimes_{R}D \to C\otimes_{R}D\otimes_{R}C\otimes_{R}D\] produces a co-multiplication on $C\otimes_{R}D$, which makes it into a co-algebra with the co-unit $m_{R}\circ(\varepsilon\circ\epsilon):C\otimes_{R}D \to R$.
\end{enumerate} \label{coalg}
\end{defn}
Using Definitions \ref{alg} and \ref{coalg}, one constructs the following ring.
\begin{prop}
Let $A$ be an algebra with multiplication $m$ and unit $\eta$, and let $C$ be a co-algebra with co-multiplication $\Delta$ and co-unit $\varepsilon$. Then the $R$-module $\operatorname{Hom}_{R}(C,A)$ becomes an associative ring under the operation taking morphisms $f$ and $g$ to the \emph{convolution} $f*g:=m\circ(f \otimes g)\circ\Delta$. The map $\eta\circ\varepsilon$ is the unit. \label{conv}
\end{prop}
It is clear that the if $A$ is commutative, or $C$ is co-commutative, then ring from Proposition \ref{conv} is commutative (but not in general).

\medskip

We view $R$ as an algebra with $m=m_{R}$ and $\eta=\operatorname{Id}_{R}$, as well as a co-algebra via $\Delta=m_{R}^{-1}$ and $\varepsilon=\operatorname{Id}_{R}$. Then the notion of a bi-algebra is obtained by merging the notions from Definitions \ref{alg} and \ref{coalg} in a compatible way.
\begin{defn}
Consider an $R$-module $B$, carrying the structure of a multiplication $m:B\otimes_{R}B \to B$ and unit $\eta:R \to B$ making it into an algebra, as well as a co-multiplication $\Delta:B \to B\otimes_{R}B$ and a co-unit $\varepsilon:B \to R$ through which it is a co-algebra.
\begin{enumerate}[$(i)$]
\item The module $B$ is a \emph{bi-algebra} if $\Delta$ and $\varepsilon$ are maps of algebras, or equivalently $m$ and $\eta$ are maps of co-algebras. The equivalence is due to both conditions being represented by the same four equalities, namely \[\varepsilon\circ\eta=\operatorname{Id}_{R}\quad,\varepsilon \circ m=m_{R}\circ(\varepsilon\otimes\varepsilon),\quad\Delta\circ\eta=(\eta\otimes\eta) \circ m_{R}^{-1},\] the latter two as maps $B\otimes_{R}B \to R$ and $R \to B\otimes_{R}B$, as well as \[\Delta \circ m=(m \otimes m)\circ(\operatorname{Id}_{B}\otimes\tau_{23}\otimes\operatorname{Id}_{B})\circ(\Delta\otimes\Delta)\] as maps $B\otimes_{R}B \to B\otimes_{R}B$, where the right hand side passes through the quadruple tensor product $B\otimes_{R}B\otimes_{R}B\otimes_{R}B$ (in which $\tau_{23}$ takes place).
\item We say that $B$ is commutative if it is so as an algebra, and co-commutative is its co-algebra structure is such.
\item An $R$-module homomorphism between two bi-algebras is a bi-algebra homomorphism if is it a homomorphism of both algebras and co-algebras.
\item The tensor product of bi-algebras is a bi-algebra with its structure as an algebra and co-algebra.
\end{enumerate} \label{bialg}
\end{defn}
Proposition \ref{conv} shows that when $B$ is a bi-algebra (and even when it is just an algebra and a co-algebra, as in Definition \ref{bialg}), the module $\operatorname{End}_{R}(B)$ is a unital ring with the convolution operation. In this case there is the special element $\operatorname{Id}_{B}$ in that ring, using which one makes the following definition.
\begin{defn}
A bi-algebra $B$ is called a \emph{Hopf algebra} if the element $\operatorname{Id}_{B}$ has a two-sided inverse $S$ in the unital ring $\operatorname{End}_{R}(B)$, namely if there is a map $S:B \to B$ such that $S*\operatorname{Id}_{B}=\operatorname{Id}_{B}*S=\eta\circ\varepsilon$. Equivalently, $S$ has to satisfy \[m\circ(S\otimes\operatorname{Id}_{B})\circ\Delta=\eta\circ\varepsilon=m\circ(\operatorname{Id}_{B} \otimes S)\circ\Delta.\] A map $S$ with this property is called the \emph{antipode} of the Hopf algebra $B$. \label{Hopfdef}
\end{defn}
Since two-sided inverses are unique in unital rings, it is clear that the antipode from Definition \ref{Hopfdef} is unique when it exists. Thus being a Hopf algebra is not an extra structure to a bi-algebra, but rather a special class of a bi-algebra.

The antipode of a Hopf algebra $H$, as a map $S:H \to H$, has the following properties.
\begin{prop}
The map $S$ is an anti-homomorphism of $H$ both as an algebra and as a bi-algebra. In case $H$ is either commutative or co-commutative, it is an anti-involution. \label{Sprop}
\end{prop}
Explicitly the anti-homomorphism property from Proposition \ref{Sprop} means that $S$ satisfies $S \circ m=m\circ\tau\circ(S \otimes S)$ and $\Delta \circ S=(S \otimes S)\circ\tau\circ\Delta$, in addition to the usual properties $S\circ\eta=\eta$ and $\varepsilon \circ S=\varepsilon$ (these equalities are as maps $H\otimes_{R}H \to H$, $H \to H\otimes_{R}H$, $R \to H$, and $H \to R$ respectively). Thus $S^{2}:H \to H$ is a homomorphism of bi-algebras, and the anti-involution property is the statement that $S^{2}=\operatorname{Id}_{H}$. For the proofs see, e.g., Propositions 13, 14, and 15 of \cite{[D]}.

\medskip

The Hopf algebras that this paper considers have the following property.
\begin{defn}
Let $M$ be a module over $R$.
\begin{enumerate}[$(i)$]
\item We say that $M$ is \emph{graded} if it is a direct sum $\bigoplus_{n=0}^{\infty}M_{n}$, where each $M_{n}$ is a finitely generated $R$-module.
\item An algebra $A$ is a \emph{graded algebra} if its module structure is graded as $\bigoplus_{n=0}^{\infty}A_{n}$ and the multiplication takes $A_{n} \otimes A_{m}$ into $A_{n+m}$ for every two natural numbers $n$ and $m$.
\item A co-algebra $C$ is called a \emph{graded co-algebra} when it is a graded module $\bigoplus_{n=0}^{\infty}C_{n}$ and the image of $C_{n}$ under the co-multiplication is contained in $\bigoplus_{k=0}^{n}(C_{k} \otimes C_{n-k})$ for any integer $n\geq0$.
\item A bi-algebra $B$ is a \emph{graded bi-algebra} in case its algebra structure is a graded algebra and its co-algebra structure makes it a graded co-algebra with the same grading.
\item We call a Hopf algebra $H$ a \emph{graded Hopf algebra} when its bi-algebra structure is graded.
\end{enumerate} \label{graded}
\end{defn}
The following remark will be relevant for the quantum generalization below.
\begin{rmk}
Let $x$ and $y$ be homogeneous elements of a graded bi-algebra $B$, say with $x \in B_{n}$ and $y \in B_{m}$. Definition \ref{graded} implies that we can write $\Delta(x)=\sum_{r=0}^{n}\Delta_{r}^{(n)}(x)$ with $\Delta_{r}^{(n)}(x) \in B_{r} \otimes B_{n-r}$ (some of which may possibly vanish), and similarly $\Delta(y)=\sum_{s=0}^{m}\Delta_{s}^{(m)}(y)$ where $\Delta_{s}^{(m)}(y) \in B_{s} \otimes B_{m-s}$. Then the compatibility condition between $m$ and $\Delta$ in Definition \ref{bialg} is the equality $\Delta(xy)=\sum_{r=0}^{n}\sum_{s=0}^{m}\Delta_{r}^{(n)}(x)\Delta_{s}^{(m)}(y)$ inside $B\otimes_{R}B$, which can also be decomposed into the equality $\Delta_{t}^{(m+n)}(xy)=\sum_{r+s=t}\Delta_{r}^{(n)}(x)\Delta_{s}^{(m)}(y)$ for every $t$ between 0 and the homogeneity degree $m+n$ of $xy$ (with $\Delta_{r}^{(n)}(x)$ vanishing when $r>n$, and similarly $\Delta_{s}^{(m)}(y)=0$ if $s>m$). \label{mDeltagr}
\end{rmk}

\medskip

These objects have the following properties.
\begin{lem}
If $A$ is a graded algebra then the unit $\eta$ has image in $A_{0}$. When $C$ is a graded co-algebra then $\varepsilon(C_{n})=0$ for every $n>0$. \label{propgr}
\end{lem}
For completeness, we give the proof over an arbitrary commutative ring $R$.
\begin{proof}
For the first assertion, write $\eta(1) \in A$ as $\sum_{n}a_{n}$ with $a_{n} \in A_{n}$. Since multiplying it by $x \in A_{m}$ produces $\sum_{n}a_{n}x$ with $a_{n}x \in A_{n+m}$, but the total result has to be $x$, we deduce that $a_{n}x=0$ for every $n>0$ and every $x \in A_{m}$. Thus $a_{n}A=0$ for all $n>0$, which means that $a_{n}=a_{n}\eta(1)=0$ for every such $n$, and $\eta(1)=a_{0} \in A_{0}$ as desired.

We now turn to the second assertion, and we set $I:=\varepsilon(C_{0}) \subseteq R$, which is a submodule hence an ideal. Take $y \in C_{0}$, express it as $l_{C}\circ(\varepsilon\otimes\operatorname{Id}_{C})\circ\Delta(y)$, with $\Delta(y) \in C_{0} \otimes C_{0}$ by assumption, and we get that $y \in IC_{0}$. Thus $C_{0}$ is a finitely generated $R$-module satisfying $IC_{0}=C_{0}$. It thus follows from Nakayama's Lemma that there is an element of $I$, which we can write as $\varepsilon(t)$ for some $t \in C_{0}$, such that $\varepsilon(t)y=y$ for all $y \in C_{0}$.

This implies that for every $y \in C_{0}$ we have $\varepsilon(y)=\varepsilon\big(\varepsilon(t)y\big)=\varepsilon(t)\varepsilon(y)$, so that $\varepsilon(t)$ generates the ideal $I$. It follows that if $C_{0}^{0}$ the kernel of $\varepsilon|_{C_{0}}$ then $C_{0}=Rt+C_{0}^{0}$. We claim that this sum is direct. Indeed, we have $\varepsilon(t)t=t$ (as $t \in C_{0}$), so that if $r\varepsilon(t)=0$ for $r \in R$ then $rt=r\varepsilon(t)t=0$, and indeed the intersection between $Rt$ and $C_{0}^{0}$ is trivial.

We now take $x \in C_{n}$ for some $n>0$, and consider $\Delta(x)$. The graded assumption and our decomposition of $C_{0}$ allows us to write
\begin{equation}
\Delta(x) \in x_{l} \otimes t+t \otimes x_{r}+C_{n} \otimes C_{0}^{0}+C_{0}^{0} \otimes C_{n}+\bigoplus_{k=1}^{n-1}(C_{k} \otimes C_{n-k}) \label{Deltadecom}
\end{equation}
for some $x_{l}$ and $x_{r}$ from $C_{n}$, in a unique manner. After applying $l_{C}\circ(\varepsilon\otimes\operatorname{Id}_{C})$, the left hand side of Equation \eqref{Deltadecom} becomes $x$, while the right hand side produces an element of $\varepsilon(t)x_{r}+\bigoplus_{k=0}^{n-1}C_{k}$, with the explicit contribution coming from the second term.

Let us multiply both sides of Equation \eqref{Deltadecom} by $\varepsilon(t) \in R$, and let it act on the left tensor multiplier in each term on the right hand side. As $\varepsilon(t)t=t$, the second term remains unchanged, so that after the action of $l_{C}\circ(\varepsilon\otimes\operatorname{Id}_{C})$ yet again, we get that $\varepsilon(t)x$ is also an element of $\varepsilon(t)x_{r}+\bigoplus_{k=0}^{n-1}C_{k}$. Since both $x$ and $\varepsilon(t)x$ are in the summand $C_{n}$, we get $\varepsilon(t)x=\varepsilon(t)x_{r}=x$ for all $x \in C_{n}$.

But then $x=\varepsilon(t)x_{r}=x_{r}$ (since $x_{r} \in C_{n}$ as well). Working in a similar manner with the operation $r_{C}\circ(\operatorname{Id}_{C}\otimes\varepsilon)$, we deduce that $x_{l}=x$ as well. Thus when we let $l_{C}\circ(\varepsilon\otimes\operatorname{Id}_{C})$ act on the right hand side of Equation \eqref{Deltadecom}, the first term yields $\varepsilon(x)t$, and the other terms land in $C_{0}^{0}$ or in $C_{k}$ for $k>0$. As the left hand side becomes $x$, with no component from $Rt \subseteq C_{0}$, this yields $\varepsilon(x)t=0$.

But from this it follows that $\varepsilon(x)\varepsilon(t)=0$. As this value is the same as $\varepsilon\big(\varepsilon(t)x\big)$, and we saw that $\varepsilon(t)x=x$, we deduce that $\varepsilon(x)=0$ as desired. As $n>0$ and $x \in C_{n}$ was arbitrary, this proves the lemma.
\end{proof}
Note that the element $\varepsilon(t) \in R$ need not necessarily be 1 (or invertible). For example, if $R$ is the product $\tilde{R}\times\hat{R}$ of two non-trivial rings, and $C$ is a co-algebra over $\tilde{R}$ thus over $R$ when we view $\tilde{R}$ as a quotient of $R$, then we can get, for example, $\varepsilon(t)$ to be the element $(1,0)$ of $R$, which is not a unit.

\medskip

Lemma \ref{propgr} implies that in $A_{0}$ is a sub-algebra $A$ in the graded case, and $C_{0}$ is, in a natural sense, a sub-co-algebra of $C$. Note that for a graded bi-algebra $B=\bigoplus_{n=0}^{\infty}B_{n}$, we get that the element $t:=\eta(1)$, which lies in $B_{0}$ as we saw, satisfies $\varepsilon(t)=1$ by definition, so that the ideal $I$ from the proof is the full ring $R$ and many calculations are simplified because the work done to cancel $\varepsilon(t)$ from some objects becomes trivial. It is not hard to show that if the graded bi-algebra is a Hopf algebra, then the antipode preserves degrees, but in our case of interest, this will be clear from the explicit formula for it below.

We thus make the following definition.
\begin{defn}
A graded bi-algebra $B$ is called \emph{connected} if $\eta:R \to B_{0}$ is an isomorphism, or equivalently $\varepsilon:B_{0} \to R$ is an isomorphism. \label{condef}
\end{defn}
It is clear from the proof of Lemma \ref{propgr} and the fact that $\varepsilon\circ\eta(1)=1$ that $B_{0}$ decomposes as $R\eta(1) \oplus B_{0}^{0}$, with $B_{0}^{0}:=\ker\varepsilon|_{B_{0}}$. As both conditions in Definition \ref{condef} are equivalent to the vanishing of $B_{0}^{0}$, they are indeed equivalent to one another. In case $R$ is a field (where the proof of Lemma \ref{propgr} is again simplified significantly because $\varepsilon(t)\neq0$ implies its invertibility), that definition is equivalent to the classical dimension (or rank) 1 condition.

We will henceforth identify, for a connected graded bi-algebra, the subspace $B_{0}$ with $R$ via $\eta$ and $\varepsilon$, so that in particular we will write just 1 for $\eta(1) \in B_{0}$. Thus $B=R\oplus\bigoplus_{n=1}^{\infty}B_{n}$, the map $\eta$ is the embedding of $R$ inside the direct sum, and $\varepsilon$ is the projection onto it, annihilating the other components, and these maps no longer need to be specified. Moreover, the proof of Lemma \ref{propgr} implies that in this case for every $x \in C_{n}$ with $n>0$ we have
\begin{equation}
\Delta(x)=x\otimes1+1 \otimes x+\overline{\Delta}(x)\qquad\mathrm{for}\quad\overline{\Delta}(x)\in\bigoplus_{k=1}^{n-1}(B_{k} \otimes B_{n-k}), \label{redDelta}
\end{equation}
where $\overline{\Delta}:C_{n}\to\bigoplus_{k=1}^{n-1}(B_{k} \otimes B_{n-k})$ is called the \emph{reduced co-multiplication}.

The axioms for $\varepsilon$ in this case are either clear or follow from Equation \eqref{redDelta}, and those for $\eta$ are again either clear or are expressed as the convention that the multiplication, from either side, by elements of $R=B_{0}$ corresponds to the $R$-module structure on $B$.

\medskip

Connected graded bi-algebras are always Hopf algebras, as follows.
\begin{lem}
Let $B=R\oplus\bigoplus_{n=1}^{\infty}B_{n}$ be a connected graded bi-algebra as above. Define a map $S:B \to B$ by induction as follows. $S|_{R}=\operatorname{Id}_{R}$, and assume that $S$ is defined on $R\oplus\bigoplus_{k=1}^{n-1}B_{k}$. Given $x \in B_{n}$, write $\overline{\Delta}(x)$ from Equation \eqref{redDelta} as $\sum_{i}y_{i} \otimes z_{i}$, with each $y_{i}$ in some $B_{k}$ for $1 \leq k<n$, with $z_{i} \in B_{n-k}$. Then the expressions $-x-\sum_{i}S(y_{i})z_{i}$ and $-x-\sum_{i}y_{i}S(z_{i})$ and coincide, and by defining $S(x)$ to be the joint value we obtain a map $S:B \to B$ that respects the grading and produces the antipode on $B$, making it a Hopf algebra.
\label{grconS}
\end{lem}
Lemma \ref{grconS} appears in many references, like \cite{[E]}, \cite{[LMvW]}, and others. For the sake of completeness, we provide the proof, as it is short.
\begin{proof}
Let $S_{l}:B \to B$ be defined using $x\mapsto-x-\sum_{i}S(y_{i})z_{i}$, and write $S_{r}:B \to B$ for the map given by $x\mapsto-x-\sum_{i}y_{i}S(z_{i})$. As we get, by induction, that $S_{l}(y_{i}) \in B_{k}$ and $S_{r}(z_{i}) \in B_{n-k}$ (since $1 \leq k<n$), the graded property of the multiplication shows that $S_{l}(x)$ and $S_{r}(x)$ are in $B_{n}$, hence both preserve the grading.

Now, the fact that $\Delta(1)=1\otimes1$ and $S_{l}(1)=S_{r}(1)=1$ shows that both maps satisfy the required relation to be the antipode on $R=B_{0}$. But the defining property of $S_{l}$ and the fact that $\varepsilon(x)=0$ for $x \in B_{n}$ when $n>0$ implies, via Equation \eqref{redDelta}, that $m\circ(S_{l}\otimes\operatorname{Id}_{B})\circ\Delta=0=\eta\circ\varepsilon$ also on $\bigoplus_{n=1}^{\infty}B_{n}$. Similarly we obtain the equality $m\circ(\operatorname{Id}_{B} \otimes S_{r})\circ\Delta=\eta\circ\varepsilon$.

But this means that $S_{l}*\operatorname{Id}_{B}=\operatorname{Id}_{B}*S_{r}=\eta\circ\varepsilon$ inside the ring structure that Proposition \ref{conv} endows on $\operatorname{End}_{R}(B)$. Hence $\operatorname{Id}_{B}$ has both a left inverse and a right inverse. Since in this case the inverses from both sides coincide, and produce the inverse of $\operatorname{Id}_{B}$ in that ring, the equality $S_{l}=S_{r}$ and the Hopf property follow, with our joint $S$ as the antipode. This proves the lemma.
\end{proof}
Using Lemma \ref{grconS}, one can obtain the general formula for the antipode of a connected graded Hopf algebra as in \cite{[T]}, but the number of calculations required in order to produce an explicit value using this formula grows with the degree. This formula also involves, typically, a lot of cancelations, and \cite{[BS]} presents a way for evaluating the antipode of a connected graded Hopf algebra in a cancelation-free manner in many cases.

We will use an alternative way for determining antipodes in this setting, using the following consequence of Lemma \ref{grconS}. Since we work over an arbitrary ring, we say that a collection $\{g_{\alpha}\}_{\alpha}$ of an $R$-module $M$ is called \emph{linearly independent} if the natural map from the formal free module $\bigoplus_{\alpha}Rg_{\alpha}$ into $M$ is injective (which reduces to the usual definition when $R$ is a field).
\begin{cor}
Let $\{g_{\alpha}\}_{\alpha}$ be a linearly independent collection of elements in a connected graded Hopf algebra $H$, in which each $g_{\alpha}$ lies in a graded component $H_{n(\alpha)}$. Assume that for every $\alpha$ and every $0 \leq k \leq n(\alpha)$, the part of $\Delta(g_{\alpha})$ that lies in $H_{k} \otimes H_{n-k}$ is a tensor product $g_{\beta_{k}^{l}(\alpha)} \otimes g_{\beta_{n-k}^{r}(\alpha)}$ of two elements from that set. Then if $\{\tilde{g}_{\alpha}\}_{\alpha}$ is another collection of such linearly independent elements, with $\tilde{g}_{\alpha} \in H_{n(\alpha)}$ as well, such that for every $\alpha$ the sum $\sum_{k=0}^{n(\alpha)}\tilde{g}_{\beta_{k}^{l}(\alpha)}g_{\beta_{n-k}^{r}(\alpha)}$ equals $\delta_{n(\alpha),0}$, then we have $S(g_{\alpha})=\tilde{g}_{\alpha}$ for every $\alpha$. \label{detS}
\end{cor}
The symbol $\delta_{n(\alpha),0}$ is, as usual, the Kronecker delta symbol, which equals 1 in case $n(\alpha)=0$ and 0 otherwise.
\begin{proof}
The form of $\Delta(g_{\alpha})$ as given in Equation \eqref{redDelta} implies that we may assume, without loss of generality, that one of the $g_{\alpha}$'s is $1 \in R=H_{0}$. Hence for every $\alpha$ we have $g_{\beta_{0}^{l}(\alpha)}=g_{\beta_{0}^{r}(\alpha)}=1$ and thus $g_{\beta_{n(\alpha)}^{l}(\alpha)}=g_{\beta_{n(\alpha)}^{r}(\alpha)}=g_{\alpha}$, so that $\beta_{n(\alpha)}^{l}(\alpha)=\beta_{n(\alpha)}^{r}(\alpha)=\alpha$ (by linear independence). The condition for $g_{\alpha}=1$ yields $\tilde{g}_{\alpha}=1=S(g_{\alpha})$, so that we may consider $\alpha$ with $n(\alpha)>0$.

We apply induction on $n(\alpha)$, namely take some $\alpha$ with $n(\alpha)=n>0$, and assume that the result holds for every $\beta$ with $n(\beta)<n$ (which we already verified for $n(\beta)=0$ by connectedness). Lemma \ref{grconS} implies, via the formula for $\Delta(g_{\alpha})$, that $S(g_{\alpha})=-g_{\alpha}-\sum_{k=1}^{n(\alpha)-1}S(g_{\beta_{k}^{l}(\alpha)})g_{\beta_{n-k}^{r}(\alpha)}$, and the induction hypothesis allows us to write $\tilde{g}_{\beta_{k}^{l}(\alpha)}$ as the left multiplier for every $1 \leq k<n(\alpha)$.

But we recall that $\sum_{k=0}^{n(\alpha)}\tilde{g}_{\beta_{k}^{l}(\alpha)}g_{\beta_{n-k}^{r}(\alpha)}=0$ under our assumptions. We may thus replace the sum over $k$ in the expression for $S(g_{\alpha})$, with its negative sign, by the remaining two terms with $k=0$ and $k=n(\alpha)$ (appearing with a positive sign). But since we saw that $\tilde{g}_{\beta_{0}^{l}(\alpha)}=1$ and $g_{\beta_{n-k}^{r}(\alpha)}=g_{\alpha}$, the summand with $k=0$ cancels with the extra term $-g_{\alpha}$. Recalling that $g_{\beta_{0}^{r}(\alpha)}=1$ as well, and we have $\beta_{n(\alpha)}^{r}(\alpha)=\alpha$, the remaining summand, with $k=n(\alpha)$, which thus produces the full value of $S(g_{\alpha})$, is $\tilde{g}_{\alpha}$ as desired. This proves the corollary.
\end{proof}
It is possible that the linear independence assumption in Corollary \ref{detS} can be relaxed or removed altogether, but the collections for which we will apply that corollary for determining antipodes below will satisfy this property.

\begin{rmk}
Replacing the last assumption in Corollary \ref{detS} by the assumption that $\sum_{k=0}^{n(\alpha)}g_{\beta_{k}^{l}(\alpha)}\tilde{g}_{\beta_{n-k}^{r}(\alpha)}=\delta_{n(\alpha),0}$ yields the same conclusion. Indeed, putting the antipode $S$ on the second multiplier in each of the expressions in its proof, which we may do via Lemma \ref{grconS}, yields the result for this case as well. \label{firstS}
\end{rmk}

\medskip

We conclude this section by discussing duality. Recall that for an $R$-module $M$, the \emph{dual module} is $M^{*}:=\operatorname{Hom}_{R}(M,R)$, and if $N$ is another $R$-module and $f:M \to N$ is a homomorphism of $R$-modules then there is the \emph{dual homomorphism} $f^{*}:N^{*} \to M^{*}$. We also have $R^{*} \cong R$, and hence the natural map $\iota_{M}:M \to M^{**}$, which is injective if $M$ is projective. The $R$-module $M$ is called \emph{reflexive} in case $\iota_{M}$ is an isomorphism (which is always the case when $M$ is finitely generated and projective).

We also note that there is an injective map from $M^{*}\otimes_{R}N^{*}$ into $(M\otimes_{R}N)^{*}$ for every two $R$-modules $M$ and $N$, which is an isomorphism if either $M$ or $N$ are finitely generated, but not otherwise.

The following well-known results are given in, e.g., Propositions 4, 5, 6, 7, and 8 of \cite{[D]}.
\begin{prop}
The following assertions hold:
\begin{enumerate}[$(i)$]
\item If $(C,\varepsilon,\Delta)$ is a co-algebra then $C^{*}$ becomes an algebra with the maps $\eta:=\varepsilon^{*}:R=R^{*} \to M^{*}$ and $m:=\Delta^{*}|_{C^{*}\otimes_{R}C^{*}}:C^{*}\otimes_{R}C^{*} \to C^{*}$. It is commutative if and only if $C$ is co-commutative.
\item Given an algebra $(A,\eta,m)$ over $R$, with $A$ finitely generated as an $R$-module, the $R$-module $A^{*}$ becomes a co-algebra, which is co-commutative precisely when $A$ is commutative, once we endow it with the two maps $\varepsilon:=\eta^{*}:A^{*} \to R^{*}=R$ and $\Delta:=m^{*}:A^{*}\to(A\otimes_{R}A)^{*}=A^{*}\otimes_{R}A^{*}$.
\item Let $B$ be a bi-algebra, with the $R$-module $B$ being finitely generated. Then the two constructions from part $(i)$ and $(ii)$ make $B^{*}$ into a bi-algebra, as well, with the corresponding commutativity and co-commutativity conditions.
\item If $H$ is a Hopf algebra and finitely generated as an $R$-module, then the bi-algebra $H^{*}$ from part $(iii)$ is also a Hopf algebra, with $S^{*}:H^{*} \to H^{*}$ being the antipode.
\end{enumerate} \label{dual}
\end{prop}
The reason for part $(ii)$ of Proposition \ref{dual} (hence parts $(iii)$ and $(iv)$ as well) not working for general algebras is that the image of $m^{*}:A^{*}\to(A\otimes_{R}A)^{*}$ is not always contained in $A^{*}\otimes_{R}A^{*}$ (see, e.g., Lemma 9.1.1 of \cite{[Mo]} for when this containment is satisfied). If $C$, $A$, $B$, and $H$ are reflexive there, with $C$ finitely generated in part $(i)$ there, then applying this proposition twice yields a structure on $C^{**}$, $A^{**}$, $B^{**}$, and $H^{**}$ respectively. It is clear that this structure coincides with the original structure via the respective isomorphism $\iota_{C}$, $\iota_{A}$, $\iota_{B}$, and $\iota_{H}$.

\medskip

Graded modules, as in Definition \ref{graded}, are typically not finitely generated, so the construction from Proposition \ref{dual} does not apply to them directly. However, we define the \emph{graded dual} $M^{*,\mathrm{gr}}$ of every such $R$-module $M=\bigoplus_{n=0}^{\infty}M_{n}$ to be $\bigoplus_{n=0}^{\infty}M_{n}^{*}\subseteq\prod_{n=0}^{\infty}M_{n}^{*}=M^{*}$. For graded modules $M$ and $N$, and an $R$-module homomorphism that is \emph{graded}, namely preserves the grading (namely takes $M_{n}$ to $N_{n}$ for every $n$), we have the graded map $f^{*}:N^{*,\mathrm{gr}} \to M^{*,\mathrm{gr}}$.

The tensor product $M \otimes N$ has the two gradings, from $M$ and from $N$, but the natural grading to put on it is to define the elements of degree $n$ to be those of $\bigoplus_{k=0}^{n}(M_{k} \otimes N_{n-k})$. We also consider $R$ to be graded and localized in degree 0. Then the maps from Definition \ref{graded} are all graded, with Lemma \ref{propgr} verifying it for the unit and the co-unit.

Analogously to Proposition \ref{dual}, we have the following result.
\begin{prop}
Consider the maps as defined as in Proposition \ref{dual}.
\begin{enumerate}[$(i)$]
\item When $C$ is a graded co-algebra, the graded dual $C^{*,\mathrm{gr}}$ is a graded algebra. Then $C$ is co-commutative if and only if $C^{*}$ is commutative.
\item If $A$ is a graded algebra, then $A^{*,\mathrm{gr}}$ is a graded co-algebra, and it is co-commutative if and only if $A$ is commutative.
\item Given a graded bi-algebra $B$, so is $B^{*,\mathrm{gr}}$, with the respective conditions for commutativity and co-commutativity.
\item In case $H$ is a graded Hopf algebra, the graded dual $H^{*,\mathrm{gr}}$ is also such.
\end{enumerate} \label{grdual}
\end{prop}
For every graded module $M$ we have the graded map $\iota_{M}^{\mathrm{gr}}:M\to(M^{*,\mathrm{gr}})^{*,\mathrm{gr}}$ which is the direct sum of $\iota_{M_{n}}:M_{n} \to M_{n}^{**}$ for every $n$. We call $M$ \emph{reflexive} (as a graded $R$-module) if so are all the $M_{n}$'s, which is equivalent to all the $M_{n}$'s being reflexive. Also in Proposition \ref{grdual}, if $C$, $A$, $B$, and $H$ carry the structure from the respective parts, then so do $(C^{*,\mathrm{gr}})^{*,\mathrm{gr}}$, $(A^{*,\mathrm{gr}})^{*,\mathrm{gr}}$, $(B^{*,\mathrm{gr}})^{*,\mathrm{gr}}$, and $(H^{*,\mathrm{gr}})^{*,\mathrm{gr}}$, and then the former is reflexive the induced structure on the latter coincides with the original one via $\iota_{C}^{\mathrm{gr}}$, $\iota_{A}^{\mathrm{gr}}$, $\iota_{B}^{\mathrm{gr}}$, and $\iota_{H}^{\mathrm{gr}}$ respectively.

\section{Quasi-Symmetric Functions \label{QSym}}

In this section we state the known results about the Hopf algebra $\mathtt{QSym}$, but formulated in an alternative way that will be more appropriate for establishing our results later.

\subsection{Compositions and Operations on Them}

Quasi-symmetric functions are parameterized by compositions of integers. We thus begin by setting up the relevant notations and relations between them.

\begin{defn}
Let $n$ be a non-negative integer.
\begin{enumerate}[$(i)$]
\item A \emph{composition} of $n$ is an ordered sequence $\alpha=(\alpha_{1},\ldots,\alpha_{\ell})$ of positive integers summing to $n$. We write this as $\alpha \vDash n$.
\item The number $\ell$ is called the \emph{length} of $\alpha$, and is denoted by $\ell(\alpha)$.
\item A \emph{partition} of $n$, written as $\alpha \vdash n$, is a composition whose entries are non-increasing.
\end{enumerate} \label{defcomp}
\end{defn}
We will sometimes write explicit compositions while omitting the parentheses and commas from Definition \ref{defcomp}, in case no confusion may arise.

For $n\geq1$, denote by $\mathbb{N}_{n-1}$ the set of integers between 1 and $n-1$. The following bijection between the sets $\{\alpha \vDash n\}$ and $\{T\subseteq\mathbb{N}_{n-1}\}$ is well-known.
\begin{prop}
Let $T\subseteq\mathbb{N}_{n-1}$ be a subset of size $\ell-1$, which we write in increasing order as $\{t_{i}\;|\;1 \leq i<\ell\}$. Then by adding $t_{0}:=0$ and $t_{\ell}:=n$, let $\operatorname{comp}_{n}T$ be the composition of $n$ whose length is $\ell$ and whose $j$th entry is $t_{j}-t_{j-1}$. Then $\operatorname{comp}_{n}$ is a bijection, whose inverse map $\operatorname{comp}_{n}^{-1}$ sends $\alpha \vDash n$ of length $\ell$ to the set $\big\{\sum_{j=1}^{i}\alpha_{j}\;|\;1 \leq i<\ell\big\}\subseteq\mathbb{N}_{n-1}$. \label{compn}
\end{prop}
We denote the sets from Proposition \ref{compn} by $T$ rather than $S$, since $S$ will be the antipode in this paper.

\medskip

There are two useful operations on pairs of compositions.
\begin{defn}
Take $\beta=(\beta_{1},\ldots,\beta_{\ell}) \vDash r$ and $\gamma=(\gamma_{1},\ldots,\gamma_{k}) \vDash s$.
\begin{enumerate}[$(i)$]
\item The \emph{concatenation} $\beta\gamma$ is $(\beta_{1},\ldots,\beta_{\ell},\gamma_{1},\ldots,\gamma_{k})$ whose is a composition of $r+s$, of length $\ell+k$.
\item If $r>0$ and $s>0$, or equivalently $\ell>0$ and $k>0$, the we define the \emph{near concatenation} $\beta\odot\gamma:=(\beta_{1},\ldots,\beta_{\ell-1},\beta_{\ell}+\gamma_{1},\gamma_{2},\ldots,\gamma_{k})$. It is also a composition of $r+s$, and its length is $\ell+k-1$.
\end{enumerate} \label{defconc}
\end{defn}
Note that $\beta$ or $\gamma$ can be empty for concatenations in Definition \ref{defconc}, and then the concatenation simply produces the other composition, with the same length. However, for near concatenations both compositions need to be non-empty, both for $\beta_{\ell}$ and $\gamma_{1}$ to be defined, and since the length $\ell+k-1$ cannot be smaller than those of the original compositions

We will need expressions, for a given composition $\alpha \vDash n$, as a concatenation $\beta\gamma$ or as a near concatenation $\beta\odot\gamma$ via Definition \ref{defconc}. It will be useful for us to make this explicit, which we do using the following notions.
\begin{defn}
Take a composition $\alpha \vDash n$, which is $\operatorname{comp}_{n}T$ for some subset $T\subseteq\mathbb{N}_{n-1}$, as well as an integer $1 \leq r<n$. Write $T_{|r}$ for $T\cap\mathbb{N}_{r-1}$, set $T_{r|}:=\{t-r\;|\;r<t \in T\}\subseteq\mathbb{N}_{n-r-1}$, and we define $\alpha_{|r}:=\operatorname{comp}_{r}T_{|r} \vDash r$ and $\alpha_{r|}:=\operatorname{comp}_{n-r}T_{r|} \vDash n-r$ to be the corresponding compositions. We extend the latter part to $r=0$ and to $r=n$ by setting $\alpha_{|0}=\alpha_{n|}=\emptyset$, as well as $\alpha_{n|}=\alpha_{0|}=\alpha$. \label{cutatr}
\end{defn}
The idea behind that notation, as well as the one from Definition \ref{setwithr} below, is related to the notation $u_{[i]}$ and $u^{[i]}$ from \cite{[LM]}.

The notions from Definition \ref{cutatr} are useful for the following result.
\begin{lem}
Let $n>0$, $T\subseteq\mathbb{N}_{n-1}$, and $\alpha:=\operatorname{comp}_{n}T \vDash n$ as above, set $\tilde{T}:=T\cup\{0,n\}$, and take an integer $0 \leq r \leq n$. Then there exists precisely one pair of compositions $\beta \vDash r$ and $\gamma \vDash n-r$ such that either $\beta\gamma=\alpha$ or $\beta\odot\gamma=\alpha$. This pair consists of $\beta=\alpha_{|r}$ and $\gamma=\alpha_{r|}$, and we have $\alpha=\beta\gamma$ when $r\in\tilde{T}$ and $\alpha=\beta\odot\gamma$ in case $r\not\in\tilde{T}$. \label{concexp}
\end{lem}

\begin{proof}
Assume that $\alpha=\beta\gamma$ or $\alpha=\beta\odot\gamma$ for some $\beta \vDash r$ and $\gamma \vDash n-r$. If $r=0$ or $r=n$ then either $\beta$ or $\gamma$ are empty, so that the other composition must be $\alpha$, and we are in the setting from the end of Definition \ref{cutatr}, where we have a concatenation and $r\in\tilde{T}$.

We thus assume that $0<r<n$, and write $\beta=\operatorname{comp}_{r}X$ with $X\subseteq\mathbb{N}_{r-1}$ and $\gamma=\operatorname{comp}_{n-r}Y$ for some $Y\subseteq\mathbb{N}_{n-r-1}$. The description of $\operatorname{comp}_{n}^{-1}$ implies that $T$ must contain precisely the integers from $X$, and then $r$ in case $\alpha=\beta\gamma$ (but not if $\alpha=\beta\odot\gamma$), followed by the integers $\{r+y\;|\;y \in Y\}$.

But this determines $X$ and $Y$ as $T_{|r}$ and $T_{r|}$ from Definition \ref{cutatr}, so that only $\beta=\alpha_{|r}$ and $\gamma=\alpha_{r|}$ need to be considered. This produces the uniqueness, and the fact that $r \in T$ when $\alpha=\beta\gamma$ and $r \not\in T$ in case $\alpha=\beta\odot\gamma$ establishes also the existence and the last assertion. This proves the lemma.
\end{proof}
Lemma \ref{concexp} will allow us to write the known formulae from the theory of quasi-symmetric functions in a different manner, which will be more adapted to our applications later. This lemma also clearly holds for $n=0$, where $r=0$, $\alpha=\emptyset$, $\tilde{T}=\{0\}$, and $\alpha_{|0}=\alpha_{0|}=\emptyset$.

\begin{ex}
Take $\alpha$ to be the composition $213\vDash6$, which is $\operatorname{comp}_{6}T$ for $T:=\{2,3\}$. Then the sets $T_{1|}$ and $T_{2|}$ from Definition \ref{cutatr} are empty, and we have $T_{3|}=\{2\}$, and $T_{4|}=T_{5|}=T$ (but viewed as subsets of $\mathbb{N}_{3}$ and $\mathbb{N}_{4}$ respectively). It follows that $\alpha_{1|}=1$, $\alpha_{2|}=2$, $\alpha_{3|}=21$, $\alpha_{4|}=211$, and $\alpha_{5|}=212$. The sets $T_{|5}$, $T_{|4}$, and $T_{|3}$ are empty, while $T_{|2}=\{1\}\subseteq\mathbb{N}_{3}$ and $T_{|1}=\{1,2\}\subseteq\mathbb{N}_{4}$, and the corresponding compositions are $\alpha_{|1}=113$, $\alpha_{|2}=13$, $\alpha_{|3}=3$, $\alpha_{|4}=2$, and $\alpha_{|5}=1$. This expresses our $\alpha$ as the concatenations $\alpha_{2|}\alpha_{|2}$ and $\alpha_{3|}\alpha_{|3}$ and the near concatenations $\alpha_{1|}\odot\alpha_{|1}$, $\alpha_{4|}\odot\alpha_{|4}$, and $\alpha_{5|}\odot\alpha_{|5}$, in correspondence, via Lemma \ref{concexp}, with the elements of the set $T$. \label{excomp}
\end{ex}
Note that as 1 has a single composition as well, which corresponds to the emptiness of $\mathbb{N}_{0}$, we have $\alpha_{|1}=\alpha_{n-1|}=(1)$ for every $\alpha \vDash n\geq1$, as we saw in particular in Example \ref{excomp}.
\begin{rmk}
In most of the constructions in this paper, the object more related to a composition $\alpha \vDash n$ is not the subset $T:=\operatorname{comp}_{n}^{-1}\alpha\subseteq\mathbb{N}_{n-1}$, but rather $\tilde{T}:=T\cup\{0,n\}$ from Lemma \ref{concexp}. In fact, the correspondece $\alpha\leftrightarrow\tilde{T}$ produces a bijection between the set of finite subsets of $\mathbb{N}\cup\{0\}$ which contain 0 (or equivalently the finite subsets of $\mathbb{N}$, without 0) and the set of all compositions of all integers. \label{gencomp}
\end{rmk}
Recall that $\operatorname{comp}_{n}$ is defined in Proposition \ref{compn} only when $n>0$. The bijection from Remark \ref{gencomp} does extend to the empty composition $\emptyset\vDash0$, and relates it to $\{0\}$ (or to $\emptyset\subseteq\mathbb{N}$). However, we will stick with $\operatorname{comp}_{n}$ and its inverse as they are the ones appearing in the literature.

\medskip

Here are the three involutions on compositions and subsets of $\mathbb{N}_{n-1}$.
\begin{defn}
Take some natural number $n\geq1$.
\begin{enumerate}[$(i)$]
\item The \emph{reversal} of $T\subseteq\mathbb{N}_{n-1}$ is $T^{r}:=\{n-t\;|\;t \in T\}$. This is also an involution on subsets of $\mathbb{N}_{n-1}$.
\item The \emph{complement} of $T\subseteq\mathbb{N}_{n-1}$ is $T^{c}:=\mathbb{N}_{n-1} \setminus T$. It defines an involution on subsets of $\mathbb{N}_{n-1}$.
\item These two involutions commute, and their composition (in either order) produces a third involution on subsets of $\mathbb{N}_{n-1}$, which we call \emph{transposition}, and write as $T \mapsto T^{t}:=\{n-t\;|\;t \in T^{c}\}=\{n-t\;|\;t \in T\}^{c}$.
\item If $\alpha:=\operatorname{comp}_{n}T \vDash n$ then set $\alpha^{r}:=\operatorname{comp}_{n}T^{r} \vDash n$, $\alpha^{c}:=\operatorname{comp}_{n}T^{c} \vDash n$, and $\alpha^{t}:=\operatorname{comp}_{n}T^{t} \vDash n$ and we call them the \emph{reverse}, \emph{complement}, and {transpose} composition of $\alpha$ respectively.
\item For $\alpha=\emptyset\vDash0$, we set $\alpha^{r}:=\alpha^{c}:=\alpha^{t}:=\emptyset$ as well.
\end{enumerate} \label{invdef}
\end{defn}
It is clear, via the relation between the sizes of sets and lengths of compositions in Proposition \ref{compn}, that if $\ell(\alpha)=\ell$ for $\alpha \vDash n\geq1$ then $\ell(\alpha^{r})=\ell$ and $\ell(\alpha^{c})=\ell(\alpha^{t})=n+1-\ell$. Note that this holds for reversion also when $\alpha=\emptyset\vDash0$, but not for the other two involutions (since we cannot work with the sets anymore, and indeed $n+1-\ell=1>0=n$ when $n=\ell=0$).

The reversal of sets as in Definition \ref{invdef} was used extensively in \cite{[Z1]}, \cite{[Z2]}, and \cite{[Z3]}, through the relation with ascent and descent sets for standard Young tableaux and Sch\"{u}tzenbeger's evacuation involution. The reversal of $\alpha$ simply means writing the entries of $\alpha$ in the reverse order.
\begin{ex}
If $\alpha$ and $T$ are as in Example \ref{excomp}, then we have $T^{r}:=\{3,4\}$, $T^{c}:=\{1,4,5\}$, and $T^{t}:=\{1,2,5\}$, so that $\alpha^{r}=312$ (indeed of the same length 3 as $\alpha$) and $\alpha^{c}=1311$ and $\alpha^{t}=1131$ have length $6+1-3=4$. \label{invex}
\end{ex}

Here are the relations between the notions from Definitions \ref{defconc}, \ref{cutatr}, and \ref{invdef}.
\begin{lem}
Assume that $\alpha=\operatorname{comp}_{n}T \vDash n\geq1$ for $T\subseteq\mathbb{N}_{n-1}$, $1 \leq r<n$, $\beta \vDash r>0$ and $\gamma \vDash s>0$.
\begin{enumerate}[$(i)$]
\item We have the equalities $(T^{c})_{|r}=(T_{|r})^{c}$, $(T^{c})_{r|}=(T_{r|})^{c}$, $(T^{r})_{|r}=(T_{n-r|})^{r}$, $(T^{r})_{r|}=(T_{|n-r})^{r}$, $(T^{t})_{|r}=(T_{n-r|})^{t}$, and $(T^{t})_{r|}=(T_{|n-r})^{t}$ as sets.
\item We also have the equalities of compositions given by $(\alpha^{c})_{|r}=(\alpha_{|r})^{c}$, $(\alpha^{c})_{r|}=(\alpha_{r|})^{c}$, $(\alpha^{r})_{|r}=(\alpha_{n-r|})^{r}$, $(\alpha^{r})_{r|}=(\alpha_{|n-r})^{r}$, $(\alpha^{t})_{|r}=(\alpha_{n-r|})^{t}$, and $(\alpha^{t})_{r|}=(\alpha_{|n-r})^{t}$. These are valid also in case $\alpha=\emptyset\vDash0$.
\item The six equalities $(\beta\gamma)^{c}=\beta^{c}\odot\gamma^{c}$, $(\beta\odot\gamma)^{c}=\beta^{c}\gamma^{c}$, $(\beta\gamma)^{r}=\gamma^{r}\beta^{r}$, $(\beta\odot\gamma)^{r}=\gamma^{r}\odot\beta^{r}$, $(\beta\gamma)^{t}=\gamma^{t}\odot\beta^{t}$, and $(\beta\odot\gamma)^{t}=\gamma^{t}\beta^{t}$ hold as well. If either $\beta$ or $\gamma$ are $\emptyset$ then $(\beta\gamma)^{c}=\beta^{c}\gamma^{c}$, $(\beta\gamma)^{r}=\gamma^{r}\beta^{r}$, and $(\beta\gamma)^{t}=\gamma^{t}\beta^{t}$.
\end{enumerate} \label{relsinv}
\end{lem}

\begin{proof}
It is clear that taking complements commutes with the intersection with $\mathbb{N}_{r-1}$ and the subtraction in Definition \ref{cutatr}, yielding the first two equalities in part $(i)$. Reversing sends the largest $r-1$ elements of $\mathbb{N}_{n-1}$ onto $\mathbb{N}_{r-1}$ in the opposite order as well as $\mathbb{N}_{n-r-1}$ onto the $n-r-1$ largest elements of $\mathbb{N}_{n-1}$ in the opposite order, producing the following two equalities there. The remaining equalities in part $(i)$ follow by combining both facts.

Part $(ii)$ follows immediately from part $(i)$ via Definition \ref{cutatr}, and its extension to the empty composition case is trivial. We now write $n$ for $r+s$ and $\alpha$ for either $\beta\gamma$ or $\beta\odot\gamma$, so that $\beta=\alpha_{|r}$ and $\gamma=\alpha_{r|}$. Then part $(iii)$ follows from part $(ii)$ via Lemma \ref{concexp} (including the dichotomy between the cases $r \in T$ and $r \not\in T$), with the empty composition case being immediate as well, and takes this form because $\emptyset$ cannot be nearly concatenated. This proves the lemma.
\end{proof}
Parts $(ii)$ and $(iii)$ of Lemma \ref{relsinv} can also be proven directly.

\medskip

The set of compositions of a fixed integer $n$ carries an order, which is very useful for our purposes.
\begin{defn}
Let $\alpha=(\alpha_{1},\ldots,\alpha_{\ell})$ and $\beta$ be compositions of the same number $n$. We say that $\beta$ is a \emph{refinement} of $\alpha$, denoted by $\beta\preceq\alpha$ (or equivalently $\alpha\succeq\beta$), if it is the concatenations of compositions of the $\alpha_{j}$'s, in the right order. We also write $\beta\prec\alpha$ (or $\alpha\succ\beta$) to indicate that $\beta\preceq\alpha$ and $\beta\neq\alpha$. \label{refine}
\end{defn}
An example for the refinement from Definition \ref{refine} is given in Definition \ref{defconc}, where we have $\beta\gamma\prec\beta\odot\gamma$ for every $\beta \vDash r\geq1$ and $\gamma \vDash s\geq1$. In fact, this is a \emph{covering} in that order (namely if $\beta\gamma\preceq\delta\preceq\beta\odot\gamma$ then $\delta$ equals either $\beta\gamma$ or $\beta\odot\gamma$), and every covering is obtained in this manner.

When $n=0$ there is only the empty composition (as is the case when $n=1$). When $n\geq1$, the refinement $\beta\preceq\alpha$ from Definition \ref{refine} is equivalent to the assertion that if $\alpha=\operatorname{comp}_{n}T$ and $\beta=\operatorname{comp}_{n}X$ then $T \subseteq X$. It thus follows, in relation to the involutions from Definition \ref{invdef}, that if $\beta\preceq\alpha$ then $\beta^{r}\preceq\alpha^{r}$, $\beta^{c}\succeq\alpha^{c}$, and $\beta^{t}\succeq\alpha^{t}$, as well as $\beta_{|r}\preceq\alpha_{|r}$ and $\beta_{r|}\preceq\alpha_{r|}$ for every $0 \leq r \leq n$.

\subsection{Quasi-Symmetric Functions}

We write $R\ldbrack\mathbf{x}_{\infty}\rdbrack$ for the ring of power series in the infinitely many variables $\{x_{i}\}_{i=1}^{\infty}$ over our ring $R$. If $I:=\{i_{j}\}_{j=1}^{\ell}$ is a set of $\ell$ integers (which we express as $|I|=\ell$) written in increasing order, and $\alpha$ is a composition of length $\ell$, then we write $x_{I}^{\alpha}$ for the monomial $\prod_{j=1}^{\ell}x_{i_{j}}^{\alpha_{j}}$. Thus every element of $R\ldbrack\mathbf{x}_{\infty}\rdbrack$ can be expressed as $\sum_{I,\alpha}c_{I,\alpha}x_{I}^{\alpha}$, where $I$ runs over all finite subsets of integers and $\alpha$ goes over all the compositions of length $|I|$, and $c_{I,\alpha} \in R$.

Let $S_{\mathbb{N}}$ be the group of permutations of the set $\mathbb{N}$ of positive integers. For every $\ell$ it operates on sets of integers of size $\ell$, which we write as $\rho:I \mapsto \rho I$, and we can write each such set in increasing order. We recall from \cite{[Hi2]} that this produces an action of $\rho \in S_{\mathbb{N}}$ on the additive group of $R\ldbrack\mathbf{x}_{\infty}\rdbrack$ (not respecting multiplication) by taking each power series $\sum_{I,\alpha}c_{I,\alpha}x_{I}^{\alpha}$ to $\sum_{I,\alpha}c_{I,\alpha}x_{\rho I}^{\alpha}$.
\begin{defn}
An element $F \in R\ldbrack\mathbf{x}_{\infty}\rdbrack$ is called a \emph{quasi-symmetric function} (over $R$) if it has finite degree and it is invariant under the action from \cite{[Hi2]}. The latter condition is equivalent to saying that if $I$ and $J$ are of the same size $\ell$ and $\alpha$ has length $\ell$ then $x_{I}^{\alpha}$ and $x_{I}^{\alpha}$ appear in $F$ with the same coefficient. We denote the set of quasi-symmetric function by $\mathtt{QSym}$. \label{QSymdef}
\end{defn}
We are omitting the ring $R$ (which can be assumed to be $\mathbb{Q}$ throughout) from the notation. The finite degree condition in Definition \ref{QSymdef} implies that $\mathtt{QSym}$ is contained in the direct sum $\bigoplus_{d=0}^{\infty}R\ldbrack\mathbf{x}_{\infty}\rdbrack_{d}$, where the summand associated with $d$ is the homogeneous part of degree $d$ inside $R\ldbrack\mathbf{x}_{\infty}\rdbrack$.

For an integer $a$ and a linear combination of expressions parameterized by compositions, say $f=\sum_{\gamma}c_{\gamma}f_{\gamma}$ with $c_{\gamma} \in R$, we write $l_{a}f$ for $\sum_{\gamma}c_{\gamma}f_{a,\gamma}$, where $a,\gamma$ is the concatenation of the length 1 composition $a$ with $\gamma$. Definition \ref{QSymdef} and an easy verification yield the following result.
\begin{lem}
For every $\alpha \vDash d$ of length $\ell$, let $M_{\alpha}:=\sum_{|I|=\ell}x_{I}^{\alpha}$ be the \emph{monomial quasi-symmetric function} associated with the composition $\alpha$, including $M_{\emptyset}=1$. Then $\mathtt{QSym}$ is spanned, in a linearly independent manner, by $\{M_{\alpha}\}_{\alpha}$ where $\alpha$ runs over all compositions. Moreover, $\mathtt{QSym}$ is closed under the product from $R\ldbrack\mathbf{x}_{\infty}\rdbrack$, via the \emph{quasi-shuffle rule} $M_{\alpha}M_{\beta}=l_{\alpha_{1}}M_{\bar{\alpha}}M_{\beta}+l_{\beta_{1}}M_{\alpha}M_{\bar{\beta}}+l_{\alpha_{1}+\beta_{1}}M_{\bar{\alpha}}M_{\bar{\beta}}$ for any pair $\alpha=(\alpha_{1},\bar{\alpha})$ and $\beta=(\beta_{1},\bar{\beta})$ of non-trivial compositions. \label{monQSym}
\end{lem}
As the number of compositions of $d$ is finite (it is 1 for $d=0$ and $2^{d-1}$ wherever $d>0$, via Proposition \ref{compn}), Lemma \ref{monQSym} implies that $\mathtt{QSym}$ is a graded commutative algebra in the sense of Definition \ref{graded} (this is in correspondence with the action from \cite{[Hi2]} preserving degrees). The unit is the isomorphism from $R$ onto $R\ldbrack\mathbf{x}_{\infty}\rdbrack_{0}$, which also the part $RM_{\emptyset}$ of degree 0 in $\mathtt{QSym}$.

\medskip

The explicit formula for the product from Lemma \ref{monQSym} is given in Subsection 3.3.1 of \cite{[LMvW]}, among others. We replace the path terminology from that reference by that of sequences, but the equivalence of the two notions is immediate.
\begin{defn}
Let $\ell$ and $k$ be two non-negative integers.
\begin{enumerate}[$(i)$]
\item For every $0 \leq h\leq\min\{\ell,k\}$, let $P_{\ell,k}^{h}$ be the set of all sequences $\omega$ consisting of the letters l, r, and b, in which b shows up $h$ times, there are $\ell-h$ instances of l, and r appears $k-h$ times. There are $\binom{\ell+k-h}{h,\ell-h,k-h}$ such sequences.
\item We set $P_{\ell,k}:=\bigcup_{h=0}^{\min\{\ell,k\}}P_{\ell,k}^{h}$. Its size is thus $\sum_{h=0}^{\min\{\ell,k\}}\binom{\ell+k-h}{h,\ell-h,k-h}$.
\item Take two compositions $\alpha \vDash n$ and $\beta \vDash m$ such that $\ell(\alpha)=\ell$ and $\ell(\beta)=k$, and take $\omega \in P_{\ell,k}^{h}$. Then we define a composition $\gamma_{\omega}(\alpha,\beta) \vDash n+m$, of length $\ell+k-h$, as follows. Put the entries of $\alpha$ in the locations of the l's and the b's according to the order they show up in $\omega$. Do the same with those of $\beta$ in those of the r's and the b's, and for b in $\omega$, add the entries from $\alpha$ and from $\beta$ that show up there.
\end{enumerate} \label{pathsmult}
\end{defn}
In fact, $P_{\ell,k}$ from Definition \ref{pathsmult} can be described directly as the set of sequences satisfying the restriction that number of l's and b's together is $\ell$, and the number of r's and b's together is $k$. But as the length of $\gamma_{\omega}(\alpha,\beta)$ depends on $h$, we stated the definition as it is. The letters l, r, and b there stand for the left composition, the right one, and both.

The result is now as follows.
\begin{cor}
For two compositions $\alpha$ and $\beta$, of lengths $\ell$ and $k$, the product $M_{\alpha}M_{\beta}$ is given by $\sum_{\omega \in P_{\ell,k}}M_{\gamma_{\omega}(\alpha,\beta)}$. \label{prodMs}
\end{cor}
It is clear that if $\ell$ of $k$ vanishes in Definition \ref{pathsmult} then $|P_{\ell,k}|=1$, and the unique sequence $\omega$ there produces the initial composition of the other length. This corresponds, via Corollary \ref{prodMs}, to the fact that $M_{\emptyset}=1$. The simplest non-trivial examples are as follows.
\begin{ex}
The set $P_{1,1}$ consists of the sequences lr and rl (with $h=0$) and b (for $h=1$). Using these, Corollary \ref{prodMs} yields $M_{a}M_{b}=M_{a,b}+M_{b,a}+M_{a+b}$. In $P_{2,1}^{0}$ we have llr, lrl, and rll, while $P_{2,1}^{1}$ consists of bl and lb. Thus $M_{a,b}M_{c}$ equals $M_{a,b,c}+M_{a,b+c}+M_{a,c,b}+M_{a+c,b}+M_{c,a,b}$. Similarly, $P_{1,2}^{0}$ contains lrr, rlr, and rrl, and in $P_{1,2}^{1}$ we find br and rb, so that the product $M_{c}M_{a,b}$ indeed equals $M_{c,a,b}+M_{c+a,b}+M_{a,c,b}+M_{a,c+b}+M_{a,b,c}$. \label{pathex}
\end{ex}
It is not hard to describe the bijection $\omega\mapsto\tilde{\omega}$ between $P_{\ell,k}^{h}$ and $P_{k,\ell}^{h}$ for every $\ell$, $k$, and $h$, such that $\gamma_{\omega}(\alpha,\beta)=\gamma_{\tilde{\omega}}(\beta,\alpha)$ for every pair of compositions $\alpha$ of length $\ell$ and $\beta$ of length $k$. This is in correspondence with the commutativity of $\mathtt{QSym}$, and is visible in Example \ref{pathex}. However, later we will have non-commutative algebras whose multiplications uses similar descriptions.
\begin{ex}
For $\ell+k=4$, the set $P_{3,1}^{0}$ contains lllr, llrl, lrll, and rlll, and in $P_{3,1}^{1}$ we find llb, lbl, and bll. Similarly, the elements of $P_{1,3}^{0}$ are rrrl, rrlr, rlrr, and lrrr, while those of $P_{1,3}^{1}$ are rrb, rbr, and brr (written in the order to the involution $\omega\mapsto\tilde{\omega}$. In the remaining case where $\ell+k=4$, the elements of $P_{2,2}^{0}$ are llrr, rrll, lrlr, rlrl, lrrl, and rllr (the involution interchanges each pair in the order written), the set $P_{2,2}^{1}$ consists of lrb, rlb, lbr, rbl, blr, and brl (again ordered well for the involution), and $P_{2,2}^{2}$ contains the single element bb. \label{expath}
\end{ex}

\medskip

Another basis for $\mathtt{QSym}$, which consists of the original quasi-symmetric functions defined in \cite{[Ge]}, is the following one, for which we recall the refinement relation from Definition \ref{refine}.
\begin{defn}
Let $\alpha \vDash n$ be a set composition.
\begin{enumerate}[$(i)$]
\item The \emph{fundamental quasi-symmetric function}, which we write as either $F_{\alpha}$ or as $F_{n,T}$ (Remark \ref{gencomp} also suggests $F_{\tilde{T}}$), is defined to be $\sum_{\beta\preceq\alpha}M_{\beta}$.
\item We also set $E_{\alpha}:=\sum_{\beta\succeq\alpha}M_{\beta}$.
\end{enumerate} \label{fundQSym}
\end{defn}
It is clear that when $n=0$, Definition \ref{fundQSym} produces $F_{\emptyset}=E_{\emptyset}=M_{\emptyset}=1$.

We shall be using the original formulation of $F_{\alpha}=F_{n,T}$ from \cite{[Ge]}, which we now express in the form that will be convenient for our applications.
\begin{defn}
A \emph{multiset} of size $n$ is a non-decreasing sequence of $n$ integers. For any such multiset $J$ we write $x^{J}$ for the product $\prod_{k=1}^{n}x_{j_{k}}$. Given a subset $T\subseteq\mathbb{N}_{n-1}$ we write $\mathcal{M}_{n,T}$ for the set of multisets $J=(j_{1},\ldots,j_{n})$ of size $n$ which satisfy $j_{k}<j_{k+1}$ wherever $k \in T$. \label{multisets}
\end{defn}
It is clear from Definition \ref{multisets} that the set of all multisets of size $n$ is $\mathcal{M}_{n,\emptyset}$, while the set of sets of size $n$, when written each in increasing order, is $\mathcal{M}_{n,\mathbb{N}_{n-1}}$.

The formula from \cite{[Ge]} is now given, in terms of Definition \ref{multisets}, as follows.
\begin{lem}
For every $n>0$ and $T\subseteq\mathbb{N}_{n-1}$ we have $F_{n,T}=\sum_{J\in\mathcal{M}_{n,T}}x^{J}$. \label{sumFnT}
\end{lem}
The equivalence of the expressions from Definition \ref{fundQSym} and Lemma \ref{sumFnT} is straightforward and easy.

\medskip

The co-unit on $\mathtt{QSym}$ is the direct sum of the isomorphism from the degree 0 part $RM_{\emptyset}$ onto $R$ and the map taking all the homogeneous quasi-symmetric functions of positive degree to 0 (as expected via Lemma \ref{propgr}). For the co-product, which is not co-commutative, we use the following definition.
\begin{defn}
Given $f\in\mathtt{QSym}$, consider it as a function in two ordered sets of variables $\{x_{i}\}_{i=1}^{\infty}$ and $\{y_{i}\}_{i=1}^{\infty}$, in which we view $y_{1}$ as having a larger index than any $x_{i}$. Then the expression $f(\mathbf{x}_{\infty},\mathbf{y}_{\infty})$ is quasi-symmetric in $\{x_{i}\}_{i=1}^{\infty}$ and $\{y_{i}\}_{i=1}^{\infty}$ separately, so that it can be written as $\sum_{j=1}^{m}g_{j}(\mathbf{x}_{\infty})h_{j}(\mathbf{y}_{\infty})$. Then we set $\Delta(f):=\sum_{j=1}^{m}g_{j} \otimes h_{j}$. \label{DeltaQSym}
\end{defn}
For the two bases $\{M_{\alpha}\}_{\alpha}$ and $\{F_{\alpha}\}_{\alpha}$ of $\mathtt{QSym}$, the co-multiplication from Definition \ref{DeltaQSym} takes, in the terminology from Definition \ref{cutatr}, the following form.
\begin{prop}
Let $\alpha \vDash n>0$ be the composition that is associated with $T\subseteq\mathbb{N}_{n-1}$ via Proposition \ref{compn}, and write $\tilde{T}$ for $T\cup\{0,n\}$. Then we have the equalities $\Delta(M_{\alpha})=\sum_{r\in\tilde{T}}M_{\alpha_{|r}} \otimes M_{\alpha_{r|}}$ and $\Delta(F_{\alpha})=\sum_{r=0}^{n}F_{\alpha_{|r}} \otimes F_{\alpha_{r|}}$. \label{coprodQSym}
\end{prop}
The well-known formulae for $\Delta(M_{\alpha})$ and $\Delta(F_{\alpha})$ in the literature involves all the pairs of compositions $\beta$ and $\gamma$ with $\alpha=\beta\gamma$, or with $\alpha=\beta\odot\gamma$ in the second sum, and Lemma \ref{concexp} transfers these expressions to those from Proposition \ref{coprodQSym}. The values $r=0$ and $r=n$, which are always in $\tilde{T}$, correspond the fact that $M_{\emptyset}=F_{\emptyset}=1$ to the parts in Equation \eqref{redDelta} that distinguish the value under $\Delta$ from those under the reduced co-multiplication $\overline{\Delta}$. Note that this proposition also holds for $n=0$ and $\alpha=\emptyset$, by the known value that $\Delta$ assigns to $M_{\emptyset}=F_{\emptyset}=1$.

\begin{ex}
When $T=\emptyset$ and $\alpha=(n)$ with $\ell(\alpha)=1$, Proposition \ref{coprodQSym} gives $\Delta(M_{n})=M_{n}\otimes1+1 \otimes M_{n}$ (so that $\overline{\Delta}(M_{n})=0$, namely $M_{n}$ is \emph{primitive}) and $\Delta(F_{n})=\sum_{r=0}^{n}F_{r} \otimes F_{n-r}$. In the opposite extreme, where $T\subseteq\mathbb{N}_{n-1}$ and $\alpha$ is the longest composition $1^{n}$, we have $F_{1^{n}}=M_{1^{n}}$, and their $\Delta$-image is $\sum_{r=0}^{n}F_{1^{r}} \otimes F_{1^{n-r}}=\sum_{r=0}^{n}M_{1^{r}} \otimes M_{1^{n-r}}$. \label{coprodQSymext}
\end{ex}

\begin{ex}
For the composition $\alpha=213\vDash6$ from Example \ref{excomp}, Proposition \ref{coprodQSym} shows that $\overline{\Delta}(M_{\alpha})=M_{2} \otimes M_{13}+M_{21} \otimes M_{3}$, while $\overline{\Delta}(F_{\alpha})$ is $F_{1} \otimes F_{113}+F_{2} \otimes F_{13}+F_{21} \otimes F_{3}+F_{211} \otimes F_{2}+F_{212} \otimes F_{1}$. \label{DeltaQSymex}
\end{ex}

With Lemma \ref{monQSym} and Definition \ref{DeltaQSym}, $\mathtt{QSym}$ becomes a connected graded bi-algebra via Definitions \ref{graded} and Definition \ref{condef}. It therefore admits an antipode via Lemma \ref{grconS}. Its explicit formula are given in the following result, proved initially in \cite{[MR]} and \cite{[E]} but also later in \cite{[BS]}, \cite{[Gr]}, and others. This paper will also provide a proof, which as far as I know is a novel one, for this formula later.
\begin{thm}
The antipode $S$ on $\mathtt{QSym}$ sends, for every $n$ and $\alpha \vDash n$, the fundamental quasi-symmetric function $F_{\alpha}$ to $(-1)^{n}F_{\alpha^{t}}$. \label{SQSymF}
\end{thm}
From Theorem \ref{SQSymF} one deduces the following expression for $S$ using the monomial quasi-symmetric basis, for which we recall the expressions $E_{\alpha}$ from Definition \ref{fundQSym}. We include the proof, because we will later apply the same argument in other Hopf algebras.
\begin{prop}
If $\alpha \vDash n$ then we have $S(M_{\alpha})=(-1)^{\ell(\alpha)}E_{\alpha^{r}}$. \label{SQSymM}
\end{prop}

\begin{proof}
The assertion is clear for $\alpha=\emptyset$, so that we may assume that $n>0$ and $\alpha=\operatorname{comp}_{n}T$ for $T\subseteq\mathbb{N}_{n-1}$. We claim that $M_{\alpha}=\sum_{\beta\preceq\alpha}(-1)^{\ell(\beta)-\ell(\alpha)}F_{\beta}$. To see this, we recall that for every $\beta \vDash n$ we have $F_{\beta}=\sum_{\gamma\preceq\beta}M_{\gamma}$, which we plug into the right hand side, and consider, for every $\gamma\preceq\alpha$, the coefficient of $M_{\gamma}$ in the resulting expression.

We write $\gamma$ as $\operatorname{comp}_{n}Y$ for some subset $Y\subseteq\mathbb{N}_{n-1}$, which contains $T$ because $\gamma\preceq\alpha$. The total expression is $(-1)^{\ell(\beta)-\ell(\alpha)}$ for every $\gamma\preceq\beta\preceq\alpha$. But these compositions $\beta$ are precisely those that are $\operatorname{comp}_{n}X$ for $T \subseteq X \subseteq Y$, and the exponent in the sign is the difference $|X|-|T|$ of the sizes of these sets.

But each element of $Y \setminus T$ can appear in $X$ in an independent manner, and adding it inverts the sign. Hence the total coefficient is $(1-1)^{|Y|-|T|}$, which vanishes wherever $T \subsetneq Y$, or equivalently $\gamma\prec\alpha$ (which means, of course, $\gamma\preceq\alpha$ but $\gamma\neq\alpha$). As for $\gamma=\alpha$ (and $Y=T$) we only get one summand $\beta=\gamma=\alpha$, with the positive sign, our expression indeed produces $M_{\alpha}$.

It follows that $S(M_{\alpha})=\sum_{\beta\preceq\alpha}(-1)^{\ell(\beta)-\ell(\alpha)}S(F_{\beta})$, and Theorem \ref{SQSymF} allows us to replace each $S(F_{\beta})$ by $(-1)^{n}F_{\beta^{t}}$. We expand each $F_{\beta^{t}}$ as $\sum_{\gamma\preceq\beta^{t}}M_{\gamma}$, and we again take some $\gamma \vDash n$ and see the coefficient of $M_{\gamma}$ in this expression.

We write again $\alpha=\operatorname{comp}_{n}T$ and $\gamma=\operatorname{comp}_{n}Y$ for subsets $T$ and $Y$ of $\mathbb{N}_{n-1}$, and each summand $\beta$ will be $\operatorname{comp}_{n}X$ for a subset $X$. We now note that $\gamma\preceq\beta^{t}$ is equivalent to $Y$ containing the complement $X^{t}$ of $X^{r}$. But this is the same as saying that $Y \cup X^{r}=\mathbb{N}_{n-1}$, or equivalently $X \cup Y^{r}=\mathbb{N}_{n-1}$, and hence to $X$ containing the complement $Y^{t}$ of $Y^{r}$.

This means that for every $Y$, the collection of sets $X$ that contribute to the coefficient of $M_{\gamma}$ in $S(M_{\alpha})$ are those that contain both $T$ and $Y^{t}$, namely their union. As they are contained in $\mathbb{N}_{n-1}$, the coefficient is a sign times $(1-1)^{n-1-|T \cup Y^{t}|}$, so that only terms where $T \cup Y^{t}=\mathbb{N}_{n-1}$ survive. Moreover, for every such $Y$, the only set $X$ in the sum is $\mathbb{N}_{n-1}$, corresponding to the longest partition $\beta$, of length $n$.

Now, the sign is $(-1)^{n}$ times $(-1)^{\ell(\beta)-\ell(\alpha)}$ with $\ell(\beta)=n$, which reduces to the asserted sign $(-1)^{\ell(\alpha)}$. Moreover, the condition that $T \cup Y^{t}=\mathbb{N}_{n-1}$ is equivalent, as above, to $T$ containing the complement $Y^{r}$ of $Y^{t}$, which after reversion corresponds to the inclusion $Y \subseteq T^{r}$. As this means that the compositions $\gamma=\operatorname{comp}_{n}Y$ appearing there are those satisfying $\gamma\succeq\operatorname{comp}_{n}T^{r}=\alpha^{r}$, this sum is $E_{\alpha^{r}}$ via Definition \ref{fundQSym}. This completes the proof of the proposition.
\end{proof}

\begin{ex}
In the extreme cases $\alpha=n \vDash n$ and $\alpha=1^{n}$, Theorem \ref{SQSymF} produces $S(F_{n})=(-1)^{n}F_{1^{n}}$ and $S(F_{1^{n}})=(-1)^{n}F_{n}$. As we have $F_{1^{n}}=M_{1^{n}}$ and $E_{1^{n}}=F_{n}$ via Definition \ref{fundQSym}, the latter antipode equality also reads as $S(M_{1^{n}})=(-1)^{n}E_{1^{n}}$, in correspondence with Proposition \ref{SQSymM} via the fact that $\alpha=1^{n}$ satisfies $\ell(\alpha)=n$ and $\alpha^{r}=\alpha$. For $M_{n}$, again associated with a composition that is invariant under reversal, the fact that $E_{n}=M_{n}$ and that composition has length 1 shows, via Proposition \ref{SQSymM}, that $S(M_{n})=-M_{n}$. \label{antQSymext}
\end{ex}
The value of $S(M_{n})$ in Example \ref{antQSymext} corresponds to the fact that $M_{n}$ was seen to be primitive in Example \ref{coprodQSymext}, and thus must be taken to its additive inverse via Lemma \ref{grconS}.
\begin{ex}
If we once again set $\alpha=213\vDash6$ as in Examples \ref{excomp} and \ref{DeltaQSymex}, with $\ell(\alpha)=3$, then Theorem \ref{SQSymF} and Proposition \ref{SQSymM} produce, via the values from Example \ref{invex}, the expressions $S(F_{\alpha})=F_{1131}$ and $S(M_{\alpha})=-E_{312}$. \label{SQSymex}
\end{ex}

\medskip

We comment briefly on how $\mathtt{QSym}$ contains the Hopf algebra $\mathtt{Sym}$ of symmetric functions. Recall that the natural action of $S_{\mathbb{N}}$ on $R\ldbrack\mathbf{x}_{\infty}\rdbrack$ coincides with that from \cite{[Hi2]} on $R\ldbrack\mathbf{x}_{\infty}\rdbrack_{1}$ (and on $R\ldbrack\mathbf{x}_{\infty}\rdbrack_{0}=R$ trivially) and respects the ring operations, and $\mathtt{Sym}$ consists of those elements of $R\ldbrack\mathbf{x}_{\infty}\rdbrack$ that have finite degree and are invariant under this operation.

The unit and co-unit are the identification of $R$ with $R\ldbrack\mathbf{x}_{\infty}\rdbrack_{0}$, the multiplication is the one induced from $R\ldbrack\mathbf{x}_{\infty}\rdbrack$, and the co-multiplication is the restriction of the one from Definition \ref{DeltaQSym}. The details of the containment are as follows.
\begin{rmk}
A partition $\lambda \vdash n$ of length $\ell$ represents an orbit of compositions $\alpha \vDash n$ of length $\ell$ under the action of the symmetric group $S_{\ell}$ interchanging the order of the entries. The monomial symmetric function $m_{\lambda}\in\mathtt{Sym}$ associated $\lambda$ is expressed, inside $\mathtt{QSym}$, as $\sum_{\alpha \in S_{\ell}\lambda}M_{\alpha}$. Then $\mathtt{Sym}$ is a commutative and co-commutative Hopf sub-algebra of $\mathtt{QSym}$ (which is commutative but not co-commutative). The elementary symmetric function $e_{n}$ is $m_{1^{n}}=M_{1^{n}}=F_{1^{n}}$ (where $1^{n}$ is the longest partition of $n$), and the complete homogeneous symmetric function $h_{n}$ is $F_{n}$. \label{SymQSym}
\end{rmk}
We recall that the antipode on $\mathtt{Sym}$ is the involution $\omega$ but with elements of degree $n$ multiplied by $(-1)^{n}$. Since $\omega$ interchanges $e_{n}$ and $h_{n}$, this is in correspondence with Theorem \ref{SQSymF} for these compositions, and extended by the algebra (and even bi-algebra) homomorphism property (as follows from the commutativity and co-commutativity of $\mathtt{Sym}$ via Proposition \ref{Sprop}).

\section{Deformations of Quasi-Symmetric Functions \label{qDefs}}

There are two types of $q$-deformations of $\mathtt{QSym}$, one which is commutative, and the other is not (and produces a more general notion than a Hopf algebra).

\subsection{The Commutative $q$-Deformation}

Recall from Lemma \ref{monQSym} that in the monomial basis $\{M_{\alpha}\}_{\alpha}$, the Hopf algebra $\mathtt{QSym}$ is a quasi-shuffle algebra. For $q$-deformation we consider modules over the ring of polynomials $R[q]$ in a variable $q$ over $R$. The commutative $q$-deformation, which is the $q$-shuffle algebra,
\begin{defn}
Let $\mathtt{QSym}_{q}$ be the algebra spanned freely over $R[q]$ by basis elements $M_{\alpha,q}$ where $\alpha$ runs over all compositions, where $M_{\alpha,q}$ is of degree $n$ when $\alpha \vDash n$. We define the unit and co-unit of $\mathtt{QSym}_{q}$ to be the isomorphisms between $R$ and the degree 0 part $R[q]M_{\emptyset,q}$ (namely $M_{\emptyset,q}=1$), and non-trivial compositions $\alpha=(\alpha_{1},\bar{\alpha})$ and $\beta=(\beta_{1},\bar{\beta})$ the product is defined by the \emph{$q$-shuffle rule} $M_{\alpha,q}M_{\beta,q}=l_{\alpha_{1}}M_{\bar{\alpha},q}M_{\beta,q}+l_{\beta_{1}}M_{\alpha,q}M_{\bar{\beta},q}+ql_{\alpha_{1}+\beta_{1}}M_{\bar{\alpha},q}M_{\bar{\beta},q}$ The co-multiplication is again the one from Definition \ref{DeltaQSym}. \label{comqdef}
\end{defn}
Also in $\mathtt{QSym}_{q}$ from Definition \ref{comqdef} there are analogues of Definition \ref{fundQSym}.
\begin{defn}
For every composition $\alpha \vDash n$ set $F_{\alpha,q}:=\sum_{\beta\preceq\alpha}q^{n-\ell(\beta)}M_{\beta,q}$, as well as $E_{\alpha,q}:=\sum_{\beta\succeq\alpha}q^{\ell(\alpha)-\ell(\beta)}M_{\beta,q}$, so that in particular for the empty composition we have $F_{\emptyset,q}=E_{\emptyset,q}=M_{\emptyset,q}=1$. \label{fundqQSym}
\end{defn}
It is clear from Definition \ref{refine} that if $\beta\succeq\alpha$ then we have $\ell(\alpha)\geq\ell(\beta)$, and the equality $\ell(\beta) \leq n$ is evident when $\alpha \vDash n$. Hence the exponents of $q$ in Definition \ref{fundqQSym} are all non-negative.

Substituting $q=1$ in Definitions \ref{comqdef} and \ref{fundqQSym} produces Definitions \ref{QSymdef} and \ref{fundQSym} and Lemma \ref{monQSym} back again. For $q=0$ we obtain a shuffle algebra with an alternative grading, as described in Remark \ref{shdual} below.

\medskip

As $\mathtt{QSym}_{q}$ from Definition \ref{comqdef} is a connected graded bi-algebra, Lemma \ref{grconS} implies that it is a Hopf algebra. The formula for the antipode is given by the following analogue of Theorem \ref{SQSymF} and Proposition \ref{SQSymM}.
\begin{thm}
The antipode on $\mathtt{QSym}_{q}$ takes $F_{\alpha,q}$ for $\alpha \vDash n$ to $(-1)^{n}F_{\alpha^{t},q}$, and $M_{\alpha,q}$ to $(-1)^{\ell(\alpha)}E_{\alpha^{r},q}$. \label{SqQSym}
\end{thm}
One can prove Theorem \ref{SqQSym} by the same method applied for Theorem \ref{SQSymF} and Proposition \ref{SQSymM} (in this paper or others), but we will obtain it as a direct consequence of these results as follows.
\begin{proof}
We claim that $\mathtt{QSym}_{q}$ embeds as a Hopf subalgebra inside the scalar extension of $\mathtt{QSym}$ from $R$ to $R[q]$. Consider the map sending $M_{\alpha,q}$ to $q^{\ell(\alpha)}M_{\alpha}$, which respects the gradings. As with $\alpha=\emptyset$ it takes $M_{\emptyset,q}=1$ to $q^{\ell(\emptyset)}M_{\emptyset}=1$, this map commutes with the units and co-units.

Recall that if $\beta\gamma=\alpha$ (or equivalently $\beta=\alpha_{|r}$ and $\gamma=\alpha_{r|}$ for $r \in T\cup\{0,n\}$ when $\alpha=\operatorname{comp}_{n}T$) then $\ell(\alpha)=\ell(\beta)+\ell(\gamma)$. Thus, each summand $M_{\beta,q} \otimes M_{\gamma,q}$ appearing in $\Delta(M_{\alpha,q})$ is sent by our map to $q^{\ell(\beta)}M_{\beta} \otimes q^{\ell(\gamma)}M_{\gamma}=q^{\ell(\alpha)}M_{\beta} \otimes M_{\gamma}$, and the sum is indeed the $\Delta$-image of $q^{\ell(\alpha)}M_{\alpha}$.

For the product, we argue by induction on the sum of the lengths of $\alpha$ and $\beta$, where the behavior with respect to the empty partition gives us the base case. We thus express $M_{\alpha,q}M_{\beta,q}$ via the $q$-shuffle rule, and the induction hypothesis is that the three products $M_{\bar{\alpha},q}M_{\beta,q}$, $M_{\alpha,q}M_{\bar{\beta},q}$, and $M_{\bar{\alpha},q}M_{\bar{\beta},q}$ are taken, by our map, to $q^{\ell(\bar{\alpha})+\ell(\beta)}M_{\bar{\alpha}}M_{\beta}$, $q^{\ell(\alpha)+\ell(\bar{\beta})}M_{\alpha}M_{\bar{\beta}}$, and $q^{\ell(\bar{\alpha})+\ell(\bar{\beta})}M_{\bar{\alpha}}M_{\bar{\beta}}$ respectively.

However, we apply $l_{\alpha_{1}}$ to the first expression, $l_{\beta_{1}}$ to the second one, and $ql_{\alpha_{1}+\beta_{1}}$ to the third one, before applying the map. This increases the length of every composition by 1, and thus in the map we have to add 1 to all the powers of $q$, with another power in the third one from the extra factor. As $\ell(\alpha)=\ell(\alpha)+1$ and $\ell(\beta)=\ell(\bar{\beta})+1$, this produces a total coefficient of $q^{\ell(\alpha)+\ell(\beta)}$ times the expression $l_{\alpha_{1}}M_{\bar{\alpha}}M_{\beta}+l_{\beta_{1}}M_{\alpha}M_{\bar{\beta}}+l_{\alpha_{1}+\beta_{1}}M_{\bar{\alpha}}M_{\bar{\beta}}$ for $M_{\alpha}M_{\beta}$ from Lemma \ref{monQSym}. This is thus the product of the images $q^{\ell(\alpha)}M_{\alpha}$ and $q^{\ell(\beta)}M_{\beta}$, as desired.

Hence our map is an embedding of connected graded bi-algebras. It thus has to commute with the two antipodes. The fact that our map is clearly injective allows us to determine the antipode on $\mathtt{QSym}_{q}$ by applying the map into $\mathtt{QSym}$, evaluate the antipode there, and take the desired pre-image.

Now, given $\alpha \vDash n$, our map sends $F_{\alpha,q}$ to $\sum_{\beta\preceq\alpha}q^{n-\ell(\beta)} \cdot q^{\ell(\beta)}M_{\beta}$, which equals $q^{n}F_{\alpha}$. As the antipode of $\mathtt{QSym}$ takes it to $(-1)^{n}q^{n}F_{\alpha^{t}}$ via Theorem \ref{SQSymF}, which is the image of $(-1)^{n}F_{\alpha^{t},q}$ by our map, the first assertion is established.

Similarly, Proposition \ref{SQSymM} implies that the antipode of $\mathtt{QSym}$ sends the image $q^{\ell(\alpha)}M_{\alpha}$ of $M_{\alpha,q}$ to $(-1)^{\ell(\alpha)}q^{\ell(\alpha)}E_{\alpha^{r}}$. We decompose the latter symmetric function as $\sum_{\beta\succeq\alpha^{r}}M_{\beta}$, and after we multiply each summand by $q^{\ell(\alpha)}=q^{\ell(\alpha^{r})}$, we get the product of $q^{\ell(\alpha^{r})-\ell(\beta)}$ with $q^{\ell(\beta)}M_{\beta}$. As the latter multiplier is the image of $M_{\beta,q}$, we deduce via Definition \ref{fundqQSym} that $(-1)^{\ell(\alpha)}q^{\ell(\alpha)}E_{\alpha^{r}}$ is the image of $(-1)^{\ell(\alpha)}E_{\alpha^{r},q}$, yielding the second assertion as well. This completes the proof of the theorem.
\end{proof}
Of course, substituting $q=1$ in Theorem \ref{SqQSym} gives back Theorem \ref{SQSymF} and Proposition \ref{SQSymM} (this is circular reasoning with our proof, but one can prove it without these results, and thus get a non-circular argument there).

\begin{defn}
For every $R$-module $M$, the \emph{tensor algebra} $T(M)$, which is a graded algebra whose part of degree $n$ is $M^{\otimes n}$ and whose multiplication is a concatenation of tensors. The co-product (which is co-multiplicative) takes a tensor $x_{1}\otimes\ldots \otimes x_{n}$ to the sum over all partitions of $\mathbb{N}_{n}$ as a disjoint union $J_{l} \cup J_{r}$ of $x_{J_{l}} \otimes x_{J_{r}}$, where $x_{I}$ for $I\subseteq\mathbb{N}_{n}$ is the tensor product of $\{x_{i}\}_{i \in I}$ in the order they show up (this can be expressed in terms of the sets $\mathcal{S}_{r,n-r}$ from Definition \ref{shuffle} below). The resulting antipode sends a tensor of $n$ elements to $(-1)^{n}$ times the tensor of the same elements in the opposite order.
\label{tensalg}
\end{defn}
The structure from Definition \ref{tensalg}, which when dualizing the module $M$ yields the construction from Proposition \ref{grdual}, and is defined as follows.
\begin{defn}
If $M$ is an $R$-module, then we write $T(M)^{*}$ for the same module $T(M)$ from Definition \ref{tensalg}, but with the following operations. The product of two tensors, of lengths $r$ and $s$, is the sum of all their shuffles (also rigorously defined using Definition \ref{shuffle} below), and is thus commutative). The co-multiplication sends each tensor $x_{1}\otimes\ldots \otimes x_{n}$ to the expression $\sum_{r=0}^{n}[x_{1}\otimes\ldots \otimes x_{r}]\otimes[x_{r+1}\otimes\ldots \otimes x_{n}]$, where the tensors inside the brackets are those from the definition of the module $T(M)$, and the one between them represents the landing inside $T(M)^{*}\otimes_{R}T(M)^{*}$. Also here the antipode reverses the order of tensor products and multiplies by the same sign. This Hopf algebra is the \emph{shuffle algebra} of $M$. \label{tensfual}
\end{defn}

\begin{rmk}
Consider the module $M$ over $R$ that is generated by the free basis $\{v_{a}\}_{a=1}^{n}$. Definition \ref{tensfual} produces the Hopf algebra $T(M)^{*}$, in which the grading arises from considering every $v_{a}$ as having degree 1. If we identify $M_{\alpha,0}$ for $\alpha=(\alpha_{1},\ldots,\alpha_{\ell}) \vDash n$ with $v_{\alpha_{1}}\otimes \ldots \otimes v_{\alpha_{\ell}}$, then it has degree $\ell(\alpha)$. However, when we change the grading such that $v_{a}$ has degree $a$, so that our analogue of $M_{\alpha,0}$ has degree $n$ rather than $\ell(\alpha)$, the resulting algebra is the one obtained by substituting $q=0$ in $\mathtt{QSym}_{q}$. \label{shdual}
\end{rmk}

\begin{rmk}
To connect Remark \ref{shdual} to Theorem \ref{SqQSym}, note that substituting $q=0$ in Definition \ref{fundqQSym} reduces $E_{\alpha,0}$ to just $M_{\alpha}$ (as if $\beta\succ\alpha$ then the exponent $\ell(\alpha)-\ell(\beta)$ of $q$ there is positive). Thus sending $M_{\alpha,0}$ to the expression $(-1)^{\ell(\alpha)}E_{\alpha^{r},0}=(-1)^{\ell(\alpha)}M_{\alpha^{r},0}$ from that theorem is precisely the antipode from that remark, via Definition \ref{tensfual}. For $F_{\alpha,0}$ we have $n-\ell(\beta)>0$ for every $\beta\neq1^{n}$, and $1^{n}\succeq\alpha$ for every $\alpha \vDash n$, so that $F_{\alpha,0}=M_{1^{n},0}$ for all $\alpha \vDash n$. As $F_{\alpha^{t},0}=M_{1^{n},0}$ as well, this part of Theorem \ref{SqQSym} reduces, for $q=0$, to the equality $S(M_{1^{n},0})=(-1)^{n}M_{1^{n},0}$, independently of $\alpha$. \label{q0com}
\end{rmk}

\subsection{The Non-Commutative $q$-Deformation}

We now turn to describing the non-commutative $q$-deformation of $\mathtt{QSym}$. The algebra structure is defined following \cite{[TU]}, with the co-algebra being the one from \cite{[L]}. The resulting structure is, in fact, not a Hopf algebra as in Definition \ref{Hopfdef} (nor a bi-algebra via Definition \ref{bialg}), but rather a more general notion of a $q$-Hopf algebra, as mentioned in \cite{[L]} (a related point of view appears in Equation (10.5.10) of \cite{[Mo]})---see Remark \ref{qHopf} below.

\medskip

Let $R\langle\langle\mathbf{x}_{\infty}\rangle\rangle$ be the ring of power series in the variables $\{x_{i}\}_{i=1}^{\infty}$, but now they are non-commuting. Let $R[q]\langle\langle\mathbf{x}_{\infty}\rangle\rangle$ be defined in the same manner, and we define $R\langle\langle\mathbf{x}_{\infty}\rangle\rangle^{(q)}$ to be the quotient of $R[q]\langle\langle\mathbf{x}_{\infty}\rangle\rangle$ by the ideal generated by the expressions $x_{j}x_{i}-qx_{i}x_{j}$ wherever $i<j$. As the ideal is homogeneous, the homogeneous part $R\langle\langle\mathbf{x}_{\infty}\rangle\rangle^{(q)}_{d}$ of degree $d$ there is well-defined for any $d$.

Every monomial in $R[q]\langle\langle\mathbf{x}_{\infty}\rangle\rangle$ is equal to a power of $q$ times a monomial in which the variables are multiplied in non-decreasing order. Thus, whenever $I:=\{i_{j}\}_{j=1}^{\ell}$ with $i_{j}<i_{j+1}$ for any $1 \leq j<\ell$ and $\alpha$ is some composition of length $\ell$, the monomial $x_{I}^{\alpha}:=\prod_{j=1}^{\ell}x_{i_{j}}^{\alpha_{j}}$ will be understood, inside $R\langle\langle\mathbf{x}_{\infty}\rangle\rangle^{(q)}$, as the product with the indices increasing from left to right.

It follows that as $R[q]$-modules, $R\langle\langle\mathbf{x}_{\infty}\rangle\rangle^{(q)}$ is isomorphic to $R[q]\ldbrack\mathbf{x}_{\infty}\rdbrack$. Hence we may again consider the action of from \cite{[Hi2]}, and make the following definition.
\begin{defn}
We say that an element $F \in R\langle\langle\mathbf{x}_{\infty}\rangle\rangle^{(q)}$ is a \emph{$q$-quasi-symmetric function} (over $R$) when its degree is finite and the action from \cite{[Hi2]}, as considered here, leaves it invariant. Equivalently, $F$ contains monomials $x_{I}^{\alpha}$ with $\alpha \vDash n$ for bounded $n$, and if $J$ has the same length as $I$ and $\alpha$ then $x_{J}^{\alpha}$ appears in $F$ with the same coefficient as $x_{I}^{\alpha}$. The set of $q$-quasi-symmetric functions will be denoted by $\mathtt{QSym}^{(q)}$. \label{QSymqdef}
\end{defn}
We deduce the following result from Definition \ref{QSymqdef}.
\begin{lem}
Given $\alpha \vDash d$ of length $\ell$, write $M_{\alpha}^{(q)}:=\sum_{|I|=\ell}x_{I}^{\alpha} \in R\langle\langle\mathbf{x}_{\infty}\rangle\rangle^{(q)}$, which we call the \emph{monomial $q$-quasi-symmetric function} corresponding to $\alpha$, so that in particular $M_{\emptyset}^{(q)}=1$. Then $\{M_{\alpha}^{(q)}\}_{\alpha}$, with $\alpha$ taken over the set of all compositions, is a basis for $\mathtt{QSym}^{(q)}$, and the product from $R\langle\langle\mathbf{x}_{\infty}\rangle\rangle^{(q)}$ reduces to $M_{\alpha}^{(q)}M_{\beta}^{(q)}=l_{\alpha_{1}}M_{\bar{\alpha}}^{(q)}M_{\beta}^{(q)}+q^{\beta_{1}(\alpha_{1}+m)}l_{\beta_{1}}M_{\alpha}^{(q)}M_{\bar{\beta}}^{(q)}+  q^{\beta_{1}m}l_{\alpha_{1}+\beta_{1}}M_{\bar{\alpha}}^{(q)}M_{\bar{\beta}}^{(q)}$ wherever $\alpha=(\alpha_{1},\bar{\alpha})$ with $\bar{\alpha} \vDash m$ and $\beta=(\beta_{1},\bar{\beta})$. \label{monQSymq}
\end{lem}
The multiplication rule from Lemma \ref{monQSymq} is the one appearing in \cite{[TU]}, \cite{[Ho1]}, and \cite{[L]}, and is also referred to as a \emph{quantum quasi-shuffle product} in \cite{[JRZ]}. That lemma implies that with the unit identifying $R$ with $RM_{\emptyset}^{(q)}=R\langle\langle\mathbf{x}_{\infty}\rangle\rangle^{(q)}_{0}$ (which is the part of degree 0 in $\mathtt{QSym}^{(q)}$), $\mathtt{QSym}^{(q)}$ is a graded algebra in the sense of Definition \ref{graded}, which is now non-commutative.

We extend the notions from Definition \ref{pathsmult} as follows.
\begin{defn}
Take an element $\omega \in P_{\ell,k}^{h}$ for some $0 \leq h\leq\min\{\ell,k\}$.
\begin{enumerate}[$(i)$]
\item We say that a pair of locations $1 \leq i<j\leq\ell+k-j$ is an \emph{inversion} for $\omega$ of the $i$th and $j$th entries of $\omega$ are rl, rb, bl, or bb. We denote the set of inversions of $\omega$ by $\operatorname{Inv}(\omega)$.
\item If $\alpha$ and $\beta$ are compositions of lengths $\ell$ and $k$ respectively, then it follows from Definition \ref{pathsmult} that whenever $(i,j)$ is an inversion for $\omega$, the $i$th entry of $\gamma_{\omega}(\alpha,\beta)$ is an entry of $\beta$ (plus perhaps one of $\alpha$), and the $j$th one is an entry of $\alpha$ (plus perhaps one of $\beta$). We shall write the former as $\beta_{\omega,i}$ and the latter as $\alpha_{\omega,j}$.
\item We define the \emph{$\omega$-inversion number} of $\alpha$ and $\beta$ to be $\sum_{(i,j)\in\operatorname{Inv}(\omega)}\beta_{\omega,i}\alpha_{\omega,j}$, and denote it by $\operatorname{inv}_{\omega}(\alpha,\beta)$. It vanishes, of course, when $\operatorname{Inv}(\omega)=\emptyset$.
\end{enumerate} \label{pathsinv}
\end{defn}

Here is the expression from Corollary \ref{prodMs} for this case, using the parameters from Definition \ref{pathsinv}.
\begin{cor}
Let $\alpha$ and $\beta$ be compositions, with $\ell(\alpha)=\ell$ and $\ell(\beta)=k$. We then have the equality $M_{\alpha}^{(q)}M_{\beta}^{(q)}$ is given by $\sum_{\omega \in P_{\ell,k}}q^{\operatorname{inv}_{\omega}(\alpha,\beta)}M_{\gamma_{\omega}(\alpha,\beta)}^{(q)}$. \label{prodMqs}
\end{cor}
The unique element of $P_{\ell,k}$ when $\ell=0$ or $k=0$ contains only a single letter r or l, so that its inversion number vanishes, in correspondence with the fact that $M_{\emptyset}^{(q)}=1$. Otherwise, we have the following observation.
\begin{rmk}
When $\ell$ and $k$ are both positive, there are precisely two elements $\omega \in P_{\ell,k}$ for which $\operatorname{Inv}(\omega)=\emptyset$. One is the word in $P_{\ell,k}^{0}$ that consists of $\ell$ instances of l followed by $k$ r's, and the second one, which lies in $P_{\ell,k}^{1}$, begins with $\ell-1$ instances of l, followed by a single b, and ends with the $k-1$ r's. The corresponding composition $\gamma_{\omega}(\alpha,\beta)$ is the concatenation $\alpha\beta$ with the former one, and the near concatenation $\alpha\odot\beta$ with the latter. \label{invw0}
\end{rmk}

We consider the cases from Example \ref{pathex}.
\begin{ex}
There are no inversions for lr and b from $P_{1,1}$, but rl has one. Substituting into Corollary \ref{prodMqs}, we get $M_{a}^{(q)}M_{b}^{(q)}=M_{a,b}+q^{ab}M_{b,a}+M_{a+b}$ (this differs from the opposite product $M_{b}^{(q)}M_{a}^{(q)}$, in which $q^{ab}$ will now multiply $M_{a,b}$ instead of $M_{b,a}$). In $P_{2,1}$, the elements llr and lb have no inversions, while lrl and bl has one inversion and llr has two. The resulting expression for $M_{a,b}^{(q)}M_{c}^{(q)}$ is thus $M_{a,b,c}^{(q)}+M_{a,b+c}^{(q)}+q^{bc}M_{a,c,b}^{(q)}+q^{bc}M_{a+c,b}^{(q)}+q^{ac+bc}M_{c,a,b}^{(q)}$. In a similarly, lrr and br have no inversions in $P_{1,2}$, while there is one inversion for each of rlr and rb and two for rrl. The product $M_{c}^{(q)}M_{a,b}^{(q)}$ therefore  equals $M_{c,a,b}^{(q)}+M_{c+a,b}^{(q)}+q^{ac}M_{a,c,b}^{(q)}+q^{ac}M_{a,c+b}^{(q)}+q^{ac+bc}M_{a,b,c}^{(q)}$. \label{pathexq}
\end{ex}
The fact that $\mathtt{QSym}$ is commutative while $\mathtt{QSym}^{(q)}$ is not is related to the fact, already visible in Example \ref{pathexq}, that the involution $\omega\mapsto\tilde{\omega}$ between $P_{\ell,k}$ and $P_{k,\ell}$ does not preserve inversion numbers (and even relate sequences without inversions with those that do have them). The sets of inversions, and the corresponding inversion numbers for appropriate $\alpha$ and $\beta$, in the cases from Example \ref{expath} are as follows.
\begin{ex}
Assume that the composition of length 3 is $(a,b,c)$ and that of length 1 consists of the entry $d$. Then lllr and llb from $P_{3,1}$ and lrrr and brr from $P_{1,3}$ have no inversions (and hence vanishing inversion numbers). The elements llrl and lbl of $P_{3,1}$ have a single inversion, with inversion number $cd$, as do rlrr and rbr from $P_{1,3}$, producing the inversion number $ad$. For lrll and bll from $P_{3,1}$ we have two inversions and the inversion number $bd+cd$, and there are two inversions also in rrlr and rrb from $P_{1,3}$, whose inversion numbers are both $ad+bd$. The remaining elements, rlll in $P_{3,1}$ and rrrl from $P_{1,3}$, have three inversions and the number $ad+bd+cd$. Inside $P_{2,2}$, where $\alpha=(a,b)$ and $\beta=(c,d)$, the elements llrr and lbr have no inversions (hence inversion number 0), while lrlr, lrb, blr, and bb have a single inversion each and the inversion number $bc$. For lrrl and brl we have two inversions and inversion number $bc+bd$, while rllr and rlb also have two inversions but their inversion number is $ac+bc$. Next we have rlrl and rbl with three inversions each and inversion number $ac+bc+bd$, and the remaining element rrll has four inversions, yielding the inversion number $ac+ad+bc+bd$. \label{expathq}
\end{ex}
As Example \ref{expathq} illustrates, there is a unique element of each $P_{\ell,k}$, namely the one of $P_{\ell,k}^{0}$ starting with $k$ r's followed by $\ell$ l's, which has the maximal number of inversions (namely $\ell k$), and whose inversion number with $\alpha \vDash n$ of length $\ell$ and $\beta \vDash m$ of length $k$ is $nm$.

\medskip

We recall from Proposition \ref{Sprop} that antipodes reverse the order of multiplications (as is visible in, e.g., Definition \ref{tensfual} and Remark \ref{shdual}). Our analogue of Definitions \ref{fundQSym} and \ref{fundqQSym} will thus involve several objects---see Sections 4 and 5 of \cite{[TU]} for some related definitions.
\begin{defn}
Let $\alpha=(\alpha_{1},\ldots,\alpha_{\ell})$ be a composition of a number $n$.
\begin{enumerate}[$(i)$]
\item The \emph{inversion number} of $\alpha$ is $\operatorname{inv}(\alpha):=\prod_{1 \leq i<j\leq\ell}\alpha_{i}\alpha_{j}$.
\item The \emph{fundamental $q$-quasi-symmetric function} is $F_{\alpha}^{(q)}:=\sum_{\beta\preceq\alpha}M_{\beta}^{(q)}$.
\item We also define the \emph{modified fundamental $q$-quasi-symmetric function} that is associated with $\alpha$ by $\widetilde{F}_{\alpha}^{(q)}:=\sum_{\beta\preceq\alpha}q^{\operatorname{inv}(\beta)-\operatorname{inv}(\alpha)}M_{\beta}^{(q)}$.
\item Finally, we set $E_{\alpha}^{(q)}:=\sum_{\beta\succeq\alpha}q^{\operatorname{inv}(\beta)}M_{\beta}^{(q)}$.
\end{enumerate} \label{fundQSymq}
\end{defn}
The inversion number $\operatorname{inv}(\alpha)$ from Definition \ref{fundQSymq} should not be confused with the $\omega$-inversion number $\operatorname{inv}_{\omega}(\alpha,\beta)$ appearing in Definition \ref{pathsinv}, as we will not use the explicit formula from Corollary \ref{prodMqs} in this paper.

The trivial examples for Definition \ref{fundQSymq} are when $\ell(\alpha)=1$, in which case $\operatorname{inv}(\alpha)=0$ (and we clearly have $\operatorname{inv}(\emptyset)=0$ as well). For the longest composition $1^{n}$ of $n$ we have $\operatorname{inv}(1^{n})=\binom{n+1}{2}$. For the other compositions of $n=3$ and $n=4$, we get the $\operatorname{inv}$-value 2 for $(2,1)$ and $(1,2)$, 3 when the composition is $(3,1)$ or $(1,3)$, we have $\operatorname{inv}(2,2)=4$, and for $(2,1,1)$, $(1,2,1)$, and $(1,1,2)$ the value is 5. If $\alpha=312\vDash6$ as in Example \ref{excomp}, then $\operatorname{inv}(\alpha)=11$.

\begin{ex}
For every $n\geq0$ we have $\widetilde{F}_{1^{n}}^{(q)}=F_{1^{n}}^{(q)}=M_{1^{n}}^{(q)}$, as well as $E_{n}^{(q)}=M_{n}^{(q)}$. The smallest cases not covered by these formulae are given by $\widetilde{F}_{2}^{(q)}=M_{2}^{(q)}+qM_{11}^{(q)}$ (compare with $F_{2}^{(q)}=M_{2}^{(q)}+M_{11}^{(q)}$), which coincides with $E_{11}^{(q)}$ (in fact, the equality $E_{1^{n}}^{(q)}=\widetilde{F}_{n}^{(q)}$ also holds for any $n\geq0$). The remaining modified fundamental $q$-quasi-symmetric functions from Definition \ref{fundQSymq} for $n=3$ are $\widetilde{F}_{21}^{(q)}=M_{21}^{(q)}+qM_{111}^{(q)}$, $\widetilde{F}_{12}^{(q)}=M_{12}^{(q)}+qM_{111}^{(q)}$, and the longest one $\widetilde{F}_{3}^{(q)}=M_{3}^{(q)}+q^{2}M_{21}^{(q)}+q^{2}M_{12}^{(q)}+q^{3}M_{111}^{(q)}$. For $n=4$, the shorter ones are $\widetilde{F}_{211}^{(q)}=M_{211}^{(q)}+qM_{1^{4}}^{(q)}$, $\widetilde{F}_{121}^{(q)}=M_{121}^{(q)}+qM_{1^{4}}^{(q)}$, and $\widetilde{F}_{112}^{(q)}=M_{112}^{(q)}+qM_{1^{4}}^{(q)}$. The modified fundamental $q$-quasi-symmetric functions $\widetilde{F}_{31}^{(q)}$ and $\widetilde{F}_{13}^{(q)}$ are given by $M_{31}^{(q)}+q^{2}M_{211}^{(q)}+q^{2}M_{121}^{(q)}+q^{3}M_{1^{4}}^{(q)}$ and  $M_{13}^{(q)}+q^{2}M_{121}^{(q)}+q^{2}M_{112}^{(q)}+q^{3}M_{1^{4}}^{(q)}$ respectively, and we have $\widetilde{F}_{22}^{(q)}=M_{22}^{(q)}+qM_{211}^{(q)}+qM_{112}^{(q)}+q^{2}M_{1^{4}}^{(q)}$.
\label{exQSymq}
\end{ex}

It is clear that $F_{\emptyset}^{(q)}=\widetilde{F}_{\emptyset}^{(q)}=E_{\emptyset}^{(q)}=M_{\emptyset}^{(q)}=1$, and for the unique partition of $n=1$ again all four $q$-quasi-symmetric functions coincide (these are the case $n=0$ and $n=1$ in Example \ref{exQSymq}). Recall from Proposition \ref{compn} that any $\alpha \vDash n>0$ is $\operatorname{comp}_{n}T$ for some $T\subseteq\mathbb{N}_{n-1}$, and we write $F_{n,T}^{(q)}=F_{\alpha}^{(q)}$ and $\widetilde{F}_{n,T}^{(q)}=\widetilde{F}_{\alpha}^{(q)}$ as in Definition \ref{fundQSym}.

\begin{lem}
The number $\operatorname{inv}$ has the following properties.
\begin{enumerate}[$(i)$]
\item When $\ell(\alpha)=\ell$, the number $\operatorname{inv}$ is constant on the $S_{\ell}$-orbit of $\alpha$.
\item If $\alpha=\beta\gamma$ for $\beta \vDash r$ and $\gamma \vDash s$, then $\operatorname{inv}(\alpha)=\operatorname{inv}(\beta)+\operatorname{inv}(\gamma)+rs$.
\item Recall from Definition \ref{refine} that if $\beta\preceq\alpha$, then $\beta$ is a consecutive sequence of compositions $\gamma_{i}\vDash\alpha_{i}$. We then have $\operatorname{inv}(\beta)=\operatorname{inv}(\alpha)+\sum_{i=1}^{\ell(\alpha)}\operatorname{inv}(\alpha_{i})$.
\item In case $\alpha=\beta\odot\gamma$, with $\ell(\beta)=\ell>0$ and $\gamma\neq\emptyset$, we have the equality $\operatorname{inv}(\alpha)=\operatorname{inv}(\beta)+\operatorname{inv}(\gamma)+rs-\beta_{\ell}\gamma_{1}$.
\end{enumerate}
if $\beta\preceq\alpha$ then one can view $\beta$ as a collection of compositions $\gamma_{i}\vDash\alpha_{i}$, and one can verify that $\operatorname{inv}(\beta)=\operatorname{inv}(\alpha)+\sum_{i=1}^{\ell(\alpha)}\operatorname{inv}(\alpha_{i})\geq\operatorname{inv}(\alpha)$, with equality holding if and only if $\beta=\alpha$. \label{invprop}
\end{lem}

\begin{proof}
Part $(i)$ follows from the fact that $\operatorname{inv}(\alpha)$ is the sum of products of pairs of distinct entries of $\alpha$, irrespective of the order. In the expression for $\operatorname{inv}(\alpha)$ in $(ii)$, the summands $\alpha_{i}\alpha_{j}$ with $i<j$ are separated into three cases. The products with $i<j\leq\ell(\beta)$ combine to $\operatorname{inv}(\beta)$, those in which $\ell(\beta)<i<j$ produce $\operatorname{inv}(\gamma)$, and we consider the produces involving indices $i\leq\ell(\beta)<j$. Since $\alpha_{i}$ runs freely over the entries of $\beta$ (which sum to $r$), $\alpha_{j}$ goes over those of $\gamma$ (whose sum is $s$), and they are independent, the sum is $rs$ as desired.

For part $(iii)$, note that the products $\beta_{j}\beta_{k}$ in the sum defining $\operatorname{inv}(\beta)$ can be partitioned into those of pairs contributing to the same $\alpha_{i}$, and those contributing to two different parts of $\alpha$. An argument like the one proving part $(ii)$ shows that the products arising from the latter sum to $\operatorname{inv}(\alpha)$, and the former yields, for fixed $i$, the expression $\operatorname{inv}(\gamma_{i})$ via Definition \ref{fundQSymq}, and the result follows.

Finally, note that $\beta\gamma$ is a refinement of $\alpha=\beta\odot\gamma$, where $\alpha_{\ell}=\beta_{\ell}+\gamma_{1}$ is replaced by the composition $(\beta_{\ell},\gamma_{1})$, and all the other compositions are of length 1. Part $(iii)$ implies that $\operatorname{inv}(\beta\gamma)-\operatorname{inv}(\beta\odot\gamma)$ is the sum of these partial decompositions, where we know that those of length 1 has a trivial inversion number, and the remaining one have inversion number $\beta_{\ell}\gamma_{1}$. As part $(ii)$ provides us the value of $\operatorname{inv}(\beta\gamma)$, the asserted value of our $\operatorname{inv}(\alpha)$ in part $(iv)$ follows. This completes the proof of the lemma.
\end{proof}

\begin{rmk}
Part $(i)$ of Lemma \ref{invprop} is visible in the examples above. Part $(iii)$ implies, in particular, that if $\beta\preceq\alpha$ then $\operatorname{inv}(\beta)\geq\operatorname{inv}(\alpha)$, with equality holding if and only if $\beta=\alpha$. Moreover, consider a composition $\alpha \vDash n>0$, which is we write as $\operatorname{comp}_{n}T$ for some $T\subseteq\mathbb{N}_{n-1}$, and some $0 \leq r \leq n$. When $r\in\tilde{T}:=T\cup\{0,n\}$, we set $d(\alpha,r):=0$, which we also do in case $n=0$ (and $\alpha=\emptyset$, and $r=0$). Otherwise let $t$ be the maximal element of $\tilde{T}$ that is smaller than $r$, set $s$ to be the minimal element of $\tilde{T}$ that is larger than $r$, and we define $d(\alpha,r):=(r-t)(s-r)>0$. Recalling from Lemma \ref{concexp} that parts $(ii)$ and $(iv)$ correspond to the compositions from Definition \ref{cutatr} (where $s$ is replaced by $n-r$), one translates these parts to the combined equality $\operatorname{inv}(\alpha_{|r})+\operatorname{inv}(\alpha_{r|})+r(n-r)=\operatorname{inv}(\alpha)+d(\alpha,r)$. \label{invineq}
\end{rmk}
It is clear that $\{F_{\alpha}^{(q)}\}_{\alpha}$ from Definition \ref{fundQSymq}, like $\{M_{\alpha}^{(q)}\}_{\alpha}$, a basis for $\mathtt{QSym}^{(q)}$. Remark \ref{invineq} ensures that $\{\widetilde{F}_{\alpha}^{(q)}\}_{\alpha}$ is also contained in that algebra, and forms yet another basis for it.
\begin{ex}
It is clear from Remark \ref{invineq} that $d(1^{n},r)=0$ for all $0 \leq r \leq n$, that $d(n,r)=r(n-r)$, and that $d(\alpha,r)$ is bounded between these two values for every $\alpha \vDash n$ (it thus vanishes if $r=0$ or $r=n$ for any such $\alpha$). The remaining values when $n=3$ are $d(21,1)=d(12,2)=1$ and $d(21,2)=d(12,1)=0$. The numbers $d(211,2)$, $d(211,3)$, $d(121,1)$, $d(121,3)$, $d(112,1)$, and $d(112,2)$ also vanish, and $d(211,1)=d(121,2)=d(112,3)=1$. The other compositions of 4 produce $d(31,3)=d(22,2)=d(13,1)=0$, as well as $d(22,1)=d(22,3)=1$ and $d(31,1)=d(31,2)=d(13,2)=d(13,3)=2$. Our running example $\alpha=213\vDash6$ from Example \ref{excomp} yields, for $d(\alpha,r)$, the values 1, 0, 0, 2, and 2 as $r$ is 1, 2, 3, 4, and 5 respectively. \label{dalpharex}
\end{ex}

\medskip

Here is the analogue of the multisets from Definition \ref{multisets} for this setting.
\begin{defn}
If $J$ is a multiset of size $n$, we write $x^{J}$ for the product $\prod_{k=1}^{n}x_{j_{k}}$, with the indices taken in non-decreasing order. We also denote $x^{\tilde{J}}$ for the same product, but taken in the non-increasing order. \label{multq}
\end{defn}
In order for our analogue of Lemma \ref{sumFnT} in this setting to match the notation, we shall denote $F_{\alpha}^{(q)}$ and $\widetilde{F}_{\alpha}^{(q)}$ from Definition \ref{fundQSymq} also by $F_{n,T}^{(q)}$ and $\widetilde{F}_{n,T}^{(q)}$ respectively in case $\alpha:=\operatorname{comp}_{n}T \vDash n>0$ for some $T\subseteq\mathbb{N}_{n-1}$ (one might also suggest the notation $F_{\tilde{T}}^{(q)}$ and $\widetilde{F}_{\tilde{T}}^{(q)}$ via Remark \ref{gencomp}). Recalling the set $\mathcal{M}_{n,T}$ from Definition \ref{multisets}, and applying Definition \ref{multq}, the argument proving that lemma here yields the following result.
\begin{lem}
Given $T\subseteq\mathbb{N}_{n-1}$ for $n>0$ we have $F_{n,T}^{(q)}=\sum_{J\in\mathcal{M}_{n,T}}x^{J}$, as well as the equality $\sum_{J\in\mathcal{M}_{n,T}}x^{\tilde{J}}=q^{\operatorname{inv}(\alpha)}\widetilde{F}_{n,T}^{(q)}$ for $\alpha:=\operatorname{comp}_{n}T$. \label{sumFnTq}
\end{lem}
The idea for the second equality in Lemma \ref{sumFnTq} is that if $x^{J}$ is a monomial that shows up in $F_{\alpha}^{(q)}$ (which is the same as $\alpha$ describing the composition of coinciding elements of the multiset $J$), then $x^{\tilde{J}}=q^{\operatorname{inv}(\alpha)}x^{J}$. We will mainly be interested, via Lemma \ref{sumFnTq}, in the expressions $q^{\operatorname{inv}(\alpha)}\widetilde{F}_{n,T}^{(q)}$, or equivalently $q^{\operatorname{inv}(\alpha)}\widetilde{F}_{\alpha}^{(q)}$, but our motivation for normalizing $\widetilde{F}_{\alpha}^{(q)}=\widetilde{F}_{n,T}^{(q)}$ as in Definition \ref{fundQSymq} is the basis property mentioned above.

\medskip

Lemma \ref{monQSymq} implies that $\mathtt{QSym}^{(q)}$ is isomorphic, as a module over $R[q]$, to the extension of scalars of $\mathtt{QSym}$ from $R$ to $R[q]$. We may therefore endow it with the co-multiplication given in Definition \ref{DeltaQSym}, and we note that its image is in $\mathtt{QSym}^{(q)}\otimes_{R[q]}\mathtt{QSym}^{(q)}$, and not $\mathtt{QSym}^{(q)}\otimes_{R}\mathtt{QSym}^{(q)}$. The co-unit is, as usual and expected from Lemma \ref{propgr} (over $R[q]$), the isomorphism of the degree 0 part $R[q]M_{\emptyset}^{(q)}$ onto $R[q]$ and vanishes on the other components.

The argument proving Proposition \ref{coprodQSym} then yields the following result.
\begin{prop}
Take a subset $T\subseteq\mathbb{N}_{n-1}$ with $\alpha=\operatorname{comp}_{n}T \vDash n>0$, and set $\tilde{T}:=T\cup\{0,n\}$ as in Lemma \ref{concexp}. Then $\Delta(M_{\alpha}^{(q)})=\sum_{r\in\tilde{T}}M_{\alpha_{|r}}^{(q)} \otimes M_{\alpha_{r|}}^{(q)}$ and $\Delta(F_{\alpha}^{(q)})=\sum_{r=0}^{n}F_{\alpha_{|r}}^{(q)} \otimes F_{\alpha_{r|}}^{(q)}$. For the modified fundamental $q$-quasi-symmetric functions we have $\Delta(q^{\operatorname{inv}(\alpha)}\widetilde{F}_{\alpha}^{(q)})=\sum_{r=0}^{n}q^{r(n-r)} \cdot q^{\operatorname{inv}(\alpha_{|r})}\widetilde{F}_{\alpha_{|r}}^{(q)} \otimes q^{\operatorname{inv}(\alpha_{r|})}\widetilde{F}_{\alpha_{r|}}^{(q)}$ in the rescaling from Lemma \ref{sumFnTq}, and for the usual ones from Definition \ref{fundQSymq}, the equality is $\Delta(\widetilde{F}_{\alpha}^{(q)})=\sum_{r=0}^{n}q^{d(\alpha,r)}\cdot\widetilde{F}_{\alpha_{|r}}^{(q)}\otimes\widetilde{F}_{\alpha_{r|}}^{(q)}$, using the parameter from Remark \ref{invineq}. \label{coprodQSymq}
\end{prop}
The penultimate formula in Proposition \ref{coprodQSymq} is obtained in the same manner as the first two there, and the last one is related to it via Remark \ref{invineq}. It is clear that all the equalities in Proposition \ref{coprodQSymq} hold also when $n=0$, where they all reduce to the assertion that $\Delta(1)=1\otimes1$ since $\widetilde{F}_{\emptyset}^{(q)}=F_{\emptyset}^{(q)}=M_{\emptyset}^{(q)}=1$ and $\operatorname{inv}(\emptyset)=0$. Note that our co-multiplication on $\mathtt{QSym}^{(q)}$ is the one from \cite{[L]}, but \emph{not} the same as that from Section 5 of \cite{[Ho1]}---the latter produces a Hopf algebra, and ours does not, as Remark \ref{qHopf} below details.
\begin{ex}
We have $\Delta(\widetilde{F}_{1^{n}}^{(q)})=\sum_{r=0}^{n}\widetilde{F}_{1^{r}}^{(q)}\otimes\widetilde{F}_{1^{n-r}}^{(q)}$, in correspondence with the formula from Examples \ref{coprodQSymext} via the fact that $\widetilde{F}_{1^{m}}=F_{1^{m}}=M_{1^{m}}$ for every $m\geq0$ and the vanishing of $d(1^{n},r)$ for every $r$ as seen in Example \ref{dalpharex}. In the expressions with the inversion numbers, note that $\operatorname{inv}(1^{n})=\frac{n(n+1)}{2}$, which is the sum of $\operatorname{inv}(1^{r})=\frac{r(r+1)}{2}$, $\operatorname{inv}(1^{n-r})=\frac{(n-r)(n-r+1)}{2}$, and $r(n-r)$. We also get $\Delta(\widetilde{F}_{n}^{(q)})=\sum_{r=0}^{n}q^{r(n-r)}\cdot\widetilde{F}_{r}^{(q)}\otimes\widetilde{F}_{n-r}^{(q)}$, via the values from Example \ref{dalpharex}, which is unaffected by the vanishing inversion numbers of the compositions involved. \label{DeltaQSymqext}
\end{ex}
As $d(\alpha,r)>0$ whenever $r\not\in\tilde{T}$, the only situation in the last equality in Proposition \ref{coprodQSymq} in which no positive powers of $q$ show up is the case where $T=\mathbb{N}_{n-1}$ and $\alpha=1^{n}$ from Example \ref{DeltaQSymqext}, with $\widetilde{F}_{1^{n}}^{(q)}=F_{1^{n}}^{(q)}=M_{1^{n}}^{(q)}$.

\begin{ex}
Take again the composition $\alpha=213\vDash6$ from Example \ref{excomp}, so that Proposition \ref{coprodQSym} and the values from Example \ref{dalpharex} express $\overline{\Delta}(\widetilde{F}_{\alpha}^{(q)})$ as $q\widetilde{F}_{1}^{(q)}\otimes\widetilde{F}_{113}^{(q)}+\widetilde{F}_{2}^{(q)}\otimes\widetilde{F}_{13}^{(q)}+\widetilde{F}_{21}^{(q)}\otimes \widetilde{F}_{3}^{(q)}+q^{2}\widetilde{F}_{211}^{(q)}\otimes\widetilde{F}_{2}^{(q)}+q^{2}\widetilde{F}_{212}^{(q)}\otimes\widetilde{F}_{1}^{(q)}$. One can verify that multiplying it by $q^{\operatorname{inv}(\alpha)=11}$ and putting $q^{\operatorname{inv}(\alpha_{|r})}$ in front of $\widetilde{F}_{\alpha_{|r}}^{(q)}$ as well as $q^{\operatorname{inv}(\alpha_{r|})}$ before $\widetilde{F}_{\alpha_{r|}}^{(q)}$ for every $1 \leq r\leq5$, the remaining powers of $q$ are the respective values 5, 8, 9, 8, and 5 of $r(6-r)$. \label{coprodQSymqex}
\end{ex}
As the formulae for $\Delta(M_{\alpha}^{(q)})$ and $\Delta(F_{\alpha}^{(q)})$ in Proposition \ref{coprodQSymq} are like those from Proposition \ref{coprodQSym}, adding the superscript $(q)$ in Examples \ref{coprodQSymext} and \ref{DeltaQSymex} still produces valid expressions, which is why Examples \ref{DeltaQSymqext} and \ref{coprodQSymqex} only involved the modified fundamental $q$-quasi-symmetric functions.

\medskip

\begin{rmk}
The algebra $\mathtt{QSym}^{(q)}$ is, with the structure from Lemma \ref{monQSymq} and Proposition \ref{coprodQSymq}, \emph{not} a bi-algebra in the sense of Definition \ref{bialg}, as $m$ and $\Delta$ do not satisfy the required equality there. In fact, for homogeneous $x$ and $y$ in that algebra, when we consider the decompositions of $\Delta(x)$ and $\Delta(y)$ as in Remark \ref{mDeltagr}, we obtain the equality $\Delta(xy)=\sum_{r=0}^{n}\sum_{s=0}^{m}q^{s(n-r)}\Delta_{r}^{(n)}(x)\Delta_{s}^{(m)}(y)$ (namely $\Delta_{t}^{(m+n)}(xy)=\sum_{r+s=t}q^{s(n-r)}\Delta_{r}^{(n)}(x)\Delta_{s}^{(m)}(y)$ for every $0 \leq t \leq m+n$). This makes $\mathtt{QSym}^{(q)}$ a \emph{$q$-bi-algebra}, in the terminology of \cite{[EK]}, \cite{[L]} and others, rather than a bi-algebra. However, as all the formulae required for the existence and the properties of the antipode do not use the compatibility between $m$ and $\Delta$ in Definition \ref{bialg}, they are equally valid for $q$-bi-algebras. Thus the connected property implies the existence of an antipode on $\mathtt{QSym}^{(q)}$ as in Lemma \ref{grconS} (so that $\mathtt{QSym}^{(q)}$ is a \emph{$q$-Hopf algebra}), and this antipode can be evaluated via Corollary \ref{detS}. \label{qHopf}
\end{rmk}
Indeed, one can establish via the product rule from Lemma \ref{monQSymq}, by induction, the equality $\Delta(M_{\alpha}^{(q)}M_{\beta}^{(q)})=\sum_{r\in\tilde{T}}\sum_{s\in\tilde{U}}q^{s(n-r)}M_{\alpha_{|r}}^{(q)}M_{\beta_{|s}}^{(q)} \otimes M_{\alpha_{r|}}^{(q)}M_{\beta_{s|}}^{(q)}$ for every pair of compositions $\alpha:=\operatorname{comp}_{n}T \vDash n>0$ and $\beta:=\operatorname{comp}_{m}U \vDash m>0$, where $T\subseteq\mathbb{N}_{n-1}$, $U\subseteq\mathbb{N}_{m-1}$, and $\tilde{T}:=T\cup\{0,n\}$ and similarly $\tilde{U}:=U\cup\{0,m\}$ (the relation in case $\alpha$ or $\beta$ are $\emptyset$, namely $mn=0$, is trivial). This suffices, by the basis property, to produce the formula from Definition \ref{qHopf}, and also gives the equality $\Delta(F_{\alpha}^{(q)}F_{\beta}^{(q)})=\sum_{r=0}^{n}\sum_{s=0}^{m}q^{s(n-r)}F_{\alpha_{|r}}^{(q)}F_{\beta_{|s}}^{(q)} \otimes F_{\alpha_{r|}}^{(q)}F_{\beta_{s|}}^{(q)}$ for the fundamental $q$-quasi-symmetric functions (similar expressions can be written for the modified ones, but we shall not need them).

Here is the simplest non-trivial example of this property.
\begin{ex}
Take $n=m=1$, so that both multipliers are $M_{1}^{(q)}$. Then the product is $M_{2}^{(q)}+M_{11}^{(q)}+qM_{11}^{(q)}$, and its $\Delta$-image is $(M_{1}^{(q)})^{2}\otimes1+1\otimes(M_{1}^{(q)})^{2}$ plus the two terms $M_{1}^{(q)} \otimes M_{1}^{(q)}$ and $q \cdot M_{1}^{(q)} \otimes M_{1}^{(q)}$. When we compare it with the $\Delta$-image $M_{1}^{(q)}\otimes1+1 \otimes M_{1}^{(q)}$ of the two multipliers, we see that both $r$ and $s$ can be either 0 or 1. The term with $r=s=0$ yields $1\otimes(M_{1}^{(q)})^{2}$, and that with $r=s=1$ (so that $n-r=0$) gives $(M_{1}^{(q)})^{2}\otimes1$. One additional term, with $r=1$ and $s=0$, namely the product $(M_{1}^{(q)}\otimes1)(1 \otimes M_{1}^{(q)})$, yields $M_{1}^{(q)} \otimes M_{1}^{(q)}$ as expected. But the product in the opposite order has $r=0$ and $s=1$, so that its contribution is again a copy of $M_{1}^{(q)} \otimes M_{1}^{(q)}$, but multiplied by $q^{s(n-r)=1}$. \label{qex}
\end{ex}

\medskip

The main result in this section is the following antipode formula for $\mathtt{QSym}^{(q)}$.
\begin{thm}
Let $\alpha$ be a composition of $n$, with the inversion number $\operatorname{inv}(\alpha)$, the fundamental quasi-symmetric function $F_{\alpha}^{(q)}$, and the modified fundamental quasi-symmetric function $\widetilde{F}_{\alpha}^{(q)}$ from Definition \ref{fundQSymq}. Then the antipode $S$ on $\mathtt{QSym}^{(q)}$ is given by the formula $S(F_{\alpha}^{(q)})=(-1)^{n}q^{\operatorname{inv}(\alpha^{t})}\widetilde{F}_{\alpha^{t}}^{(q)}$. \label{SQSymqF}
\end{thm}
Using Theorem \ref{SQSymqF}, we obtain the following analogue of Proposition \ref{SQSymM}.
\begin{prop}
For every $\alpha \vDash n$ we have $S(M_{\alpha}^{(q)})=(-1)^{\ell(\alpha)}E_{\alpha^{r}}^{(q)}$. \label{SQSymqM}
\end{prop}

\begin{proof}
We follow the proof of Proposition \ref{SQSymM}. The assertion for $n=0$, namely for $\alpha=\emptyset$, is again trivial, so assume $n>0$. The argument yielding the equality $M_{\alpha}=\sum_{\beta\preceq\alpha}(-1)^{\ell(\beta)-\ell(\alpha)}F_{\beta}$ carries verbatim, with the added superscripts, to produce the equality $M_{\alpha}^{(q)}=\sum_{\beta\preceq\alpha}(-1)^{\ell(\beta)-\ell(\alpha)}F_{\beta}^{(q)}$. Applying Theorem \ref{SQSymqF}, we get $S(M_{\alpha}^{(q)})=\sum_{\beta\preceq\alpha}(-1)^{n+\ell(\beta)-\ell(\alpha)}q^{\operatorname{inv}(\beta)}\widetilde{F}_{\beta^{t}}^{(q)}$.

We decompose each $q^{\operatorname{inv}(\beta)}\widetilde{F}_{\beta^{t}}^{(q)}$ as $\sum_{\gamma\preceq\beta}q^{\operatorname{inv}(\gamma)}M_{\gamma}^{(q)}$ via Definition \ref{fundQSymq}, and we consider, for fixed $\gamma$, the coefficient of $q^{\operatorname{inv}(\gamma)}M_{\gamma}^{(q)}$ in the resulting expression. But this coefficient is exactly the same as in the proof of Proposition \ref{SQSymM}, which was seen to be $(-1)^{\ell(\alpha)}$ in case $\gamma\succeq\alpha^{r}$ and 0 otherwise. Definition \ref{fundQSymq} shows that the total expression is indeed the asserted one. This proves the proposition.
\end{proof}
It is clear that substituting $q=1$ in Theorem \ref{SQSymqF} and Proposition \ref{SQSymqM} produces Theorem \ref{SQSymF} and Proposition \ref{SQSymM} respectively.
\begin{ex}
Theorem \ref{SQSymqF} yields $S(F_{n}^{(q)})=(-1)^{n}q^{n(n+1)/2}\widetilde{F}_{1^{n}}^{(q)}$, which also equals $(-1)^{n}q^{n(n+1)/2}F_{1^{n}}^{(q)}$, as well as $S(F_{1^{n}}^{(q)})=(-1)^{n}F_{n}^{(q)}$, which is the same, via Definition \ref{fundQSymq}, as the equality $S(M_{1^{n}}^{(q)})=(-1)^{n}E_{1^{n}}^{(q)}$ from Proposition \ref{SQSymqM}. Also here we get $S(M_{n}^{(q)})=-E_{n}^{(q)}=-E_{n}^{(q)}$ from the latter proposition, as $M_{n}^{(q)}$ is again primitive in $\mathtt{QSym}^{(q)}$. \label{SQSymqext}
\end{ex}

\begin{ex}
Taking $\alpha=213\vDash6$ as in our running example, Theorem \ref{SQSymqF} yields $S(F_{\alpha}^{(q)})=q^{12}\widetilde{F}_{1131}^{(q)}$, while Proposition \ref{SQSymqM} produces $S(M_{\alpha}^{(q)})=-E_{312}^{(q)}$. \label{antQSymqex}
\end{ex}
When we expand the expressions from Examples \ref{SQSymex} and \ref{antQSymqex} in the monomial basis for comparing them, with the same composition $\alpha$, the value $F_{1131}$ of $S(F_{\alpha})$ is $M_{1131}+M_{11211}+M_{11121}+M_{1^{6}}$, while $S(F_{\alpha}^{(q)})=q^{12}\widetilde{F}_{1131}^{(q)}$ is $q^{12}M_{1131}^{(q)})+q^{14}M_{11211}^{(q)})+q^{14}M_{11121}^{(q)})+q^{15}M_{1^{6}}^{(q)})$. Similarly, the expression $-E_{312}$ for $S(M_{\alpha})$ is $-M_{312}-M_{42}-M_{33}-M_{6}$, and $S(M_{\alpha}^{(q)})=-E_{312}^{(q)}$ decomposes as $-q^{11}M_{312}^{(q)}-q^{8}M_{42}^{(q)}-q^{9}M_{33}^{(q)}-M_{6}^{(q)}$.

\medskip

Note that since the transposition from Definition \ref{invdef} is an involution (and the signs remain the same), the formula from Theorem \ref{SQSymF} shows that the square of the antipode on $\mathtt{QSym}$ is the identity on that Hopf algebra. The same applies for via Theorem \ref{SqQSym}, to the one on $\mathtt{QSym}_{q}$. This is both in correspondence with the second assertion of Proposition \ref{Sprop}, as these Hopf algebras are commutative.

We now see that this is no longer the case for $\mathtt{QSym}^{(q)}$.
\begin{cor}
For generic $q$, $S^{2}$ has infinite order on $\mathtt{QSym}^{(q)}$. \label{deg2QSymq}
\end{cor}

\begin{proof}
We restrict attention to the part of $\mathtt{QSym}^{(q)}$ that is homogeneous of degree 2. We saw in Example \ref{SQSymqext} that $S$ inverts the sign of $M_{2}^{(q)}$, and takes the other other basis element $M_{11}^{(q)}=F_{11}^{(q)}$ to $\widetilde{F}_{2}^{(q)}=M_{2}^{(q)}+qM_{11}^{(q)}$. Hence the matrix of $S$ on this homogeneous part, as a free module over $R[q]$ with the monomial basis, is represented by the matrix $\binom{-1\ \ 1}{\ \ 0\ \ q}$. Thus $S^{2}$ is represented on this part by the square $\binom{1\ \ q+1}{0\ \ \ q^{2}\ }$ of this matrix, which has infinite order for generic $q$. Thus the assertion holds on all of $\mathtt{QSym}^{(q)}$, as desired. This proves the corollary.
\end{proof}
Indeed, the fact that $S$ is no longer an involution, as seen in Corollary \ref{deg2QSymq}, is possible because $\mathtt{QSym}^{(q)}$ is, in addition to being a $q$-Hopf algebra, neither commutative nor co-commutative. Note that the matrix from the proof of that corollary is not diagonalizable over $R[q]$, but only becomes such after localizing in $q+1$, and then $M_{11}^{(q)}+\frac{1}{q+1}M_{2}^{(q)}$ spans the eigenspace of $S$ that is associated with the eigenvalue $q$.

Here is the analogue of Remark \ref{q0com} for this $q$-deformation.
\begin{rmk}
Substituting $q=0$ in Corollary \ref{prodMqs} produces, via Remark \ref{invw0}, the equality $M_{\alpha}^{(0)}M_{\beta}^{(0)}=M_{\alpha\beta}^{(0)}+M_{\alpha\odot\beta}^{(0)}$. The only composition $\beta \vDash n$ for which $\operatorname{inv}(\beta)=0$ is the one of length 1, which satisfies $\beta\succeq\alpha$ for every $\alpha \vDash n$, so that $E_{\alpha}^{(0)}$ from Definition \ref{fundQSymq} is $M_{n}^{(0)}$ for every $\alpha \vDash n$. It thus follows from Proposition \ref{SQSymqM} that the antipode of $\mathtt{QSym}^{(0)}$ takes $M_{\alpha}^{(0)}$ for $\alpha \vDash n$ to $(-1)^{\ell(\alpha)}M_{n}^{(0)}$. In Theorem \ref{SQSymqF}, since $q^{\operatorname{inv}(\alpha^{t})}$ vanishes whenever $\ell(\alpha^{t})>1$ (namely $\alpha\neq1^{n}$) when $q=0$, we deduce that the antipode annihilates $F_{\alpha}^{(0)}$ for any $1^{n}\neq\alpha \vDash n$. For $F_{1^{n}}^{(0)}=M_{1^{n}}^{(0)}$, the antipode image is $(-1)^{n}\widetilde{F}_{n}^{(0)}$, namely $(-1)^{n)}\sum_{\beta \vDash n}q^{\operatorname{inv}(\beta)}M_{\beta}^{(q)}$ at $q=0$, which is indeed $(-1)^{\ell(1^{n})}M_{n}^{(0)}$. \label{q0nc}
\end{rmk}

\medskip

Recall from Remark \ref{SymQSym} that $\mathtt{QSym}$ contains the Hopf algebra $\mathtt{Sym}$, whose graded parts have, in degree 2 or higher, strictly smaller dimensions. The natural action of $S_{\mathbb{N}}$ on $R\ldbrack\mathbf{x}_{\infty}\rdbrack$ can no longer be defined on $R\langle\langle\mathbf{x}_{\infty}\rangle\rangle^{(q)}$, and the $q$-analogue of $\mathtt{Sym}$ is defined, in \cite{[EK]}, \cite{[L]}, and others, as follows.
\begin{rmk}
For any $n\geq0$ we set $e_{n}^{(q)}:=M_{1^{n}}^{(q)}=F_{1^{n}}^{(q)}$ as well as $h_{n}^{(q)}:=F_{n}^{(q)}$. Then $\Delta(e_{n}^{(q)})=\sum_{r=0}^{n}e_{r}^{(q)} \otimes e_{n-r}^{(q)}$ and $\Delta(h_{n}^{(q)})=\sum_{r=0}^{n}h_{r}^{(q)} \otimes h_{n-r}^{(q)}$ (see Example \ref{DeltaQSymqext}), and we define $\mathtt{Sym}^{(q)}$ to be the sub-algebra generated by $\{e_{n}^{(q)}\}_{n=1}^{\infty}$. It satisfies the quantum Newton relations $\sum_{r=0}^{n}(-1)^{r}q^{r(r-1)/2}e_{r}^{(q)}h_{n-r}^{(q)}=0$, as well as $\sum_{r=0}^{n}(-1)^{r}q^{r(r-1)/2}h_{n-r}^{(q)}e_{r}^{(q)}=0$ (which are not immediately equivalent by non-commutativity), from which it follows that $h_{n}^{(q)}\in\mathtt{Sym}^{(q)}$ for every $n\geq1$. \label{SymQSymq}
\end{rmk}
In fact, Theorem \ref{SQSymqF} shows, in the terminology of Remark \ref{SymQSymq}, that $S(h_{n}^{(q)})=(-1)^{n}q^{n(n-1)/2}e_{n}^{(q)}$ for every $n\geq0$. Thus the quantum Newton relations from that remark are related to Corollary \ref{detS} and Remark \ref{firstS}. The substitution $q=0$ from Remark \ref{q0nc} reduces them to $h_{n}^{(0)}=h_{n-1}^{(0)}e_{1}^{(0)}=e_{1}^{(0)}h_{n-1}^{(0)}$, so that $h_{n}^{(0)}$ is just $(e_{1}^{(0)})^{n}$.

Note that $S(e_{n}^{(q)})$ is our $(-1)^{n}\widetilde{F}_{n}^{(q)}$ by that theorem, which is not $h_{n}^{(q)}$, and Corollary \ref{deg2QSymq} shows that also on $\mathtt{Sym}^{(q)}$ the endomorphism $S^{2}$ has infinite order (since the degree 2 part of $\mathtt{Sym}^{(q)}$ is the same as that of $\mathtt{QSym}^{(q)}$). In fact, for general values of $q$, the rank of the homogeneous parts of $\mathtt{Sym}^{(q)}$ may be larger than those of $\mathtt{Sym}$---for example, in degree 3, $\mathtt{Sym}^{(q)}$ contains $e_{3}^{(q)}$, $h_{3}^{(q)}$, and also $(1-q^{2})M_{21}^{(q)}$ and $(1-q^{2})M_{21}^{(q)}$, so that it has degree 4 like that of $\mathtt{QSym}^{(q)}$. It may be interesting to check whether $\mathtt{Sym}^{(q)}$ coincides with all of $\mathtt{QSym}^{(q)}$ after localizing by enough polynomials in $q$.

\section{Word Quasi-Symmetric Functions \label{NCQSym}}

We now turn to describing the details of the larger Hopf algebra $\mathtt{NCQSym}$ of non-commutative quasi-symmetric functions, also known as the Hopf algebra $\mathtt{WQSym}$ of word quasi-symmetric functions. Some of our notions are derived from \cite{[BZ]}, but we will expands them and present them in a way that is adapted for stating and proving our result about the antipode there.

\subsection{Set Compositions and Set Partitions}

The natural basis for $\mathtt{NCQSym}$ is parameterized by set compositions. We will also mention shortly, in Remark \ref{NCSym} below, the algebra $\mathtt{NCSym}$, whose basis is parameterized by set partitions. We thus turn to introducing the required notation involving these objects, as is carried out in \cite{[BZ]}, \cite{[BHRZ]}, \cite{[BRRZ]}, \cite{[LM]}, and \cite{[D]}.
\begin{defn}
Let $X$ be a finite set of size $d$.
\begin{enumerate}[$(i)$]
\item A \emph{set composition} of $X$ is an ordered collection $\mathbf{A}:=(A_{1},\ldots,A_{\ell})$ of non-empty pairwise disjoint sets whose union is $X$. We write this as $\mathbf{A} \vDash X$ and $\ell(\mathbf{A})=\ell$
\item For another set $Y$ that is disjoint from $X$, if $\mathbf{A} \vDash X$ and $\mathbf{B} \vDash Y$ with $\ell(\mathbf{A})=\ell$ and $\ell(\mathbf{B})=k$ then their \emph{concatenation}, written as $\mathbf{A}|\mathbf{B}$, is the set composition $(A_{1},\ldots,A_{\ell},B_{1},\ldots,B_{k})$ of $X \cup Y$.
\item A \emph{set partition} of $X$ is an unordered collection $\mathbf{L}:=\{A_{1},\ldots,A_{\ell}\}$ as in part $(i)$, which we write as $\mathbf{L} \vdash X$ and $\ell(\mathbf{L})=\ell$
\item If $X$ and $Y$ are disjoint, then the \emph{union} of set partitions $\mathbf{L} \vdash X$ and $\mathbf{M} \vdash Y$ is their set-theoretic union $\mathbf{L}\cup\mathbf{M}$ as a set partition of $X \cup Y$.
\end{enumerate} \label{setcomppar}
\end{defn}

\begin{rmk}
The number of set partitions of length $\ell$ of a set $X$ of size $d$ is, by definition, the \emph{Stirling number of the second kind}, denoted by $S(d,\ell)$. Every such set partition is the free orbit under $S_{\ell}$ of precisely $\ell!$ set compositions of $X$, so that the number of set compositions of length $\ell$ of $X$ is $\ell!S(d,\ell)$. The total number $B_{d}:=\sum_{\ell=1}^{d}S(d,\ell)$ of set partitions of $X$ is known as the \emph{Bell number}. The number of set compositions of $X$ is thus $\sum_{\ell=1}^{d}\ell!S(d,\ell)$, sometimes known as the \emph{ordered Bell number}, or the \emph{Fubini number}, that is associated with $d$. \label{StirBell}
\end{rmk}
When $X$ is totally ordered, some authors define set partitions as a special case of set compositions (similarly to Definition \ref{defcomp}) by imposing the extra condition that the minimal elements of the sets $A_{j}$, $1 \leq j\leq\ell$ are in increasing order. We also recall from Remark \ref{SymQSym} that they can also be viewed as (representatives of) orbits of the latter under the action $S_{\ell}$ (where $\ell$ is the length), which corresponds to forgetting the order of the entries of the composition. As for a general, unordered set $X$ there is no canonical way to choose the representative, Definition \ref{setcomppar} uses the analogue of Remark \ref{SymQSym} here. Indeed, some authors (like \cite{[BZ]}) also refer to set compositions as \emph{ordered set partition}.

\medskip

While Definition \ref{setcomppar} considers set partitions and set compositions of arbitrary finite sets, we will be mostly interested in totally ordered ones.
\begin{defn}
Assume that the set $X$, of size $d$, is totally ordered.
\begin{enumerate}[$(i)$]
\item We denote by $\iota_{X}$ the unique order-preserving bijection from $X$ onto $\mathbb{N}_{d}$.
\item For any set composition $\mathbf{A} \vDash X$, we define the \emph{standardization} $\operatorname{st}_{X}\mathbf{A}$ to be the set composition $(\iota_{X}A_{1},\ldots,\iota_{X}A_{\ell})\vDash\mathbb{N}_{d}$.
\item The \emph{standardization} $\operatorname{st}_{X}\mathbf{L}\vdash\mathbb{N}_{d}$ of a set partition $\mathbf{L} \vdash X$ is defined in the same manner as in part $(ii)$, using $\iota_{X}$ from part $(i)$.
\item For a subset $Y\subseteq\mathbb{N}$ and $n\in\mathbb{N}$, write $Y+n$ for $\{y+n\;|\;y \in Y\}\subseteq\mathbb{N}$.
\item Assume that $X\subseteq\mathbb{N}$, and take a set composition $\mathbf{A}=(A_{1},\ldots,A_{\ell}) \vDash X$. Then the set composition $(A_{1}+n,\ldots,A_{\ell}+n) \vDash X+n$ will be denoted by $\mathbf{A}[+n]$. The set partition $\mathbf{L}[+n] \vdash X+n$ arising from a set partition $\mathbf{L} \vdash X$ is defined in the same manner.
\item Assume that $\mathbf{A}\vDash\mathbb{N}_{n}$ and $\mathbf{B}\vDash\mathbb{N}_{m}$ are (already standardized) set compositions. We then write $\mathbf{A}|\mathbf{B}$ for the concatenation $\mathbf{A}|\mathbf{B}[+n]\vDash\mathbb{N}_{n+m}$.
\item When $\mathbf{L}\vDash\mathbb{N}_{n}$ and $\mathbf{M}\vDash\mathbb{N}_{n}$, we write $\mathbf{L}|\mathbf{M}$ for the union $\mathbf{L}\cup(\mathbf{M}[+n])$ as a (standardized) set partition of $\mathbb{N}_{n+m}$.
\end{enumerate} \label{ordset}
\end{defn}
The form for the concatenation of standardized set compositions in Definition \ref{ordset}, while a slight abuse of notation (because the sets are no longer disjoint), corresponds to the idea that when $X$ and $Y$ are well-ordered, and we consider every element of $Y$ to be larger (in that order) from all elements of $X$, then $\operatorname{st}_{X \cup Y}(\mathbf{A}|\mathbf{B})$, for $\mathbf{A} \vDash X$ and $\mathbf{B} \vDash Y$, would be $\operatorname{st}_{X}\mathbf{A}|\operatorname{st}_{Y}\mathbf{B}$. The fact that order is not symmetric between $X$ and $Y$ is the reason for the notation also for set partitions, where we again have $\operatorname{st}_{X \cup Y}(\mathbf{L}\cup\mathbf{M})=\operatorname{st}_{X}\mathbf{L}|\operatorname{st}_{Y}\mathbf{M}$.

\begin{rmk}
It is clear that for a given set $X$, the standardization $\operatorname{st}_{X}$ from Definition \ref{setcomppar} is a bijection from set compositions of $X$ onto set compositions of $\mathbb{N}_{d}$. A set composition of $\mathbb{N}_{d}$ is thus the image under $\operatorname{st}_{X}$ of set compositions of $X$ for many ordered sets $X$. Thus, if we wish to invert the standardization, we need to state which set $X$ are we working with, which is the reason for our notation $\operatorname{st}_{X}$ (unlike authors who simply write $\operatorname{st}$, without the subscript). \label{stdbij}
\end{rmk}
As an example of Remark \ref{stdbij}, the map $\mathbf{A}\mapsto\mathbf{A}[+n]$ on (standardized) set compositions $\mathbf{A}\vDash\mathbb{N}_{d}$ is the inverse $\operatorname{st}_{X}^{-1}$ of the standardization map associated with the totally ordered set $X:=\mathbb{N}_{n+d}\setminus\mathbb{N}_{n}$, of size $d$.

It is clear that the empty set admits a (unique) set composition, which is its unique set partition as well, of length 0. Since we will work mainly with standardized set compositions and partitions, we will allow an abuse of notation and write it as $\emptyset\vDash\mathbb{N}_{0}$ as well as $\emptyset\vdash\mathbb{N}_{0}$.

\medskip

Recall that for the quasi-shuffle rule from Lemma \ref{monQSym} (or its $q$-analogue appearing in Lemma \ref{monQSymq}), we considered the map $l_{a}$, which is based on replacing any composition of an integer $d$ by the composition of $a+d$ obtained by adding $a$ as the initial entry. We would like to do the same for standardized set compositions. This is done as follows.
\begin{defn}
Let $A$ be a subset of $\mathbb{N}$, with $|A|=a$ and maximal element $m$.
\begin{enumerate}[$(i)$]
\item For any $d \geq m-a$, we have $A\subseteq\mathbb{N}_{a+d}$, and write $X$ for the set $\mathbb{N}_{a+d} \setminus A$, of size $d$.
\item Given any (standardized) set composition $\mathbf{C}\vDash\mathbb{N}_{d}$ for such $d$, consider the set composition $\operatorname{st}_{X}^{-1}\mathbf{C} \vDash X$ via Remark \ref{stdbij}. We denote by $l_{A}\mathbf{C}$ the composition of $\mathbb{N}_{n+a}$ that is obtained by adding $A$ to $\operatorname{st}_{X}^{-1}\mathbf{C}$ as the first set in the set composition.
\item Assume that $f$ is a sum $\sum_{\mathbf{A}}c_{\mathbf{C}}f_{\mathbf{C}}$ of some module spanned freely by a basis parameterized by set compositions $\mathbf{C}$, where each $\mathbf{C}$ that appears in $f$ is a set composition of $\mathbb{N}_{d}$ for some $d \geq m-a$. We then write $l_{A}f$ for the expression $\sum_{\mathbf{A}}c_{\mathbf{A}}f_{l_{A}\mathbf{C}}$.
\end{enumerate} \label{insset}
\end{defn}
One may also consider an analogue of Definition \ref{insset} for standardized set partitions, but we will not use such an analogue.

It is clear that any non-empty set composition $\mathbf{A}=(A_{1},\ldots,A_{\ell})\vDash\mathbb{N}_{d}$ (thus with $d>0$ and length $\ell>0$) can be written uniquely as $l_{A}\mathbf{C}$ via Definition \ref{insset}, where $A$ is the first set $A_{1}$ of $\mathbf{A}$, and $\mathbf{C}$ is $\operatorname{st}_{\mathbb{N}_{d} \setminus A_{1}}(A_{2},\ldots,A_{\ell})\vDash\mathbb{N}_{d-|A_{1}|}$. We will denote that $\mathbf{C}$ by $\overline{\mathbf{A}}$.
\begin{ex}
It is clear from Definition \ref{insset} that when $A$ is empty, $l_{A}$ is the identity map. In the other extreme case, if $\mathbf{C}=\emptyset\vDash\mathbb{N}_{0}$ then $A$ must be $\mathbb{N}_{a}$, and then $l_{A}\mathbf{C}$ is the set composition of length 1 of $\mathbb{N}_{a}$. For a non-trivial example, take $A:=\{2,4,7\}$, with $a=3$ and $m=7$, and we consider, for $d=5\geq4=m-a$, the set composition $\mathbf{C}:=(23,5,14)\vDash\mathbb{N}_{5}$. Then with $a+d=8$ the set $X:=\mathbb{N}_{8} \setminus A$ is $\{1,3,5,6,8\}$ seen above, the set composition $\operatorname{st}_{X}^{-1}\mathbf{C}$ is $(35,8,16)$ from above, and thus $\mathbf{A}:=l_{A}\mathbf{C}$ is $(247,35,8,16)\vDash\mathbb{N}_{8}$. We also note that the only way to get $\mathbf{A}$ as $l_{\tilde{A}}\tilde{\mathbf{C}}$ is with $\tilde{A}$ being the first set $\{2,4,7\}=A$, and then $\tilde{\mathbf{C}}$ has to be standard and related to $(35,8,16)\vDash\{1,3,5,6,8\}$, and thus $\tilde{\mathbf{C}}=\mathbf{C}$. \label{exlA}
\end{ex}
In Example \ref{exlA} we have introduced the simplified notations for set compositions in which we omit the curly brackets of sets, where we wrote $(35,8,16)$ for the set composition $\big(\{3,5\},\{8\},\{1,6\}\big)$ of $\{1,3,5,6,8\}$. Inside superscripts, which will contain standardized set compositions, we may remove the external parentheses as well, and write expressions like $\mathrm{M}_{23,5,14}$ instead of $\mathrm{M}_{(23,5,14)}$ for the standardization of this set composition.

In order to relate our set compositions (and we do the same for set partitions) to some objects from above, we make the following definition.
\begin{defn}
Let $X$ be a set of size $d$.
\begin{enumerate}[$(i)$]
\item For $\mathbf{A} \vDash X$ we define the composition $\operatorname{sz}(\mathbf{A}) \vDash d$ which has length $\ell(\mathbf{A})$ and whose $j$th entry is the number $|A_{j}|$.
\item Assume that $X$ is totally ordered. We then define the \emph{inversion number} $\operatorname{inv}(\mathbf{A})$ to be number of pairs $(x,y)$ from $X$ where $x<y$ by the index of the set from $\mathbf{A}$ that contains $x$ is larger than that of the one containing $y$.
\item Similarly to part $(i)$, wherever $\mathbf{L} \vdash X$, the partition $\operatorname{sz}(\mathbf{L}) \vdash d$ is obtained by ordering the numbers $|A_{j}|$, $1 \leq j\leq\ell(\mathbf{L})$ and putting them in non-increasing order.
\end{enumerate} \label{sizes}
\end{defn}
The reason for the notation from Definition \ref{sizes} is that this construction replaces each set by its size. The inversion number will be used in the relation with $\mathtt{QSym}^{(q)}$ in Proposition \ref{projQSymq} below.

\begin{rmk}
The map from Definition \ref{sizes} commutes with $l_{A}$ from Definition \ref{insset} and $l_{a}$ on compositions, where $a=|A|$, in the sense that we have the equality $\operatorname{sz}(l_{A}\mathbf{C})=l_{a}\operatorname{sz}(\mathbf{C})$ for any $\mathbf{C}\vDash\mathbb{N}_{d}$ when $d \geq m-a$ (when $m=\max A$). Using those from Example \ref{exlA}, we have $\alpha:=\operatorname{sz}(\mathbf{C})=212\vDash5$, $a=3$, and indeed the image under $\operatorname{sz}$ of the composition $l_{A}\mathbf{C}\vDash\mathbb{N}_{8}$ there is $3212=l_{3}\alpha\vDash8$. \label{szlrel}
\end{rmk}

\medskip

We will need a partial analogue of Definition \ref{cutatr} for set compositions.
\begin{defn}
Take a set composition $\mathbf{A}=(A_{1},\ldots,A_{\ell})\vDash\mathbb{N}_{n}$, with the composition $\alpha:=\operatorname{sz}(\mathbf{A}) \vDash n$. We also set $\tilde{T}$ to be $\{0\}$ when $n=0$ and to be $T\cup\{0,n\}$ in case $n>0$ and $T:=\operatorname{comp}_{n}^{-1}\alpha\subseteq\mathbb{N}_{n-1}$ via Proposition \ref{compn}, and write its elements as $\{t_{i}\}_{i=0}^{\ell}$ in increasing order as in that proposition.
\begin{enumerate}[$(i)$]
\item Take $r\in\tilde{T}$, so that $r=t_{k}$ for some $0 \leq k\leq\ell$. Set $X_{|r}:=\bigcup_{j=1}^{k}A_{j}$ and $X_{r|}:=\bigcup_{j=k+1}^{\ell}A_{j}$, and we define $\mathbf{A}_{|r}$ to be $\operatorname{st}_{X_{|r}}(A_{1},\ldots,A_{k})\vDash\mathbb{N}_{r}$ and $\mathbf{A}_{r|}:=\operatorname{st}_{X_{r|}}(A_{k+1},\ldots,A_{\ell})\vDash\mathbb{N}_{n-r}$.
\item If $n>0$ and $r\not\in\tilde{T}$, let $1 \leq k\leq\ell$ be such that $t_{k-1}<r<t_{k}$. We define $\tilde{\mathbf{A}}$ to be the set composition in which we replace $A_{k}$ by two sets, the first one containing the set of $r-t_{k-1}$ smallest elements of $A_{k}$, and the second one with the larger $t_{k}-r$ elements there. We then set $\mathbf{A}_{|r}:=\tilde{\mathbf{A}}_{|r}$ and $\mathbf{A}_{r|}:=\tilde{\mathbf{A}}_{r|}$.
\item Given a permutation $\rho \in S_{n}$, let $\mathbf{B}_{\rho}$ for $\rho \in S_{n}$ be the set composition $\big(\rho(1),\ldots,\rho(n)\big)$ of $\mathbb{N}_{n}$ into singletons in this order. This defines a bijection between $S_{n}$ and set compositions $\mathbf{A}\vDash\mathbb{N}_{n}$ with $\operatorname{sz}(\mathbf{A})=1^{n}$.
\item For any $0 \leq r \leq n$, we write $\rho_{|r} \in S_{r}$ and $\rho_{r|} \in S_{n-r}$ for the permutations for which $\mathbf{B}_{\rho_{|r}}=(\mathbf{B}_{\rho})_{|r}$ and $\mathbf{B}_{\rho_{r|}}=(\mathbf{B}_{\rho})_{r|}$ via parts $(i)$ and $(iii)$.
\end{enumerate} \label{setwithr}
\end{defn}
It is clear that $\mathbf{A}_{|0}=\mathbf{A}_{n|}=\emptyset$ as well as $\mathbf{A}_{|n}=\mathbf{A}_{0|}=\mathbf{A}$ wherever $\mathbf{A}\vDash\mathbb{N}_{n}$ in Definition \ref{setwithr}, as well as $\rho_{|0}$ and $\rho_{n|}$ being the trivial element of $S_{0}$ and $\rho_{|n}=\rho_{0|}=\rho$ for any $\rho \in S_{n}$. Here are some non-trivial values.
\begin{ex}
Let $\mathbf{A}$ be $(23,5,14)\vDash\mathbb{N}_{5}$, namely $\mathbf{C}$ from Example \ref{exlA}. Then $\mathbf{A}_{|2}=(12)$, $\mathbf{A}_{|3}=(12,3)$, $\mathbf{A}_{3|}=(12)$, and $\mathbf{A}_{2|}=(3,12)$ via part $(i)$ of Definition \ref{setwithr}. Part $(ii)$ there yields $\mathbf{A}_{|1}=\mathbf{A}_{4|}=(1)$ as well as $\mathbf{A}_{|4}=(23,4,1)$ and $\mathbf{A}_{1|}=(2,4,13)$. If $\rho \in S_{5}$ is 23514 in one-line notation, then $\rho_{|1}=\rho_{4|}=1 \in S_{1}$, $\rho_{|2}=\rho_{3|}=12 \in S_{2}$ (both are the identity elements there), $\rho_{|3}$ is the identity in $S_{3}$ as well, $\rho_{2|}=312$ there and $\rho_{|4}=2341$, and $\rho_{1|}=2413$ inside $S_{4}$. \label{sepsetcomp}
\end{ex}
The fact that $\mathbf{A}_{|1}=\mathbf{A}_{n-1|}=(1)\vDash1$ for any $\mathbf{A}\vDash\mathbb{N}_{n}$ with $n\geq1$, as seen in Example \ref{sepsetcomp}, is clear since $\mathbb{N}_{1}$ has a unique set composition (and set partition). In a similar manner, $\rho_{|1}=\rho_{n-1|}$ is the unique trivial element of $S_{1}$ whenever $\rho \in S_{n}$ and $n\geq1$ (also seen in that example).

\medskip

We will later be using the following two orders on set compositions on $\mathbb{N}_{n}$, defined in Sections 5 and 6 of \cite{[BZ]} respectively.
\begin{defn}
Let $\mathbf{A}$ and $\mathbf{B}$ be two set compositions of $\mathbb{N}_{n}$.
\begin{enumerate}[$(i)$]
\item We say that $\mathbf{B}$ \emph{refines} $\mathbf{A}$, written as $\mathbf{B}\preceq\mathbf{A}$, if each set in $\mathbf{A}$ is the union of a sets from $\mathbf{B}$ that appear consecutively. We let $\mathbf{B}\prec\mathbf{A}$ denote as usual, the combination of $\mathbf{B}\preceq\mathbf{A}$ and $\mathbf{B}\neq\mathbf{A}$.
\item We will write $\mathbf{B}\preceq_{*}\mathbf{A}$ in case $\mathbf{B}\preceq\mathbf{A}$ and whenever the union of two consecutive sets from $\mathbf{B}$, say $B_{j}$ and $B_{j+1}$, is contained in a set from $\mathbf{A}$, all the elements of $B_{i}$ are smaller than those of $B_{i+1}$. The notation $\mathbf{B}\prec_{*}\mathbf{A}$ has the resulting meaning.
\end{enumerate} \label{setord}
\end{defn}
In fact, the set compositions $\tilde{\mathbf{A}}$ and $\mathbf{A}$ in Definition \ref{setwithr} satisfy $\tilde{\mathbf{A}}\prec_{*}\mathbf{A}$, and any covering in that order is obtained in this way. This is in correspondence with the covering in the refinement order on compositions from Definition \ref{refine}.

The two orders from Definition \ref{setord} have the following properties, with respect to refinements of compositions via Definition \ref{refine} and the operation $\operatorname{sz}$ from Definition \ref{sizes}.
\begin{lem}
Take two set compositions $\mathbf{A}$ and $\mathbf{B}$ of $\mathbb{N}_{n}$, and let $\alpha:=\operatorname{sz}(\mathbf{A})$ and $\beta=\operatorname{sz}(\mathbf{B})$ be the corresponding compositions of $n$.
\begin{enumerate}[$(i)$]
\item If $\mathbf{B}\preceq\mathbf{A}$, and in particular when $\mathbf{B}\preceq_{*}\mathbf{A}$, then $\beta\preceq\alpha$.
\item When $\mathbf{B}$, $\beta=\operatorname{sz}(\mathbf{B})$ and $\alpha \vDash n$ with $\beta\preceq\alpha$ are given, then there exists a unique set composition $\mathbf{A}\vDash\mathbb{N}_{n}$ for which $\mathbf{B}\preceq\mathbf{A}$ and $\operatorname{sz}(\mathbf{A})=\alpha$.
\item For every $\rho \in S_{n}$ and $\alpha \vDash n$ there exists a unique set composition $\mathbf{A}_{\alpha}^{\rho}\vDash\mathbb{N}_{n}$ with $\operatorname{sz}(\mathbf{A}_{\alpha}^{\rho})=\alpha$ and $\mathbf{B}_{\rho}$ from
    Definition \ref{setwithr} satisfies $\mathbf{B}_{\rho}\preceq\mathbf{A}_{\alpha}^{\rho}$.
\item If $\beta\preceq\alpha$ and $\rho \in S_{n}$ are given, then $\mathbf{A}_{\beta}^{\rho}\preceq\mathbf{A}_{\alpha}^{\rho}$.
\item In case we are given $\mathbf{A}$, $\alpha=\operatorname{sz}(\mathbf{A})$, and some $\beta \vDash n$ satisfying $\beta\preceq\alpha$, the is a unique $\mathbf{B}\vDash\mathbb{N}_{n}$ that satisfies both $\operatorname{sz}(\mathbf{B})=\beta$ and $\mathbf{B}\preceq_{*}\mathbf{A}$.
\item For every $\mathbf{A}\vDash\mathbb{N}_{n}$ there is a unique $\rho(\mathbf{A}) \in S_{n}$ such that $\mathbf{B}_{\rho(\mathbf{A})}\preceq_{*}\mathbf{A}$.
\end{enumerate} \label{szord}
\end{lem}

\begin{proof}
Part $(i)$ is clear from the definitions. For part $(ii)$ we observe that the sets in $\mathbf{A}$ are determined by taking the unions of the sets from $\mathbf{B}$ according to $\alpha$. Part $(iii)$ is the special case of part $(ii)$ where $\beta=1^{n}$ (and $1^{n}\preceq\alpha$). For part $(iv)$ we apply part $(ii)$ with $\mathbf{B}=\mathbf{A}_{\beta}^{\rho}$, and as the resulting set composition $\mathbf{A}$ satisfies $\mathbf{B}_{\rho}\preceq\mathbf{A}_{\beta}^{\rho}\preceq\mathbf{A}$, it must be $\mathbf{A}_{\alpha}^{\rho}$ by the uniqueness in part $(iii)$.

In part $(v)$, we need to know how to decompose the set $A_{j}$ for $1 \leq j\leq\ell(\mathbf{A})$ into sets of sizes as determined by $\beta$. But part $(ii)$ of Definition \ref{setord} implies that the only way to do so is to take the smallest elements of $A_{j}$ to the first set in that location of $\mathbf{B}$, the next smallest one to the next set, and so forth. This shows that $\mathbf{B}$ exists, and determines it uniquely. Part $(vi)$ is again the case $\beta=1^{n}\preceq\alpha$ of part $(iv)$. This proves the lemma.
\end{proof}

\begin{ex}
When $\ell(\alpha)=1$, namely $\alpha=n \vDash n$, there is a unique set composition $(\mathbb{N}_{n})$ whose $\operatorname{sz}$-image is $\alpha$, and as it is refined by $\mathbf{B}_{\rho}$ for every $\rho \in S_{n}$, we get $\mathbf{A}_{n}^{\rho}=(\mathbb{N}_{n})$ for every $\rho$ with this $\alpha$. In the other extreme, if $\alpha=1^{n}$ then we get $\mathbf{A}_{1^{n}}^{\rho}=\mathbf{B}_{\rho}$ for every $\rho \in S_{n}$, which we henceforth write in one-line notation. When $n=1$ there is just one choice of $\alpha$, of $\rho$, and of $\mathbf{A}$, and the same for $n=0$, where all are empty. For $n=2$, where $S_{2}$ has two elements, we have the two compositions 2 and 11, for which the expressions for $\mathbf{A}_{\alpha}^{\rho}$ are already covered. When $n=3$, the two compositions 21 and 12 not yet covered yield $\mathbf{A}_{21}^{123}=\mathbf{A}_{21}^{213}=(12,3)$, $\mathbf{A}_{21}^{132}=\mathbf{A}_{21}^{312}=(13,2)$, $\mathbf{A}_{21}^{231}=\mathbf{A}_{21}^{321}=(23,1)$, $\mathbf{A}_{12}^{123}=\mathbf{A}_{21}^{132}=(1,23)$, $\mathbf{A}_{12}^{213}=\mathbf{A}_{12}^{231}=(2,13)$, and $\mathbf{A}_{12}^{312}=\mathbf{A}_{12}^{321}=(3,12)$. \label{Aalpharhoex}
\end{ex}
It is clear from Lemma \ref{szord} that every $\mathbf{A}\vDash\mathbb{N}_{n}$ is $\mathbf{A}_{\alpha}^{\rho}$ for some $\alpha$ and $\rho$. Indeed, $\alpha$ is determined as $\operatorname{sz}(\mathbf{A})$, and for $\rho$ we can take $\rho(\mathbf{A})$, and then $\mathbf{B}_{\rho}\preceq_{*}\mathbf{A}$ and thus $\mathbf{B}_{\rho}\preceq\mathbf{A}$ as desired. As the case $n=0$ is trivial (see Example \ref{Aalpharhoex}), we may assume that $n\geq1$ and get the following consequence.
\begin{cor}
Assume that the composition $\alpha \vDash n\geq1$ is $\operatorname{comp}_{n}T$ for $T\subseteq\mathbb{N}_{n-1}$, and that it is $\operatorname{sz}(\mathbf{A})$ for $\mathbf{A}\vDash\mathbb{N}_{n}$. We express $\mathbf{A}$ as $\mathbf{A}_{\alpha}^{\rho}$ for some $\rho \in S_{n}$, so that $\mathbf{B}_{\rho}\preceq\mathbf{A}$.
\begin{enumerate}[$(i)$]
\item We have $\rho=\rho(\mathbf{A})$, or equivalently $\mathbf{B}_{\rho}\preceq_{*}\mathbf{A}$, if and only if $T$ contains the descent set of $\rho$.
\item For a composition $\mathbf{B}\vDash\mathbb{N}_{n}$ we have $\mathbf{B}\preceq_{*}\mathbf{A}$ if and only if $\mathbf{B}=\mathbf{A}_{\beta}^{\rho}$ for $\rho=\rho(\mathbf{A})$ and $\beta\preceq\alpha$.
\end{enumerate} \label{weakord}
\end{cor}

\begin{proof}
The fact that $\mathbf{A}=\mathbf{A}_{\alpha}^{\rho}$ is equivalent, via Definition \ref{setord} and Lemma \ref{szord}, to the statement that every set in $\mathbf{A}$ contains elements that appear as consecutive entries of the one-line notation of $\rho$. We get $\mathbf{B}_{\rho}\preceq_{*}\mathbf{A}$ if and only if the elements in every such set appeared in increasing order. Since $T$ is the set of separating locations between these sets (via Proposition \ref{compn}), the latter property is equivalent to the complement of $T$ in $\mathbb{N}_{n-1}$, which consists of the locations that are no longer separating when we constructed $\mathbf{A}$ (or $\alpha$) contains only ascents of $\rho$. As this happens if and only if all the descents of $\rho$ are contained in $T$, part $(i)$ follows.

Now, part $(i)$ of Lemma \ref{szord} implies that if $\mathbf{B}\preceq_{*}\mathbf{A}$ and $\beta=\operatorname{sz}(\mathbf{B})$ then $\beta\preceq\alpha$, and part $(v)$ there implies that for every such $\beta$ there is a unique $\mathbf{B}$ with $\operatorname{sz}(\mathbf{B})=\beta$ and $\mathbf{B}\preceq_{*}\mathbf{A}$. So fix such $\beta$, let $\mathbf{B}$ be the corresponding set composition, and set $\sigma:=\rho(\mathbf{B})$, which determines $\mathbf{B}$ as $\mathbf{A}_{\beta}^{\sigma}$. But then $\mathbf{B}_{\sigma}\preceq_{*}\mathbf{B}\preceq_{*}\mathbf{A}$, which implies that $\sigma=\rho(\mathbf{A})=\rho$ via part $(vi)$ of Lemma \ref{szord}, and part $(ii)$ is also established. This completes the proof of the corollary.
\end{proof}

\begin{ex}
We have $\rho(\mathbf{B}_{\rho})=\rho$ for every $\rho \in S_{n}$, and $\rho(\mathbb{N}_{n})$ is $\operatorname{Id} \in S_{n}$, which we henceforth write as $\operatorname{Id}_{n}$. While we have $\mathbf{A}\preceq(\mathbb{N}_{n})$ for every $\mathbf{A}\vDash\mathbb{N}_{n}$, the relation $\mathbf{A}\preceq_{*}(\mathbb{N}_{n})$ holds, via Corollary \ref{weakord}, if and only if $\rho(\mathbf{A})=\operatorname{Id}_{n}$. In the remaining expressions appearing in Example \ref{Aalpharhoex} when $n=3$, each set composition of $\mathbb{N}_{3}$ appears as $\mathbf{A}_{\alpha}^{\rho}$ for two choices of $\rho$. In each one of them, the first one is with $\rho=\rho(\mathbf{A})$. This $\rho$ is $\operatorname{Id}_{3}=123$ for two of these set compositions of $\mathbb{N}_{3}$, in correspondence with the large part in Figure 3 of \cite{[BZ]}, and we have $\mathbf{B}_{321}$ as an isolated element in that figure. The remaining four expressions with $\rho=\rho(\mathbf{A})$ are the other maximal elements in that figure, each of which lie over its corresponding $\mathbf{B}_{\rho}$. \label{rhoAex}
\end{ex}

\begin{ex}
If $\mathbf{A}$ is $(23,5,14)\vDash\mathbb{N}_{5}$ from Example \ref{sepsetcomp}, then those $\mathbf{B}\vDash\mathbb{N}_{5}$ that satisfy $\mathbf{B}\preceq_{*}\mathbf{A}$ are $\mathbf{A}$ itself, $(2,3,5,14)$, $(23,5,1,4)$, and $\mathbf{B}_{23514}$. The set compositions that satisfy $\mathbf{B}\preceq\mathbf{A}$ but are not related via the weaker order are $(3,2,5,14)$, $(23,5,4,1)$, and the set compositions $\mathbf{B}_{32514}$, $\mathbf{B}_{23541}$, and $\mathbf{B}_{32541}$. Since $\alpha:=\operatorname{sz}(\mathbf{A})=212\vDash5$, our $\mathbf{A}$ can be written as $\mathbf{A}_{\alpha}^{\rho}$ where $\rho$ is one of the four elements of $S_{5}$ appearing as an index of $\mathbf{B}_{\rho}$ here, but only the first one, whose one-line notation is 23514 (namely $\rho$ from Example \ref{sepsetcomp}), is $\rho(\mathbf{A})$. \label{refAex}
\end{ex}

\begin{rmk}
In \cite{[BZ]} there are orders $\preceq$ and $\preceq_{*}$ also on set partitions of $\mathbb{N}_{n}$. As a set partition $\mathbf{L}\vdash\mathbb{N}_{n}$ of length $\ell$ can be seen as $S_{\ell}$-orbit of set compositions of $\mathbb{N}_{n}$, one can define them via Definition \ref{setord} by saying that $\mathbf{L}\preceq\mathbf{M}$ (resp. $\mathbf{L}\preceq_{*}\mathbf{M}$) if and only if there exist $\mathbf{A}\vDash\mathbb{N}_{n}$ whose orbit is $\mathbf{L}$ and $\mathbf{B}\vDash\mathbb{N}_{n}$ that maps to $\mathbf{M}$ such that $\mathbf{A}\preceq\mathbf{B}$ (resp. $\mathbf{A}\preceq_{*}\mathbf{B}$). \label{partord}
\end{rmk}
As an example for the orbits in Remark \ref{partord}, we note that the set partition $A(\alpha)$ from Equation (15) of \cite{[BRRZ]} corresponds to the orbit of our $\mathbf{A}_{\alpha}^{\operatorname{Id}_{n}}$.

\subsection{Non-Commutative Quasi-Symmetric Functions}

We have already seen the ring $R\langle\langle\mathbf{x}_{\infty}\rangle\rangle$. It admits two actions of $S_{\mathbb{N}}$, that reduce in the commutative quotient to those whose invariants are $\mathtt{Sym}$ and $\mathtt{QSym}$ respectively. The simpler one, which is via ring automorphisms, will be mentioned shortly later. We are interested in the one that is in the spirit of \cite{[Hi2]}, as is mentioned in \cite{[Hi1]}.

Since when there is no restriction on the variables in $R\langle\langle\mathbf{x}_{\infty}\rangle\rangle$, using exponents for monomials is no longer an adequate way for writing them (as the monomial $x_{1}x_{2}x_{1}$ shows). We thus make the following definition.
\begin{defn}
Recall the action $\rho:I \mapsto \rho I$ of $\rho \in S_{\mathbb{N}}$ on finite subsets $I$ of a fixed size $\ell$ of $\mathbb{N}$
\begin{enumerate}[$(i)$]
\item Let $\eta$ be a finite sequence of integers, of some length $d$. Then $x^{\eta} \in R\langle\langle\mathbf{x}_{\infty}\rangle\rangle$ is the monomial $\prod_{j=1}^{d}x_{\eta_{j}}$, multiplied in the order they appear in $\eta$.
\item Let $I$ be the set of numbers appearing in $\eta$, of size $\ell$, written in increasing order, so that for every $1 \leq k \leq d$ there exists $j=j(k)$ such that $\eta_{k}$ is the $j$th element of $I$. Take $\rho \in S_{\mathbb{N}}$, and write $\rho I$ also in increasing order. We then define $\rho\eta$ to be sequence of length $d$ in which each $\eta_{k}$ with $1 \leq k \leq d$ is replaced by the $j$th entry of $\rho I$.
\item The action of an element $\rho \in S_{\mathbb{N}}$ on a general element $\sum_{\eta}c_{\eta}x^{\eta} \in R\langle\langle\mathbf{x}_{\infty}\rangle\rangle$, with $c_{\eta} \in R$ for every $\eta$, takes it to $\sum_{\eta}c_{\eta}x^{\rho\eta}$ for $\rho\eta$ from part $(ii)$.
\item A \emph{non-commutative quasi-symmetric function}, or a \emph{word quasi-symmetric function}, is an element of $R\langle\langle\mathbf{x}_{\infty}\rangle\rangle$ that has finite degree and is invariant under this action of $S_{\mathbb{N}}$. The set of such elements is denoted by $\mathtt{NCQSym}$.
\end{enumerate} \label{NCQSymdef}
\end{defn}
If we denote the part of $R\langle\langle\mathbf{x}_{\infty}\rangle\rangle$ that is homogeneous of degree $d$ (which contains every $x^{\eta}$ for $\eta$ of length $d$) by $R\langle\langle\mathbf{x}_{\infty}\rangle\rangle_{d}$, then this submodule is closed under the action of $S_{\mathbb{N}}$ from Definition \ref{NCQSymdef}, and $\mathtt{NCQSym}$ is the direct sum over $d$ of the invariant subspace of $R\langle\langle\mathbf{x}_{\infty}\rangle\rangle_{d}$.

\medskip

Definition \ref{NCQSymdef} implies that a basis for $\mathtt{NCQSym}$ can be described once we parameterize the orbits of the action of $S_{\mathbb{N}}$ on monomials in $R\langle\langle\mathbf{x}_{\infty}\rangle\rangle$. To do so, we will consider an alternative notation for the monomials in $R\langle\langle\mathbf{x}_{\infty}\rangle\rangle$.
\begin{lem}
For every $d\geq0$ there is a natural bijection between the set of sequences $\eta$ of natural numbers of length $d$ and the set of pairs $(\mathbf{A},I)$ where $\mathbf{A}\vDash\mathbb{N}_{d}$ and $I\subseteq\mathbb{N}$ satisfy $\ell(\mathbf{A})=|I|$. \label{altmon}
\end{lem}

\begin{proof}
Given such a sequence $\eta$, we take $I$ and $\ell$ to be the associated set from Definition \ref{NCQSymdef} and its size respectively, and write $I$ as $\{i_{j}\}_{j=1}^{\ell}$ in increasing order. For every $1 \leq j\leq\ell$, we write $A_{j}$ to be the set of indices of entries in $\eta$ that equal $i_{j}$. Then it is clear that $\mathbf{A}:=(A_{1},\ldots,A_{\ell})$ is a set partition of $\mathbb{N}_{d}$, of length $\ell=|I|$. This gives a map from sequences to pairs.

Conversely, assume that $\mathbf{A}\vDash\mathbb{N}_{d}$ and $I$ with $\ell(\mathbf{A})=|I|$ are given, and write again $I=\{i_{j}\}_{j=1}^{\ell}$ in increasing order. Then, for every $1 \leq k \leq d$, there exists a unique $j=j(k)$ for which $k \in A_{j}$, and we set $\eta_{k}:=i_{j}$, yielding a sequence of length $d$. This is the map in the opposite direction, and it is clear that these maps are inverses, yielding the desired bijection. This proves the lemma.
\end{proof}
Using the bijection from Lemma \ref{altmon}, we can writ each monomial $x^{\eta}$ from Definition \ref{NCQSymdef} also as $x_{I}^{\mathbf{A}}$. Then it is clear that the action of $\rho \in S_{\mathbb{N}}$ on $R\langle\langle\mathbf{x}_{\infty}\rangle\rangle$ takes a monomial $x_{I}^{\mathbf{A}}$, with $\ell(\mathbf{A})=|I|$, to $x_{\rho I}^{\mathbf{A}}$. Using this, now get, via Definitions \ref{NCQSymdef} and \ref{insset}, the following result, which compares with the expressions involving $\Delta$ in Subsection 5.2 of \cite{[BZ]}.
\begin{lem}
Set $\mathrm{M}_{\mathbf{A}}$ for a set composition $\mathbf{A}\vDash\mathbb{N}_{d}$ with $\ell(\mathbf{A})=\ell$ to be the \emph{monomial non-commutative quasi-symmetric function} $\sum_{|I|=\ell}x_{I}^{\mathbf{A}}$, so that in particular $\mathrm{M}_{\emptyset}=1$. Then $\{\mathrm{M}_{\mathbf{A}}\}_{\mathbf{A}}$, in which $\mathbf{A}$ goes over all set compositions of $\mathbb{N}_{d}$ for all non-negative integers $d$, is a basis for $\mathtt{NCQSym}$ over $R$. This module is closed under the product from $R\langle\langle\mathbf{x}_{\infty}\rangle\rangle$, yielding a quasi-shuffle like rule of the form $\mathrm{M}_{\mathbf{A}}\mathrm{M}_{\mathbf{B}}=l_{A_{1}}\mathrm{M}_{\bar{\mathbf{A}}}\mathrm{M}_{\mathbf{B}}+l_{B_{1}+n}\mathrm{M}_{\mathbf{A}} \mathrm{M}_{\bar{\mathbf{B}}}+l_{A_{1}\cup(B_{1}+n)}\mathrm{M}_{\bar{\mathbf{A}}}\mathrm{M}_{\bar{\mathbf{B}}}$ wherever $\mathbf{A}=l_{A_{1}}\overline{A}\vDash\mathbb{N}_{n}$ and $\mathbf{B}=l_{B_{1}}\overline{B}\vDash\mathbb{N}_{m}$ are non-trivial set compositions. \label{monNCQSym}
\end{lem}
As we saw in Remark \ref{StirBell} that for each $d$, the rank of the degree $d$ part of $\mathtt{NCQSym}$ is finite, we get via Lemma \ref{monNCQSym} that $\mathtt{NCQSym}$ is a (non-commutative) graded algebra through Definition \ref{graded}, with the unit being the isomorphism between $R$ and the part $R\langle\langle\mathbf{x}_{\infty}\rangle\rangle_{0}=R\mathrm{M}_{\emptyset}$ of degree 0 inside $\mathtt{NCQSym}$.

\medskip

The relevant extension of Definition \ref{pathsmult} in this case is as follows.
\begin{defn}
Take an element $\omega \in P_{\ell,k}^{h}$ for some $0 \leq h\leq\min\{\ell,k\}$, as well as two set compositions $\mathbf{A}\vDash\mathbb{N}_{n}$ and $\mathbf{B}\vDash\mathbb{N}_{m}$ of respective lengths $\ell$ and $k$. We replace $\mathbf{B}$ by $\mathbf{B}{[+n]}$ from Definition \ref{ordset}, and form the set compositions $\mathbf{C}_{\omega}(\mathbf{A},\mathbf{B})\vDash\mathbb{N}_{n+m}$, of length $\ell+k-h$, in the following manner. We put the sets from $\mathbf{A}$ in the locations of the l's and the b's of $\omega$ (respecting the order). Those of $\mathbf{B}[+n]$ are put in the same manner in those of the r's and the b's, and we replace the two sets in each b of $\omega$ by their (disjoint) union. \label{pathsNCQ}
\end{defn}
The analogue of Corollaries \ref{prodMs} and \ref{prodMqs} here is the following assertion, using the notation from Definition \ref{pathsNCQ}.
\begin{cor}
If $\mathbf{A}\vDash\mathbb{N}_{n}$ and $\mathbf{B}\vDash\mathbb{N}_{m}$ are standardized set compositions with $\ell(\mathbf{A})=\ell$ and $\ell(\mathbf{B})=k$, then we have $\mathrm{M}_{\mathbf{A}}\mathrm{M}_{\mathbf{B}}=\sum_{\omega \in P_{\ell,k}}\mathrm{M}_{\mathbf{C}_{\omega}(\mathbf{A},\mathbf{B})}$. \label{prodNCMs}
\end{cor}
If $\ell=0$ or $k=0$, the unique element of $P_{\ell,k}$ is such that the expression from Definition \ref{pathsNCQ} yields the other set composition, in correspondence, via Corollary \ref{prodNCMs}, with the fact that $\mathrm{M}_{\emptyset}=1$. Noting that if $\mathbf{A}\vDash\mathbb{N}_{a}$ satisfies $\ell(\mathbf{A})=1$ then $\mathbf{A}=(\mathbb{N}_{a})$, the situation in the setting from Examples \ref{pathex} and \ref{pathexq} is as follows.
\begin{ex}
The three terms of $P_{1,1}$ from Example \ref{pathex} produce, via Corollary \ref{prodNCMs}, the equality $\mathrm{M}_{(\mathbb{N}_{a})}\mathrm{M}_{(\mathbb{N}_{b})}=\mathrm{M}_{(\mathbb{N}_{a},\mathbb{N}_{b}+a)}+\mathrm{M}_{(\mathbb{N}_{a+b})}+\mathrm{M}_{(\mathbb{N}_{b}+a,\mathbb{N}_{a})}$. We now assume that $\mathbb{N}_{a+b}$ is the disjoint union of the sets $A$ and $B$, of sizes $a$ and $b$ respectively. Then using the five terms of $P_{2,1}$, the product $\mathrm{M}_{(A,B)}\mathrm{M}_{(\mathbb{N}_{c})}$ equals $\mathrm{M}_{(A,B,\mathbb{N}_{c}+n)}+\mathrm{M}_{(A,B\cap[\mathbb{N}_{c}+n])}+\mathrm{M}_{(A,\mathbb{N}_{c}+n,B)}+\mathrm{M}_{(A\cap[\mathbb{N}_{c}+n],B)}+\mathrm{M}_{(\mathbb{N}_{c}+n,A,B)}$, where $n:=a+b$. Similarly, with $P_{2,1}$ the product $\mathrm{M}_{(\mathbb{N}_{c})}\mathrm{M}_{(A,B)}$ is the sum of the terms $\mathrm{M}_{(\mathbb{N}_{c},A+c,B+c)}$, $\mathrm{M}_{(\mathbb{N}_{c}\cup[A+c],B+c)}$, $\mathrm{M}_{(A+c,\mathbb{N}_{c},B+c)}$, $\mathrm{M}_{(A+c,\mathbb{N}_{c}\cup[B+c])}$, and $\mathrm{M}_{(A+c,B+c,\mathbb{N}_{c})}$. \label{pathNCex}
\end{ex}
the involution $\omega\mapsto\tilde{\omega}$ that we saw that exists between $P_{\ell,k}$ and $P_{k,\ell}$ does not produce commutativity here, because of the asymmetry in Definition \ref{pathsNCQ}, where it is $\mathbf{B}$ to which we apply $+n$, and not $\mathbf{A}$ (similarly to the asymmetry of $\mathbf{A}|\mathbf{B}$ in Definition \ref{ordset}). This is related to the difference in the inversion numbers via Proposition \ref{projQSymq} below.

With the expressions appearing in Examples \ref{expath} and \ref{expathq}, we get the following formulae.
\begin{ex}
If $\mathbf{A}=(1,2,3)\vDash\mathbb{N}_{3}$ and $\mathbf{B}=(1)\vDash\mathbb{N}_{1}$, then the set composition $\mathbf{C}_{\omega}(\mathbf{A},\mathbf{B})$, where $\omega$ runs over the elements of $P_{3,1}$ in the order they show up in Example \ref{expath}, give $(1,2,3,4)$, $(1,2,4,3)$, $(1,4,2,3)$, $(4,1,2,3)$, $(1,2,34)$, $(1,24,3)$, and $(14,2,3)$. When we interchange $\mathbf{A}$ and $\mathbf{B}$ and let $\omega$ go over $P_{1,3}$ from that example, we get $(2,3,4,1)$, $(2,3,1,4)$, $(2,1,3,4)$, $(1,2,3,4)$, $(2,3,14)$, $(2,13,4)$, and $(12,3,4)$. With $\mathbf{A}=\mathbf{B}=(1,2)\vDash\mathbb{N}_{1}$ and the entries of $P_{2,2}$, the resulting set compositions are $(1,2,3,4)$, $(3,4,1,2)$, $(1,3,2,4)$, $(3,1,4,2)$, $(1,3,4,2)$, and $(3,1,2,4)$, followed by $(1,3,24)$, $(3,1,24)$, $(1,23,4)$, $(3,14,2)$, $(13,2,4)$, and $(13,4,2)$, and finally $(13,24)$. \label{expathNC}
\end{ex}
As the sequences $\omega$ were ordered in Example \ref{expath} in a way that makes the involution $\omega\mapsto\tilde{\omega}$ visible, the difference in the resulting set compositions is visible in Examples \ref{pathNCex} and \ref{expathNC}.

\medskip

We would like to have an analogue of the fundamental quasi-symmetric functions also in this setting. The non-commutativity requires that the order of multiplication be incorporated into the definition. We will get too many such elements, which will thus be linearly dependent (they will contain one of the bases from \cite{[BZ]} as a subset), but some of them will play a role in the formula for the antipode below. For this we shall also need an alternative notation for the monomial basis from Lemma \ref{monNCQSym}.
\begin{defn}
Take a composition $\alpha \vDash n$ and a permutation $\rho \in S_{n}$.
\begin{enumerate}[$(i)$]
\item If $\mathbf{A}$ is $\mathbf{A}_{\alpha}^{\rho}$ from Lemma \ref{szord}, then denote $\mathrm{M}_{\mathbf{A}}$ also by $\mathrm{M}_{\alpha}^{\rho}$.
\item The \emph{fundamental non-commutative quasi-symmetric function} $\mathrm{F}_{\alpha}^{\rho}$ that is associated with $\alpha$ and $\rho$ is $\sum_{\beta\preceq\alpha}\mathrm{M}_{\beta}^{\rho}$.
\item For any set composition $\mathbf{A}\vDash\mathbb{N}_{n}$ we set $\mathrm{F}_{\mathbf{A}}:=\sum_{\mathbf{B}\preceq_{*}\mathbf{A}}\mathrm{M}_{\mathbf{B}}$.
\item We define $\mathrm{E}_{\alpha}^{\rho}$ to be $\sum_{\beta\succeq\alpha}\mathrm{M}_{\beta}^{\rho}$.
\end{enumerate} \label{fundNCQSym}
\end{defn}
Lemma \ref{szord} implies that $\mathrm{F}_{\alpha}^{\rho}$ from Definition \ref{fundNCQSym} can be written as the sum of $\mathrm{M}_{\mathbf{A}}$ over all set compositions $\mathbf{A}\vDash\mathbb{N}_{n}$ for which $\operatorname{sz}(\mathbf{A})\preceq\alpha$ and $\mathbf{B}_{\rho}\preceq\mathbf{A}$. When $n=0$, all the expressions from that definition reduce to $\mathrm{M}_{\emptyset}=1$. Note that our $\mathrm{F}_{\mathbf{A}}$ is not the same as $Q_{\mathbf{A}}$ from \cite{[BZ]}, which is defined to be $\sum_{\mathbf{B}\succeq_{*}\mathbf{A}}\mathrm{M}_{\mathbf{B}}$.
\begin{ex}
Example \ref{Aalpharhoex} implies that $\mathrm{M}_{n}^{\rho}=\mathrm{M}_{(\mathbb{N}_{n})}$ in Definition \ref{fundNCQSym} for every $\rho \in S_{n}$, and $\mathrm{E}_{n}^{\rho}=\mathrm{M}_{(\mathbb{N}_{n})}$ as well. We also get $\mathrm{F}_{1^{n}}^{\rho}=\mathrm{M}_{1^{n}}^{\rho}=\mathrm{M}_{\mathbf{B}_{\rho}}$ for every such $\rho$. For $n=0$ all the expressions are just $\mathrm{M}_{\emptyset}=1$, and when $n=1$ they are all equal to $\mathrm{M}_{(1)}$. The remaining expressions for $n=2$ are given by $\mathrm{E}_{11}^{12}=\mathrm{F}_{2}^{12}=\mathrm{M}_{(12)}+\mathrm{M}_{(1,2)}$ and $\mathrm{F}_{2}^{21}=\mathrm{F}_{11}^{21}=\mathrm{M}_{(12)}+\mathrm{M}_{(2,1)}$. When $n=3$ we get the equalities $\mathrm{M}_{21}^{123}=\mathrm{M}_{21}^{213}=\mathrm{M}_{(12,3)}$, $\mathrm{M}_{21}^{132}=\mathrm{M}_{21}^{312}=\mathrm{M}_{(13,2)}$, $\mathrm{M}_{21}^{231}=\mathrm{M}_{21}^{321}=\mathrm{M}_{(23,1)}$, $\mathrm{M}_{12}^{123}=\mathrm{M}_{21}^{132}=\mathrm{M}_{(1,23)}$, $\mathrm{M}_{12}^{213}=\mathrm{M}_{12}^{231}=\mathrm{M}_{(2,13)}$, and $\mathrm{M}_{12}^{312}=\mathrm{M}_{12}^{321}=\mathrm{M}_{(3,12)}$. The fundamental ones with indices 12 and 21 are obtained by adding $\mathrm{M}_{111}^{\rho}$ for the corresponding $\rho$, and $\mathrm{F}_{3}^{\rho}$ is the sum of the expressions $\mathrm{M}_{3}^{\rho}=\mathrm{M}_{(123)}$, $\mathrm{M}_{21}^{\rho}$ and $\mathrm{M}_{12}^{\rho}$ listed here, and $\mathrm{M}_{111}^{\rho}$. \label{exfundNCQ}
\end{ex}
As $1^{n}\preceq\alpha$ for every $\alpha \vDash n$, the monomial $\mathrm{M}_{1^{n}}^{\rho}$ appears in $\mathrm{F}_{\alpha}^{\rho}$, and determines $\rho$ as we saw in Example \ref{exfundNCQ}.

\medskip

Recall the equality $\sum_{\ell=1}^{n}\ell!S(n,\ell)=\sum_{\sigma \in S_{n}}2^{|\operatorname{Des}(\sigma)|}$, where $\operatorname{Des}(\sigma)$ is the descent set of $\sigma \in S_{n}$, known as the \emph{ordered Bell identity}. We now relate the expressions from part $(iii)$ of Definition \ref{fundNCQSym} with the fundamental non-commutative quasi-symmetric functions there, as well as with the latter identity.
\begin{prop}
Given $\mathbf{A}\vDash\mathbb{N}_{n}$, we have the equality $\mathrm{F}_{\mathbf{A}}=\mathrm{F}_{\operatorname{sz}(\mathbf{A})}^{\rho(\mathbf{A})}$. This subset of the set of fundamental non-commutative quasi-symmetric functions is a basis for degree $n$ part of $\mathtt{NCQSym}$, which lifts the ordered Bell identity. The ordered Bell number $\sum_{\ell=1}^{n}\ell!S(n,\ell)$ is bounded by $2^{n-1}n!$, and the inequality is strict when $n\geq2$. \label{basdim}
\end{prop}
The lifting property from Proposition \ref{basdim} is the statement that this basis, whose size is the dimension $\sum_{\ell=1}^{n}S(n,\ell)$ from Remark \ref{StirBell}, is the disjoint union over $\sigma \in S_{n}$, where the set corresponding to $\sigma$ has size $2^{|\operatorname{Des}(\sigma)|}$.
\begin{proof}
Corollary \ref{weakord} implies that the summands appearing in the definition of $\mathrm{F}_{\mathbf{A}}$ are $\mathbf{A}_{\beta}^{\rho}$, where $\rho=\rho(\mathbf{A})$ and $\beta$ runs over the set compositions of $n$ that satisfy $\beta\preceq\alpha$. This produces the first assertion via Definition \ref{fundNCQSym}, and the fact that they form a basis follow from the same property of $\{\mathrm{M}_{\mathbf{A}}\}_{\mathbf{A}\vDash\mathbb{N}_{n}}$ via Lemma \ref{monNCQSym}, since the matrix expressing them using that basis is unipotent (because $\preceq_{*}$ is an order).

Corollary \ref{weakord} also shows that for a given $\rho \in S_{n}$, the number of choices for $\alpha$ for which $\mathrm{F}_{\alpha}^{\rho}$ is an $\mathrm{F}_{\mathbf{A}}$ is the number of subsets of $\mathbb{N}_{n-1}$ that contain $\operatorname{Des}(\rho)$. As such a subset is determined by how it intersects the complement of $\operatorname{Des}(\rho)$ in $\mathbb{N}_{n-1}$, which is the set $\operatorname{Asc}(\rho)$ of ascents of $\rho$, the number of such sets is $2^{|\operatorname{Asc}(\rho)|}$. We associate this set with $\sigma:=w^{0}_{n}\rho$, where $w^{0}_{n} \in S_{n}$ is the longest element (taking each $i\in\mathbb{N}_{n}$ to $n+1-i$), and the first assertion follows from the fact that we then have $\operatorname{Asc}(\rho)=\operatorname{Des}(\sigma)$.

Since $\operatorname{Des}(\sigma)\subseteq\mathbb{N}_{n-1}$ for every $\sigma \in S_{n}$, we get $2^{|\operatorname{Des}(\sigma)|}\leq2^{n-1}$, yielding the desired bound. The sharp inequality follows since for $n\geq2$ there exists permutations without descents (like $\operatorname{Id}_{n}$), or using $\rho=w^{0}_{n}\sigma$ there are permutations without ascents (e.g., $w^{0}_{n}$). This proves the proposition.
\end{proof}
The total number of fundamental non-commutative quasi-symmetric functions that have degree $n$ in $\mathtt{NCQSym}$ is $2^{n-1}n!$ (because the choice of $\rho \in S_{n}$ and $\alpha \vDash n$ are independent). Proposition \ref{basdim} thus implies that for $n\geq2$ they are linearly dependent, and indeed every one that is not an $\mathrm{F}_{\mathbf{A}}$ there (and there are such) is a linear combination of that basis.
\begin{ex}
Proposition \ref{basdim} (or Definition \ref{fundNCQSym} directly) imply that $\mathrm{F}_{\mathbf{B}_{\rho}}$ equals $\mathrm{F}_{1^{n}}^{\rho}=\mathrm{M}_{1^{n}}^{\rho}$ for any $\rho \in S_{n}$, while $\mathrm{F}_{(\mathbb{N}_{n})}$ is $\mathrm{F}_{n}^{\operatorname{Id}_{n}}$. The remaining expression in degree 2 is $\mathrm{F}_{2}^{21}$, which is the linear combination $\mathrm{F}_{(12)}-\mathrm{F}_{(1,2)}+\mathrm{F}_{(2,1)}$. When $n=3$, we have $\mathrm{F}_{21}^{123}=\mathrm{F}_{(12,3)}$, $\mathrm{F}_{21}^{132}=\mathrm{F}_{(13,2)}$, $\mathrm{F}_{21}^{231}=\mathrm{F}_{(23,1)}$, $\mathrm{F}_{12}^{123}=\mathrm{F}_{(1,23)}$, $\mathrm{F}_{12}^{213}=\mathrm{F}_{(2,13)}$, and $\mathrm{F}_{12}^{312}=\mathrm{F}_{(3,12)}$, while the others from Example \ref{exfundNCQ} are given by $\mathrm{F}_{21}^{132}=\mathrm{F}_{(12,3)}-\mathrm{F}_{(1,2,3)}+\mathrm{F}_{(2,1,3)}$, $\mathrm{F}_{12}^{321}=\mathrm{F}_{(3,12)}-\mathrm{F}_{(3,1,2)}+\mathrm{F}_{(3,2,1)}$, and similarly for the remaining expressions in which $\alpha$ is 21 or 12. When $\alpha=3$ we get, for example, the equality $\mathrm{F}_{3}^{132}=\mathrm{F}_{(123)}-\mathrm{F}_{(12,3)}+\mathrm{F}_{(13,2)}$, while $\mathrm{F}_{3}^{231}=\mathrm{F}_{(123)}-\mathrm{F}_{(1,23)}+\mathrm{F}_{(2,13)}-\mathrm{F}_{(12,3)}+\mathrm{F}_{(23,1)}+\mathrm{F}_{(1,2,3)}-\mathrm{F}_{(2,1,3)}$ and $\mathrm{F}_{3}^{321}$ requires all 13 basis elements $\{\mathrm{F}_{\mathbf{A}}\}_{\mathbf{A}\vDash\mathbb{N}_{3}}$ for spanning it. \label{FAex}
\end{ex}
The expressions from Example \ref{FAex} are in correspondence with those appearing in Example \ref{rhoAex}.

\medskip

We shall be using expressions similar to those from Definitions \ref{multisets} and \ref{multq}.
\begin{defn}
Let $J$ be a multiset of size $n$ as in Definition \ref{multisets}, written as $(j_{1},\ldots,j_{n})$, and take $\rho \in S_{n}$. We define $\rho J$ to be $(j_{\rho^{-1}(1)},\ldots,j_{\rho^{-1}(n)})$, and set $x^{\rho J}$ to be the corresponding monomial from Definition \ref{NCQSymdef}. \label{multNC}
\end{defn}
The inversion in Definition \ref{multNC} is the convention with which $\rho:J\mapsto\rho J$ is an action of $S_{n}$ on sequences of length $n$. In addition, the monomial $x^{\tilde{J}}$ from Definition \ref{multq} is $x^{w^{0}_{n}J}$ in the notation from Definition \ref{multNC}.

Assuming now that $n>0$, for any subset $T\subseteq\mathbb{N}_{n-1}$ and any $\rho \in S_{n}$ we denote the fundamental non-commutative quasi-symmetric function $\mathrm{F}_{\alpha}^{\rho}$ for $\alpha:=\operatorname{comp}_{n}T$ also by $\mathrm{F}_{n,T}^{\rho}$ (or $\mathrm{F}_{\tilde{T}}^{\rho}$ through Remark \ref{gencomp}). We then obtain, using $\mathcal{M}_{n,T}$ from Definition \ref{multisets}, the following analogue of Lemmas \ref{sumFnT} and \ref{sumFnTq}.
\begin{lem}
We have $\mathrm{F}_{n,T}^{\rho}=\sum_{J\in\mathcal{M}_{n,T}}x^{\rho J}$ for every such $n$, $T$, and $\rho$. Equivalently, by setting $\mathcal{M}_{n,T}^{\rho}:=\{\rho J\;|\;J\in\mathcal{M}_{n,T}\}$ we get $\mathrm{F}_{n,T}^{\rho}=\sum_{\eta\in\mathcal{M}_{n,T}^{\rho}}x^{\eta}$, which is the sum over all monomials $x^{\eta}$ of degree $n$ in which $\eta_{\rho(k)}\leq\eta_{\rho(k+1)}$ wherever $1 \leq k<n$, and the equality is strict for any $k \in T$. \label{sumFnTNC}
\end{lem}
Note that if $J$ has strictly increasing entries, then the smallest variable in $x^{\rho J}$ appears in the location number $\rho(1)$ etc., so that the monomial non-commutative quasi-symmetric function containing this monomial is indeed $\mathrm{M}_{\mathbf{B}_{\rho}}$ via Definition \ref{setwithr}. Combining similar observations with the ideas yielding Lemmas \ref{sumFnT} and \ref{sumFnTq} proves both expressions in Lemma \ref{sumFnTNC} as well.

\medskip

For the bi-algebra structure on $\mathtt{NCQSym}$, the co-unit is, as usual through Lemma \ref{propgr}, the isomorphism between the degree 0 part $R\mathrm{M}_{\emptyset}$ and $R$ and the other homogeneous parts sent to 0. The co-product is defined, analogously to Definition \ref{DeltaQSym}, as follows.
\begin{defn}
Take an element $\mathrm{G}\in\mathtt{NCQSym}$, and replace the ordered set of non-commutative variables by the union of $\{x_{i}\}_{i=1}^{\infty}$ and $\{y_{i}\}_{i=1}^{\infty}$, where $y_{1}$ is again considered as having a larger index than $x_{i}$ for any $i$. The resulting expression $\mathrm{G}(\mathbf{x}_{\infty},\mathbf{y}_{\infty})$ in invariant under the action of $S_{\mathbb{N}}$ from Definition \ref{NCQSymdef}, and in every monomial showing up in it, we take the part of it involving $\{x_{i}\}_{i=1}^{\infty}$ (in the same order) to the left, and the with the $y_{i}$'s to the right. Thus summing the resulting tensor product of monomials over all those appearing in $\mathrm{G}(\mathbf{x}_{\infty},\mathbf{y}_{\infty})$ produces an element of $\mathtt{NCQSym}\otimes_{R}\mathtt{NCQSym}$, which we define to be $\Delta(\mathrm{G})$. \label{DeltaNCQSym}
\end{defn}
Equivalently, we can take $\mathrm{G}(\mathbf{x}_{\infty},\mathbf{y}_{\infty})$ in Definition \ref{DeltaNCQSym}, project it modulo the relation in which each $x_{i}$ commutes with each $y_{j}$, and the resulting image algebra is canonically isomorphic to $\mathtt{NCQSym}\otimes_{R}\mathtt{NCQSym}$, yielding our $\Delta(\mathrm{G})$.

Using the notation from Definition \ref{setwithr}, we obtain the following analogue of Propositions \ref{coprodQSym} and \ref{coprodQSymq}.
\begin{prop}
Assume that $\mathbf{A}=(A_{1},\ldots,A_{\ell})\vDash\mathbb{N}_{n}$ with $n>0$, and that $\alpha:=\operatorname{sz}(\mathbf{A}) \vDash n$ is $\operatorname{comp}_{n}T$ for $T\subseteq\mathbb{N}_{n-1}$, and set $\tilde{T}:=T\cup\{0,n\}$ as usual. Then we have $\Delta(\mathrm{M}_{\mathbf{A}})=\sum_{r\in\tilde{T}}\mathrm{M}_{\mathbf{A}_{|r}}\otimes\mathrm{M}_{\mathbf{A}_{r|}}$. For the fundamental non-commutative quasi-symmetric functions from Definition \ref{fundNCQSym}, the formula for the co-multiplication is $\Delta(\mathrm{F}_{\alpha}^{\rho})=\sum_{r=0}^{n}\mathrm{F}_{\alpha_{|r}}^{\rho_{|r}}\otimes\mathrm{F}_{\alpha_{r|}}^{\rho_{r|}}$. \label{coprodNCQSym}
\end{prop}
Proposition \ref{coprodNCQSym}, which clearly extends to $n=0$ using the image under $\Delta$ of $\mathrm{M}_{\emptyset}=1$, shows that the co-multiplication from Definition \ref{DeltaNCQSym} is the same one from \cite{[BZ]} (see Equation (16) there). It also involves a reminiscent of the co-multiplication in one of the algebras from \cite{[MR]}, whose formula is given by $\rho\mapsto\sum_{r=0}^{n}\rho_{|r}\otimes\rho_{r|}$ for $\rho \in S_{n}$ in the notation from Definition \ref{setwithr}. In fact, this argument proves the following consequence of Proposition \ref{coprodNCQSym}.
\begin{cor}
The map taking $\rho \in S_{n}$, as an element of that algebra from \cite{[MR]}, to our $\mathrm{F}_{1^{n}}^{\rho}=\mathrm{M}_{1^{n}}^{\rho}=\mathrm{F}_{\mathbf{B}_{\rho}}$ from Example \ref{FAex}, is a map of co-algebras. \label{MRcoalg}
\end{cor}
The map from Corollary \ref{MRcoalg} is, however, not an embedding of Hopf algebras, as it does not respect the algebra structure.

Recall from Proposition \ref{basdim} that the $\mathrm{F}_{\alpha}^{\rho}$'s from Definition \ref{fundNCQSym} are not a basis for $\mathtt{NCQSym}$, but it contains the basis $\{\mathrm{F}_{\mathbf{A}}\}_{\mathbf{A}\vDash\mathbb{N}_{n}}$. Using the notation from Definition \ref{setwithr}, we obtain the following property of this basis.
\begin{cor}
We have $\Delta(\mathrm{F}_{\mathbf{A}})=\sum_{r=0}^{n}\mathrm{F}_{\mathbf{A}_{|r}}\otimes\mathrm{F}_{\mathbf{A}_{r|}}$ for every $\mathbf{A}\vDash\mathbb{N}_{n}$. \label{FADelta}
\end{cor}

\begin{proof}
Proposition \ref{basdim} expresses $\mathrm{F}_{\mathbf{A}}$ as $\mathrm{F}_{\alpha}^{\rho}$ for $\alpha:=\operatorname{sz}(\mathbf{A})$ and $\rho:=\rho(\mathbf{A})$, and then Proposition \ref{coprodNCQSym} implies that its $\Delta$-image is $\sum_{r=0}^{n}\mathrm{F}_{\alpha_{|r}}^{\rho_{|r}}\otimes\mathrm{F}_{\alpha_{r|}}^{\rho_{r|}}$. When $n=0$ we only have the usual $\delta$-image of $\mathrm{F}_{\emptyset}=\mathrm{M}_{\emptyset}=1$ and there is nothing to prove, so we assume that $n\geq1$.

Next, Corollary \ref{weakord} shows that $T:=\operatorname{comp}_{n}^{-1}\alpha\subseteq\mathbb{N}_{n-1}$ contains $\operatorname{Des}(\rho)$. Definition \ref{setwithr} implies that $\operatorname{Des}(\rho_{|r})$ and $\operatorname{Des}(\rho_{r|})$ are $\operatorname{Des}(\rho)_{|r}$ and $\operatorname{Des}(\rho)_{r|}$ using the notation for these operations on sets in Definition \ref{cutatr}. They are therefore contained in $T_{|r}\operatorname{comp}_{r}^{-1}\alpha_{|r}$ and $T_{r|}\operatorname{comp}_{n-r}^{-1}\alpha_{r|}$ respectively, so that the $r$th summand is $\mathrm{F}_{\mathbf{B}_{r}}\otimes\mathrm{F}_{\mathbf{C}_{r}}$ for set compositions $\mathbf{B}_{r}\vDash\mathbb{N}_{r}$ and $\mathbf{C}_{r}\vDash\mathbb{N}_{n-r}$ by another application of Corollary \ref{weakord}.

Now, it is clear from Definitions \ref{setwithr} and \ref{cutatr} that $\operatorname{sz}(\mathbf{A}_{|r})=\alpha_{|r}=\operatorname{sz}(\mathbf{B}_{r})$ and $\operatorname{sz}(\mathbf{A}_{r|})=\alpha_{r|}=\operatorname{sz}(\mathbf{C}_{r})$. Moreover, the operations from the former definition commute with the refinement from Definition \ref{setord} and preserve the order (when extended to set compositions of ordered sets which are not necessarily $\mathbb{N}_{d}$ for some $d$), and thus commute also with the second order there, and clearly these operations are all preserved under the standardizations from Definition \ref{ordset}. It follows (via the uniqueness in part $(vi)$ of Lemma \ref{szord}) that the equalities $\rho(\mathbf{A}_{|r})=\rho_{|r}=\rho(\mathbf{B}_{r})$ and $\rho(\mathbf{A}_{r|})=\rho_{r|}=\rho(\mathbf{C}_{r})$ hold as well.

As knowing $\operatorname{sz}(\mathbf{B})$ and $\rho(\mathbf{B})$ determined a set composition $\mathbf{B}\vDash\mathbb{N}_{d}$ (via part $(iii)$ of Lemma \ref{szord}), we deduce that $\mathbf{B}_{r}=\mathbf{A}_{|r}$ and $\mathbf{C}_{r}=\mathbf{A}_{r|}$ for every $0 \leq r \leq n$, which establishes the asserted formula. This proves the corollary.
\end{proof}
Corollary \ref{FADelta} shows that the basis $\bigcup_{n=0}^{\infty}\{\mathrm{F}_{\mathbf{A}}\}_{\mathbf{A}\vDash\mathbb{N}_{n}}$ is the analogue for the fundamental basis for $\mathtt{QSym}$ and $\mathtt{NCQSym}$, in the sense that it is a basis and has the corresponding behavior under the co-multiplication.

\begin{ex}
Take $\mathbb{A}$ to be $(23,5,14)\vDash\mathbb{N}_{5}$ from Examples \ref{sepsetcomp} and \ref{refAex}, where we saw that $\alpha:=\operatorname{sz}(\mathbf{A})=212\vDash5$ and $\rho:=\rho(\mathbf{A})=23514 \in S_{5}$. Using the parameters from the former example, Corollary \ref{FADelta} gives that $\overline{\Delta}(\mathrm{F}_{\mathbf{A}})$ is the sum of $\mathrm{F}_{(1)}\otimes\mathrm{F}_{(2,4,13)}$, $\mathrm{F}_{(12)}\otimes\mathrm{F}_{(3,12)}$, $\mathrm{F}_{(12,3)}\otimes\mathrm{F}_{(12)}$, and $\mathrm{F}_{(23,4,1)}\otimes\mathrm{F}_{(1)}$ (and $\Delta(\mathrm{F}_{\mathbf{A}})$ is $\overline{\Delta}(\mathrm{F}_{\mathbf{A}})$ plus the two terms $1\otimes\mathrm{F}_{\mathbf{A}}$ and $\mathrm{F}_{\mathbf{A}}\otimes1$ via Equation \eqref{redDelta} as usual). Expressing each of these summands via Proposition \ref{basdim}, these four summands are $\mathrm{F}_{1}^{1}\otimes\mathrm{F}_{112}^{2413}$, $\mathrm{F}_{2}^{12}\otimes\mathrm{F}_{12}^{312}$, $\mathrm{F}_{21}^{123}\otimes\mathrm{F}_{2}^{12}$, and $\mathrm{F}_{211}^{2341}\otimes\mathrm{F}_{1}^{1}$, with the superscripts being the values of $\rho_{|r}$ and $\rho_{r|}$ from Example \ref{sepsetcomp}. \label{exDeltaFA}
\end{ex}
The parameters from Example \ref{exDeltaFA} also exemplify the fact that for every $\mathbb{A}\vDash\mathbb{N}_{n}$ and $0 \leq r \leq n$ we have $\operatorname{sz}(\mathbf{A}_{|r})=\operatorname{sz}(\mathbf{A})_{|r}$ and $\operatorname{sz}(\mathbf{A}_{r|})=\operatorname{sz}(\mathbf{A})_{r|}$.

\medskip

We can now turn to the main result of this section. As $\mathtt{NCQSym}$ is a connected graded bi-algebra as in Definitions \ref{graded} and Definition \ref{condef}, it is a Hopf algebra via Lemma \ref{grconS}. While we do not obtain the formula for the antipode for any fundamental non-commutative quasi-symmetric function from Definition \ref{fundNCQSym}, we get a result for a subset of them, of the same size (for every degree) as those appearing in Theorems \ref{SQSymF} and \ref{SQSymqF}.
\begin{thm}
Given a composition $\alpha \vDash n$, the antipode $S$ of $\mathtt{NCQSym}$ takes the element $\mathrm{F}_{\alpha}^{\operatorname{Id}_{n}}$ to $(-1)^{n}\mathrm{F}_{\alpha^{t}}^{w^{0}_{n}}$, where $w^{0}_{n}$ is the longest element of $S_{n}$. \label{SNCQSymF}
\end{thm}
Theorem \ref{SNCQSymF} considers, in fact, the antipode image of a subset of the basis considered in Proposition \ref{basdim}. Indeed, it gives a formula for $S(\mathrm{F}_{\mathbf{A}})$ for every $\mathbf{A}\vDash\mathbb{N}_{n}$ for which $\mathbf{B}_{\operatorname{Id}_{n}}\preceq_{*}\mathbf{A}$ via Definition \ref{setord} where we recall from Definition \ref{setwithr} that $\mathbf{B}_{\operatorname{Id}_{n}}$ is the set composition $(1,2,\ldots,n)\vDash\mathbb{N}_{n}$ (this is equivalent, for this choice of $\rho \in S_{n}$, also to $\mathbf{B}_{\operatorname{Id}_{n}}\preceq\mathbf{A}$). This image is, however, not of the form $(-1)^{n}\mathrm{F}_{\tilde{\mathbf{A}}}$ for any set composition $\tilde{\mathbf{A}}\vDash\mathbb{N}_{n}$, except when $\alpha=n$, where $\alpha^{t}=1^{n}$ and $\mathrm{F}_{\alpha^{t}}^{w^{0}_{n}}=\mathbf{B}_{w^{0}_{n}}$.

Using Theorem \ref{SNCQSymF} we get, in the following analogue of Propositions \ref{SQSymM} and \ref{SQSymqM} a formula for the antipode on a partial set of the monomial non-commutative quasi-symmetric functions from Lemma \ref{monNCQSym}, using the other elements of $\mathtt{NCQSym}$ from Definition \ref{fundNCQSym}.
\begin{prop}
If $\alpha \vDash n$ then $S(\mathrm{M}_{\alpha}^{\operatorname{Id}_{n}})=(-1)^{\ell(\alpha)}\mathrm{E}_{\alpha^{r}}^{w^{0}_{n}}$. \label{SNCQSymM}
\end{prop}

\begin{proof}
We again follow the proof of Proposition \ref{SQSymM}. The same argument proves that $\mathrm{M}_{\alpha}^{\operatorname{Id}_{n}}=\sum_{\beta\preceq\alpha}(-1)^{\ell(\beta)-\ell(\alpha)}\mathrm{F}_{\beta}^{\operatorname{Id}_{n}}$ (and in fact, the same holds when $\operatorname{Id}_{n}$ is replaced by any $\rho \in S_{n}$ on both sides). Applying $S$, Theorem \ref{SNCQSymF} takes the summand associated with $\beta$ to $(-1)^{n}\mathrm{F}_{\beta^{t}}^{w^{0}_{n}}$, which we expand via Definition \ref{fundNCQSym} and gather the coefficients in front of $\mathrm{M}_{\gamma}^{w^{0}_{n}}$. As the same argument from that proof shows that the total coefficient is $(-1)^{\ell(\alpha)}$ when $\gamma\preceq\alpha^{r}$, this yields the desired result through the corresponding part of that definition. This completes the proof of the proposition.
\end{proof}

We will see in, e.g., Proposition \ref{antipode3} below, that the formulae from Theorem \ref{SNCQSymF} and Proposition \ref{SNCQSymM} do not extend to almost any other $\mathrm{F}_{\alpha}^{\rho}$ or $\mathrm{M}_{\alpha}^{\rho}$, where $n\geq2$ and $\rho\neq\operatorname{Id}_{n}$. In fact, their values do not give a signed basis element in any of our bases. However, for each such $n$ there two exceptions where we do get nice expressions for the antipode.
\begin{thm}
For $n\geq2$ we have the equalities $S(\mathrm{F}_{n}^{w^{0}_{n}})=(-1)^{n}\mathrm{F}_{1^{n}}^{\operatorname{Id}_{n}}$ and $S(\mathrm{F}_{1^{n}}^{w^{0}_{n}})=S(\mathrm{M}_{1^{n}}^{w^{0}_{n}})=(-1)^{n}\mathrm{F}_{n}^{\operatorname{Id}_{n}}=(-1)^{n}\mathrm{E}_{1^{n}}^{\operatorname{Id}_{n}}$.
\label{SFalphaw0}
\end{thm}
The fact that only the superscript $\operatorname{Id}_{n}$ and $w^{0}_{n}$ can produce simple expressions for the antipode as in Theorems \ref{SNCQSymF} and \ref{SFalphaw0} is perhaps related to these permutations having simple antipode formulae in the Hopf algebra from \cite{[MR]} (see Propositions 9.4 and 9.5 of \cite{[BS]}), while others do not (as the rest of Section 9 of that reference discusses).

Considering the proof of Corollary \ref{deg2QSymq}, here in the degree 2 part of $\mathtt{NCQSym}$ is spanned by the four (linearly dependent) fundamental elements $\mathrm{F}_{11}^{12}=\mathrm{F}_{(1,2)}$, $\mathrm{F}_{11}^{21}=\mathrm{F}_{(2,1)}$,  $\mathrm{F}_{2}^{12}=\mathrm{F}_{(12)}$, and $\mathrm{F}_{2}^{21}=\mathrm{F}_{(12)}-\mathrm{F}_{(1,2)}+\mathrm{F}_{(2,1)}$. Theorems \ref{SNCQSymF} and \ref{SFalphaw0} show that $S$ interchanges them in pairs, $\mathrm{F}_{11}^{21}$ with $\mathrm{F}_{2}^{12}$ and $\mathrm{F}_{11}^{12}$ with $\mathrm{F}_{2}^{21}$ (and the linear dependence is preserved). Thus $S^{2}$ is the identity on this part. We will consider a higher part in Corollary \ref{deg3NCQSym} below.

\medskip

Recall from Remark \ref{SymQSym} that $\mathtt{QSym}$ contains $\mathtt{Sym}$, which is the invariants of the action of $S_{\mathbb{N}}$ on $R\ldbrack\mathbf{x}_{\infty}\rdbrack$ that coincides with the one from \cite{[Hi2]} on the part of degree 1 but extend as automorphisms of rings. We also have the inclusion of $\mathtt{Sym}^{(q)}$ from Remark \ref{SymQSymq} inside $\mathtt{QSym}^{(q)}$. We may similarly consider the action of that group on $R\langle\langle\mathbf{x}_{\infty}\rangle\rangle$ but via ring automorphisms which coincide with the one from Definition \ref{NCQSymdef} on the degree 1 part. This produces the following analogous situation.
\begin{rmk}
The invariants of the ring automorphisms action of $S_{\mathbb{N}}$ on $R\langle\langle\mathbf{x}_{\infty}\rangle\rangle$ is the algebra $\mathtt{NCSym}$ of \emph{non-commutative symmetric functions}. They admit a monomial basis (among several others) that is parameterized by set compositions. , the basis element $\mathrm{m}_{\mathbf{L}}\in\mathtt{NCSym}$ that is associated with $\mathbf{L}\vdash\mathbb{N}_{n}$ has the following description as an element of $\mathtt{NCQSym}$. If $\ell(\mathbf{L})=\ell$ then $\mathbf{L}$ represents the $S_{\ell}$-orbit of $l!$ set compositions $\mathbf{A}\vDash\mathbb{N}_{n}$, and we write the assertion that $\mathbf{L}$ is the orbit of $\mathbf{A}$ as $S_{\ell}\mathbf{A}=\mathbf{L}$. Then we have $\mathrm{m}_{\mathbf{L}}=\sum_{\{\mathbf{A}\;|\;S_{\ell}\mathbf{A}=\mathbf{L}\}}\mathrm{M}_{\mathbf{A}}\in\mathtt{NCQSym}$, in a manner similar to Remark \ref{SymQSym}. Note that $\mathtt{NCSym}$ is non-commutative but co-commutative, while $\mathtt{NCQSym}$ is not co-commutative either. \label{NCSym}
\end{rmk}
The space $\mathtt{NCSym}$ from Remark \ref{NCSym} was initially introduced in \cite{[W]}. The algebra structure there investigated in \cite{[RS]}, and \cite{[BRRZ]} established their Hopf algebra properties (note that \cite{[BHRZ]} also considers a different Hopf algebra structure on $\mathtt{NCSym}$). The paper \cite{[LM]} considers the primitives in that Hopf algebra, and uses it for establishing the antipode formula for that algebra (see also Subsection 8.13 of \cite{[D]}). One can compare their formula, in low degrees, with the one that we present in Remark \ref{NCSymS} below.

\medskip

Note that the ring $R\ldbrack\mathbf{x}_{\infty}\rdbrack$ is a quotient of $R\langle\langle\mathbf{x}_{\infty}\rangle\rangle$ under the ideal generated by $x_{j}x_{i}-x_{i}x_{j}$ for all $i<j$. As this ideal respects the action from \cite{[Hi1]} and \cite{[Hi2]}, it takes $\mathtt{NCQSym}$ to $\mathtt{QSym}$, as follows.
\begin{prop}
The map taking $\mathrm{M}_{\mathbf{A}}\in\mathtt{NCQSym}$ for $\mathbf{A}\vDash\mathbb{N}_{n}$ to $M_{\operatorname{sz}(\mathbf{A})}\in\mathtt{QSym}$ is a surjective homomorphism of Hopf algebras. It also sends $\mathrm{M}_{\alpha}^{\rho}$ to $M_{\alpha}$, $\mathrm{F}_{\alpha}^{\rho}$ to $F_{\alpha}$, $\mathrm{F}_{\mathbf{A}}$ to $F_{\operatorname{sz}(\mathbf{A})}$, and $\mathrm{E}_{\alpha}^{\rho}$ to $E_{\alpha}$ for any such $\mathbf{A}$ and for $\alpha \vDash n$ and $\rho \in S_{n}$. \label{projQSym}
\end{prop}
The fact that the map from Proposition \ref{projQSym} is a map of algebras follows from the way Remark \ref{szlrel} connects the quasi-shuffle rule from Lemma \ref{monQSym} with the one from Lemma \ref{monNCQSym} (and the fact that it sends the unit $\mathrm{M}_{\emptyset}$ of $\mathtt{NCQSym}$ to that of $\mathtt{QSym}$, namely $m_{\emptyset}$). Corollaries \ref{prodMs} and \ref{prodNCMs} are thus related via the equality $\operatorname{sz}\big(\mathbf{C}_{\omega}(\mathbf{A},\mathbf{B})\big)=\gamma_{\omega}(\alpha,\beta)$ where $\alpha=\operatorname{sz}(\mathbf{A})$ and $\beta=\operatorname{sz}(\mathbf{B})$, which one can verify using via Definitions \ref{pathsmult}, \ref{pathsNCQ}, and \ref{sizes}.

Definitions \ref{sizes}, \ref{setwithr}, and \ref{cutatr} yield the equalities $\operatorname{sz}(\mathbf{A}_{|r})=\operatorname{sz}(\mathbf{A})_{|r}$ and $\operatorname{sz}(\mathbf{A}_{r|})=\operatorname{sz}(\mathbf{A})_{r|}$ for every $\mathbf{A}\vDash\mathbb{N}_{n}$ and $0 \leq r \leq n$. Thus the map from Proposition \ref{projQSym} is also a map of co-algebras, and thus of bi-algebras. It is thus a map of Hopf algebras, hence commutes with the two antipodes, which makes Theorem \ref{SQSymF} and Proposition \ref{SQSymM} consequences of Theorem \ref{SNCQSymF} and Proposition \ref{SNCQSymM} respectively. However, our proof below will produce Theorems \ref{SQSymF} and \ref{SNCQSymF} simultaneously.

We can also consider, after tensoring with $R[q]$, the map from $R[q]\langle\langle\mathbf{x}_{\infty}\rangle\rangle$ onto $R\langle\langle\mathbf{x}_{\infty}\rangle\rangle^{(q)}$, which if obtained after dividing by the expressions $x_{j}x_{i}-qx_{i}x_{j}$ for all pairs with $i<j$. This projection also commutes with the operation from \cite{[Hi1]} and \cite{[Hi2]}, and we get a map from $\mathtt{NCQSym}$ (over $R[q]$) onto $\mathtt{QSym}^{(q)}$, which we now describe, using the parameter $\operatorname{inv}$ from Definition \ref{sizes}.
\begin{prop}
We have a map of algebras from $\mathtt{NCQSym}$ onto $\mathtt{QSym}^{(q)}$, which given $\mathbf{A}\vDash\mathbb{N}_{n}$, sends $\mathrm{M}_{\mathbf{A}}$ to $q^{\operatorname{inv}(\mathbf{A})}M_{\operatorname{sz}(\mathbf{A})}^{(q)}$. In particular it maps $\mathrm{M}_{\alpha}^{\operatorname{Id}_{n}}$ to $M_{\alpha}^{(q)}$, $\mathrm{F}_{\alpha}^{\operatorname{Id}_{n}}$ to $F_{\alpha}^{(q)}$, $\mathrm{F}_{\alpha}^{w^{0}_{n}}$ to $q^{\operatorname{inv}(\alpha)}\widetilde{F}_{\alpha}^{(q)}$, and $\mathrm{E}_{\alpha}^{w^{0}_{n}}$ to $E_{\alpha}^{(q)}$ wherever $\alpha \vDash n$. \label{projQSymq}
\end{prop}
In case $\mathbf{B}_{\operatorname{Id}_{n}}\preceq\mathbf{A}$, so that $\mathbf{A}=\mathbf{A}_{\alpha}^{\operatorname{Id}_{n}}$ for $\alpha=\operatorname{sz}(\mathbf{A})$, the map from Proposition \ref{projQSymq} sends $\mathrm{F}_{\mathbf{A}}$ to $F_{\alpha}$ as well. It is not hard to describe, for any $\alpha \vDash n$, an element $\rho \in S_{n}$ for which it maps $\mathrm{F}_{\alpha}^{\rho}$ to the remaining element $\widetilde{F}_{\alpha}^{(q)}$ from Definition \ref{fundQSymq}. For relating Corollaries \ref{prodMqs} and \ref{prodNCMs}, we may apply the equality used for Corollary \ref{prodMs} as well as, for any such $\mathbf{A}$, $\mathbf{B}$, $\alpha=\operatorname{sz}(\mathbf{A})$ and $\beta=\operatorname{sz}(\mathbf{B})$, the equality $\operatorname{inv}\big(\mathbf{C}_{\omega}(\mathbf{A},\mathbf{B})\big)=\operatorname{inv}(\mathbf{A})+\operatorname{inv}(\mathbf{B})+\operatorname{inv}_{\omega}(\alpha,\beta)$.

It is important to note that the map from Proposition \ref{projQSymq} is not a map of co-algebras (or bi-algebras, or Hopf algebras). Indeed, as we saw in Remark \ref{qHopf}, $\mathtt{QSym}^{(q)}$ is not a Hopf algebra, but rather a $q$-Hopf algebra, so cannot be the image of a Hopf algebra. To see this explicitly, note that both $\mathrm{M}_{(1,2)}$ and $\mathrm{M}_{(2,1)}$ have $\overline{\Delta}$-images $\mathrm{M}_{(1)}\otimes\mathrm{M}_{(1)}$. But as we saw in Example \ref{qex}, the former maps to $M_{11}^{(q)}$, with $\overline{\Delta}$-image $M_{1} \otimes M_{1}$ (as a map of co-algebras would produce), but the latter is sent to $qM_{11}^{(q)}$, whose $\overline{\Delta}$-image is $q \cdot M_{1} \otimes M_{1}$, which does not respect the co-multiplication.

We remark that applying Proposition \ref{projQSym} to the cases presented in Examples \ref{pathNCex} and \ref{expathNC} produces Examples \ref{pathex} and \ref{expath} respectively, and doing the same with \ref{projQSymq} yields, respectively, Examples \ref{pathexq} and \ref{expathq}.

\medskip

There is another Hopf algebra that is sometimes called non-commutative symmetric functions, which is denoted by $\mathtt{NSym}$ and was shown in \cite{[MR]} to be the graded dual of $\mathtt{QSym}$ (it is thus co-commutative and non-commutative). It is freely generated as an non-commutative algebra by symbols $\{H_{n}\}_{n=1}^{\infty}$, with $H_{n}$ having degree $n$, and the co-multiplication takes $H_{n}$ to $\sum_{r=0}^{n}H_{r} \otimes H_{n-r}$ (where $H_{0}=1$). It is equally generated by $\{E_{n}\}_{n=1}^{\infty}$, with the same property under the co-multiplication (and $E_{0}=1$ as well), and where the Newton identities $\sum_{r=0}^{n}(-1)^{r}E_{r}H_{n-r}=\sum_{r=0}^{n}(-1)^{n-r}H_{r}E_{n-r}=\delta_{n,0}$ hold (this means, via Corollary \ref{detS}, that its antipode takes $H_{n}$ to $(-1)^{n}E_{n}$ and $E_{n}$ to $(-1)^{n}H_{n}$).

This algebra has the following relations with the other ones discussed here.
\begin{prop}
The following assertions hold:
\begin{enumerate}[$(i)$]
\item The map taking $H_{n}$ to $h_{n}$, or equivalently $E_{n}$ to $e_{n}$, is a surjection of Hopf algebras from $\mathtt{NSym}$ onto $\mathtt{Sym}$.
\item When we view $\mathtt{NCSym}$ as a sub-algebra of $\mathtt{NCQSym}$, the map from Proposition \ref{projQSym} reduces to one from $\mathtt{NCSym}$ to $\mathtt{Sym}$, which is surjective when $R$ is a $\mathbb{Q}$-algebra.
\item Under the assumption on $R$ from part $(ii)$, we can embed $\mathtt{NSym}$ into $\mathtt{NCSym}$ in a way that the projection from that part restricts to the one from part $(i)$ from $\mathtt{NSym}$ onto $\mathtt{Sym}$.
\item There is a Hopf ideal $I$ of $\mathtt{NCSym}$ such that $\mathtt{NCSym}/I$ is naturally isomorphic to $\mathtt{NSym}$.
\item The operation of annihilating all but two of the variables from $\{x_{i}\}_{i=1}^{\infty}$ produces another quotient of $\mathtt{NCSym}$ that is isomorphic to $\mathtt{NSym}$.
\item The map sending $H_{n}$ to $\mathrm{F}_{n}^{\operatorname{Id}_{n}}=\mathrm{F}_{(\mathbb{N}_{n})}$ is an embedding of $\mathtt{NSym}$ directly into $\mathtt{NCQSym}$, which takes $E_{n}$ to $\mathrm{F}_{1^{n}}^{w^{0}_{n}}=\mathrm{M}_{1^{n}}^{w^{0}_{n}}=\mathrm{M}_{\mathbf{B}_{w^{0}_{n}}}=\mathrm{F}_{\mathbf{B}_{w^{0}_{n}}}$.
\item The map from part $(vi)$ yields, via Proposition \ref{projQSym}, the map from part $(i)$ composed with the embedding of $\mathtt{Sym}$ into $\mathtt{QSym}$ via Remark \ref{SymQSym}.
\end{enumerate} \label{NSym}
\end{prop}
For all the maps in Proposition \ref{NSym}, see Sections 4 and 5 of \cite{[BRRZ]}.
\begin{rmk}
The reason why parts $(ii)$ and $(iii)$ of Proposition \ref{NSym} require $R$ to be a $\mathbb{Q}$-algebra is that the map from Proposition \ref{projQSym} does not take the monomial $\mathrm{m}_{\mathbf{L}}$ from Remark \ref{NCSym} to $m_{\operatorname{sz}(\mathbf{L})}$ inside $\mathtt{Sym}$, but rather to an integral scalar multiple of it (by the product of factorials of the multiplicities of numbers in the partition $\operatorname{sz}(\mathbf{L})$). These scalars are why the formula from Theorem 4.6 of \cite{[BRRZ]} takes the asserted form for the image of $H_{n}$. Working with $E_{n}$, we only need the monomial non-commutative symmetric function that is associated with the unique set partition of $\mathbb{N}_{n}$ of length $n$ (the orbit of all the $\mathbf{B}_{\rho}$'s for $\rho \in S_{n}$). But $E_{n}$ is not sent to this monomial, but rather to its quotient over $n!$, since the image of that monomial under Proposition \ref{projQSym} is $n!e_{n}$. \label{notoverQ}
\end{rmk}
Remark \ref{notoverQ} explains why for getting the projection from part $(i)$ of Proposition \ref{NSym} without integral denominators, we need to embed $\mathtt{NSym}$ into $\mathtt{NCQSym}$ (as in part $(vi)$ there), and cannot work with $\mathtt{NCSym}$. The comparison with $\mathtt{Sym}^{(q)}$ from Remark \ref{SymQSymq} is as follows.
\begin{rmk}
Consider the embedding of $\mathtt{NSym}$ into $\mathtt{NCQSym}$ via part $(vi)$ of Proposition \ref{NSym}, and apply the map from Proposition \ref{projQSymq}. When working with the generators $\{H_{n}\}_{n=1}^{\infty}$, we get $H_{n} \mapsto h_{n}^{(q)}$, which respects the co-multiplication of such elements (though not on their products, because of the $q$-Hopf structure from Remark \ref{qHopf}). However, $E_{n}$ is taken to $q^{n(n-1)/2}e_{n}^{(q)}$ rather than $e_{n}^{(q)}$, which no longer commutes with the co-multiplication. This is the reason why the Newton identities in Remark \ref{SymQSymq} are not the images of those appearing in $\mathtt{NSym}$, but rather involve powers of $q$. \label{NSymq}
\end{rmk}
For Remark \ref{NSymq}, see \cite{[L]}, which also explains how to modify the structure on $\mathtt{NSym}$ in order to make it the graded dual of the Hopf algebra $\mathtt{QSym}^{(q)}$.

\section{Proofs of the Main Theorems \label{Proofs}}

The proofs of our results will be by establishing the vanishing of the products that is required in Corollary \ref{detS} for determining the images of elements under the antipodes. The products involved will be of fundamental quasi-symmetric functions and their generalizations. The product of two such objects is based on shuffling of their indices, and we will take advantage of a degree of freedom in doing so, as appears in, e.g., Corollary 5.12 of \cite{[AsS]}.

\subsection{Products of Fundamental Quasi-Symmetric Functions}

We begin with introducing our notation for shuffles and related objects.
\begin{defn}
Let $n$ and $m$ be two non-negative integers.
\begin{enumerate}[$(i)$]
\item A \emph{shuffle of size $(n,m)$} is a sequence $\sigma$ of length $m+n$ which contains $n$ instances of the letter l and $m$ instances of the letter r. We write $\mathcal{S}_{n,m}$ for the set of such shuffles, whose size is clearly $\frac{(n+m)!}{n!m!}$.
\item For $\sigma\in\mathcal{S}_{n,m}$ and $1 \leq t<n+m$, write $t$ as $a+b$ where among the first $t$ entries of $\sigma$ there are $a$ instances of l and $b$ of r. Then the entries in the locations $t$ and $t+1$ of $\sigma$ can be ll, lr, rl, and rr.
\item Assume that $n$ and $m$ are positive, and take $\sigma\in\mathcal{S}_{n,m}$, $T\subseteq\mathbb{N}_{n-1}$ and $S\subseteq\mathbb{N}_{m-1}$. We define $R^{\sigma}(T,S)\subseteq\mathbb{N}_{n+m-1}$ by running over $1 \leq t<n+m$, and consider the cases from part $(ii)$ and the parameters $a$ and $b$, as determined by $\sigma$. In the ll case we put $t$ inside $R^{\sigma}(T,S)$ if and only $a \in T$, and in the rr case we do the same with whether $b \in S$. We always put $t$ in $R^{\sigma}(T,S)$ in the rl case, and never in the lr case.
\item For $\sigma\in\mathcal{S}_{n,m}$ with positive $n$ and $m$, and for compositions $\alpha \vDash n$ and $\beta \vDash m$, we set $T:=\operatorname{comp}_{n}^{-1}\alpha\subseteq\mathbb{N}_{n-1}$ and $S:=\operatorname{comp}_{m}^{-1}\beta\subseteq\mathbb{N}_{m-1}$, and define $\gamma^{\sigma}(\alpha,\beta) \vDash n+m$ to be the image of $R^{\sigma}(T,S)$ from part $(iii)$ under
    $\operatorname{comp}_{n+m}$.
\item Let $\rho \in S_{n}$ and $\varphi \in S_{m}$ be permutations, both written in one-line notation, and take $\sigma\in\mathcal{S}_{n,m}$. Write $\varphi+n$ for the sequence obtained by increasing all the entries of $\varphi$ by $n$. We define $\psi^{\sigma}(\rho,\varphi) \in S_{n+m}$ to be the permutation in $S_{n+m}$ that is obtained by putting $\rho$ in the locations marked with l in $\sigma$ and $\varphi+n$ in those containing r, preserving the orders on both.
\end{enumerate} \label{shuffle}
\end{defn}
The expression $\varphi+n$ can be seen the permutation of the set $\mathbb{N}_{n+m}\setminus\mathbb{N}_{n}$ that is obtained after conjugating $\varphi$ by the order-preserving map between that set and $\mathbb{N}_{m}$. We extend Definition \ref{shuffle} to the case where $n$ (resp. $m$) vanishes, by setting $\gamma^{\sigma}(\alpha,\beta)$ to be $\beta$ (resp. $\alpha$). This is in correspondence with $\mathcal{S}_{n,m}$ containing a single element, the one with $m$ instances of r (resp. $n$ instances of l), and the corresponding set from part $(iii)$ there, determined only by the rr (resp. ll) case, being $S=\operatorname{comp}_{m}^{-1}\beta$ (resp. $T=\operatorname{comp}_{n}^{-1}\alpha$).

To obtain the smallest examples, recall that if $n=1$ then $T=\emptyset$ and $\alpha=1\vDash1$, and similarly when $m=1$, the set $S$ is empty and $\beta=1\vDash1$. When both are 1 there is a single element to check for $R^{\sigma}(\emptyset,\emptyset)$ and $\gamma^{\sigma}(1,1)$, and when either $n=1$ and $m=2$ or $n=2$ and $m=1$, here are two elements to consider in the corresponding shuffles, which have length 3. Also, in the latter case we have either $T=\emptyset$ and $\alpha=2\vDash2$ or $T=\{1\}$ and $\alpha=11\vDash2$, and the one before it involves either $S=\emptyset$ and $\beta=2\vDash2$ or $S=\{1\}$ and $\beta=11\vDash2$.
\begin{ex}
When $n=m=1$, the set $\mathcal{S}_{1,1}$ contains lr, with which $R^{\sigma}(\emptyset,\emptyset)$ does not contain the element and is thus empty, and then $\gamma^{\sigma}(1,1)=2\vDash2$. It also contains rl, where $R^{\sigma}(\emptyset,\emptyset)=\{1\}$ and $\gamma^{\sigma}(1,1)=11\vDash2$. For $n=2$ and $m=1$, from the element lrl of $\mathcal{S}_{2,1}$ we get the set $R^{\sigma}(T,\emptyset)=\{2\}$ and $\gamma^{\sigma}(\alpha,1)=21\vDash3$, irrespective of $\alpha$ and $T$. From llr we get $R^{\sigma}(\emptyset,\emptyset)=\emptyset$ and $\gamma^{\sigma}(2,1)=3\vDash3$ as well as $R^{\sigma}(\{1\},\emptyset)=\{1\}$ and $\gamma^{\sigma}(11,1)=12\vDash3$, while rll produces $R^{\sigma}(\emptyset,\emptyset)=\{1\}$ and $\gamma^{\sigma}(2,1)=12\vDash3$ and also $R^{\sigma}(\{1\},\emptyset)=\{1,2\}$ and $\gamma^{\sigma}(11,1)=111\vDash3$. Similarly, when $n=1$ and $m=2$, the set $\mathcal{S}_{1,2}$ contains rlr, for which $R^{\sigma}(\emptyset,S)=\{1\}$ and $\gamma^{\sigma}(1,\beta)=12$ for both values of $S$ and $\beta$. With lrr we obtain $R^{\sigma}(\emptyset,\emptyset)=\emptyset$ and $\gamma^{\sigma}(1,2)=3$ and also $R^{\sigma}(\emptyset,\{1\})=\{2\}$ and $\gamma^{\sigma}(1,2)=21$, and the element rrl yields $R^{\sigma}(\emptyset,\emptyset)=\{2\}$ and $\gamma^{\sigma}(1,2)=12$ and also $R^{\sigma}(\emptyset,\{1\})=\{1,2\}$ and $\gamma^{\sigma}(1,11)=111$. \label{exshuff}
\end{ex}

\begin{ex}
The group $S_{1}$ has a single element with one-line notation 1. The two elements $\mathcal{S}_{1,1}$, appearing in the order from Example \ref{exshuff}, produce for $\psi^{\sigma}(1,1)$ the respective elements 12 and 21 of $S_{2}$. For $\mathcal{S}_{2,1}$, there are three elements, and in their order of appearance in that example the resulting values of $\psi^{\sigma}(12,1)$ are 132, 123, and 312, and for $\psi^{\sigma}(21,1)$ they produce 231, 213, and 321. With the three elements of $\mathcal{S}_{1,2}$ in that order, $\psi^{\sigma}(1,12)$ yields 213, 123, and 231, while for $\psi^{\sigma}(1,21)$ we get 312, 132, and 321. \label{permshuff}
\end{ex}
Note that despite the similarity of writing compositions with no separation and permutations in one-line notation, when the external letter is $\gamma$ we mean compositions (as in Example \ref{exshuff}), and when it is $\psi$ the expressions are permutations (like in Example \ref{permshuff}).
\begin{rmk}
The letters l and r correspond to the left and the right composition involved in the calculation. The conditions from Definition \ref{shuffle} correspond to taking elements of $S_{n}$ and $S_{m}$ with descent sets $T$ and $S$, increasing the entries of the one from $S_{m}$ by $n$ (to make them larger than those of the one from $S_{n}$), shuffling them according to $\sigma$, and taking the descent set of the resulting permutation in $S_{n+m}$. While the result does not depend on the choice of permutations, our formulation does not require it at all. \label{withperms}
\end{rmk}
The next examples, with $n+m=4$, are as follows.
\begin{ex}
The element llrl of $\mathcal{S}_{3,1}$ yields $\gamma^{\sigma}(3,1)=\gamma^{\sigma}(21,1)=31$ and $\gamma^{\sigma}(12,1)=\gamma^{\sigma}(111,1)=121$, while lrll produces $\gamma^{\sigma}(3,1)=\gamma^{\sigma}(12,1)=22$ and $\gamma^{\sigma}(21,1)=\gamma^{\sigma}(111,1)=211$. From lllr we have $\gamma^{\sigma}(3,1)=4$, $\gamma^{\sigma}(21,1)=22$, $\gamma^{\sigma}(12,1)=13$, and $\gamma^{\sigma}(111,1)=112$. With the remaining element rlll we get $\gamma^{\sigma}(3,1)=13$, $\gamma^{\sigma}(21,1)=121$, $\gamma^{\sigma}(12,1)=112$, and $\gamma^{\sigma}(111,1)=1111$. Turning to $\mathcal{S}_{1,3}$, rlrr produces the expressions $\gamma^{\sigma}(1,3)=\gamma^{\sigma}(1,12)=13$ and $\gamma^{\sigma}(1,21)=\gamma^{\sigma}(1,111)=121$, while rrlr yields $\gamma^{\sigma}(1,3)=\gamma^{\sigma}(1,21)=22$ and $\gamma^{\sigma}(1,12)=\gamma^{\sigma}(1,111)=112$. With lrrr we get $\gamma^{\sigma}(1,3)=4$, $\gamma^{\sigma}(1,21)=31$, $\gamma^{\sigma}(1,12)=22$, and $\gamma^{\sigma}(1,111)=211$. Finally, from rrrl we obtain $\gamma^{\sigma}(1,3)=31$, $\gamma^{\sigma}(1,21)=211$, $\gamma^{\sigma}(1,12)=121$, and $\gamma^{\sigma}(1,111)=1111$. \label{shuffex}
\end{ex}

\begin{ex}
When $n=m=2$, the set $\mathcal{S}_{2,2}$ contains 6 elements, and $\alpha$ and $\beta$ can each be either $2\vDash2$ or $11\vDash2$. The element lrlr yields $\gamma^{\sigma}(\alpha,\beta)=22$ irrespective of $\alpha$ and $\beta$, and with rlrl we always have $\gamma^{\sigma}(\alpha,\beta)=121$. Taking lrrl produces $\gamma^{\sigma}(\alpha,2)=31$ and $\gamma^{\sigma}(\alpha,11)=211$ for both compositions $\alpha$, and when the element is rllr we get $\gamma^{\sigma}(2,\beta)=13$ and $\gamma^{\sigma}(11,\beta)=112$ for each $\beta$. Turning to llrr, we obtain $\gamma^{\sigma}(2,2)=4$, $\gamma^{\sigma}(2,11)=31$, $\gamma^{\sigma}(11,2)=13$, and $\gamma^{\sigma}(11,11)=121$. The last element is rrll, with the expressions $\gamma^{\sigma}(2,2)=22$, $\gamma^{\sigma}(2,11)=112$, $\gamma^{\sigma}(11,2)=211$, and $\gamma^{\sigma}(11,11)=1111$. \label{shuff22}
\end{ex}
In Example \ref{shuffex} we got 16 possible expressions using $\mathcal{S}_{3,1}$, and another 16 from $\mathcal{S}_{1,3}$. The total number of expressions in Example \ref{shuff22} is 24 (with many following a pattern). For $\psi^{\sigma}$'s in this setting, the numbers are 24 from each. Here are some of them.
\begin{ex}
Assume that $\sigma$ is the element lrll of $\mathcal{S}_{3,1}$. If we let $\rho$ run over the 6 elements 123, 213, 132, 231, 312, and 321 of $S_{3}$, then $\psi^{\sigma}(\rho,1)$ gives 1423, 2413, 1432, 2431, 3412, and 3421. Similarly we take rrlr in $\mathcal{S}_{1,3}$, and when $\varphi$ goes over $S_{3}$ in that order, the value of $\psi^{\sigma}(1,\varphi)$ becomes 2314, 3214, 2413, 3412, 4213, and 4312. We now take the element lrlr of $\mathcal{S}_{2,2}$, and get $\psi^{\sigma}(12,12)=1324$, $\psi^{\sigma}(21,12)=2314$, $\psi^{\sigma}(12,21)=1423$, and $\psi^{\sigma}(21,21)=2413$. Similarly, with rlrl the respective values are 3142, 3241, 4132, and 4231. \label{permsum4}
\end{ex}
Note that the two elements from $\mathcal{S}_{2,2}$ for which $\gamma^{\sigma}(\alpha,\beta)$ was the same regardless of the choice of compositions $\alpha$ and $\beta$ of 2, the four associated permutations $\psi^{\sigma}(\rho,\varphi)$ in Example \ref{permsum4} are all different.

\medskip

Some of the notions in Definition \ref{shuffle} are used for establishing the formula for the product of two fundamental quasi-symmetric functions in $\mathtt{QSym}$. We would like to show that this formula lifts to $\mathtt{NCQSym}$, with the appropriate parameters. In order to do so, we begin with the following construction.
\begin{defn}
For positive $n$ and $m$ and permutations $\rho \in S_{n}$ and $\varphi \in S_{m}$, take subsets $T\subseteq\mathbb{N}_{n-1}$ and $S\subseteq\mathbb{N}_{m-1}$, and let $\eta\in\mathcal{M}_{n,T}^{\rho}$ and $\kappa\in\mathcal{M}_{m,S}^{\varphi}$ be elements of the sets from Lemma \ref{sumFnTNC}. Then we define $\sigma(\eta,\kappa)\in\mathcal{S}_{n,m}$ by induction as follows. Assume that the first $a+b$ entries in $\sigma$ were determined, with $a$ of them being l's and the other $b$ are r's (the initial step is the case where $a=b=0$). If $a=n$ or $b=m$ then we complete with the remaining symbol to land inside $\mathcal{S}_{n,m}$. Otherwise, we consider the indices $\eta_{\rho(a+1)}$ and $\kappa_{\varphi(b+1)}$, and we put the symbol l as the next entry of $\sigma$ in case $\eta_{\rho(a+1)}\leq\kappa_{\varphi(b+1)}$, and fill it in with an r when $\eta_{\rho(a+1)}>\kappa_{\varphi(b+1)}$. \label{defsigma}
\end{defn}
In fact, when $n$ or $m$ vanishes the construction from Definition \ref{defsigma} still works, and produces the unique element of $\mathcal{S}_{n,m}$ for every $\eta$ and $\kappa$ (one of which runs over the set $\{\emptyset\}$). The main calculation is based on the following observation.
\begin{lem}
For $n$, $m$, $T$, and $S$ as in Definition \ref{defsigma}, and take an element $\sigma\in\mathcal{S}_{n,m}$. Then the concatenation $(\eta,\kappa)\mapsto\eta\kappa$ is a bijection between the set of pairs $(\eta,\kappa)\in\mathcal{M}_{n,T}^{\rho}\times\mathcal{M}_{m,S}^{\varphi}$ for which $\sigma(\eta,\kappa)=\sigma$ and the set $\mathcal{M}_{n+m,R}^{\psi}$, with $\psi:=\psi^{\sigma}(\rho,\varphi)$ and $R:=R^{\sigma}(T,S)\subseteq\mathbb{N}_{n+m-1}$ from Definition \ref{shuffle}. \label{matchsigma}
\end{lem}

\begin{proof}
This concatenation takes a pair of sequences $\eta$ and $\kappa$ of lengths $n$ and $m$ to a sequence $\theta:=\eta\kappa$ of length $n+m$, and it is clear that every sequence $\theta$ is a unique such composition, where $\eta$ is the first $n$ elements of $\theta$, and $\kappa$ is the last $m$ entries. We thus need to show that when $\theta=\eta\kappa$ we have $\theta\in\mathcal{M}_{n+m,R}^{\psi}$ if and only if $\eta\in\mathcal{M}_{n,T}^{\rho}$, $\kappa\in\mathcal{M}_{m,S}^{\varphi}$, and $\sigma(\eta,\kappa)=\sigma$.

We first assume that $\theta\in\mathcal{M}_{n+m,R}^{\psi}$, and note that for every $1 \leq j<n$, the index $\rho(j)$ is $\psi(k)$ where $k$ is the $j$'s location of l in $\sigma$. Then $\rho(j+1)$ is $\psi(h)$ for the next location of l in $\sigma$, so that $h>k$ and thus $\psi(h)\geq\psi(k)$. Moreover, the only case where $\theta_{\psi(k)}=\theta_{\psi(h)}$ is where none of the entries between $k$ and $h-1$ are in $R^{\sigma}(T,S)$. This means that there can be no r in between (for rl producing an element of that set), and thus $h=k+1$, and for $R^{\sigma}(T,S)$ not to contain $k$ we need $j \not\in T$.

This proves that $\eta\in\mathcal{M}_{n,T}^{\rho}$, and a similar argument with $\varphi+n$ implies that $\kappa\in\mathcal{M}_{m,S}^{\varphi}$. It also shows that the only situation where $\theta_{\psi(k)}=\theta_{\psi(h)}$ for indices $k$ and $h$ for which $\psi(k)=\varphi(j)+n$ and $\psi(h)=\varphi(j+1)+n$ is where $h=k+1$ and $j \not\in S$. We need to show that $\sigma(\eta,\kappa)$ from Definition \ref{defsigma} is our $\sigma$.

For this, we take an index $a+b+1$, with $a<n$ and $b<m$, and assume that we verified that the first $a+b$ entries of $\sigma(\eta,\kappa)$ coincide with those of $\sigma$, and contained $a$ instances of l and $b$ instances of $r$. Definition \ref{shuffle} shows that $\psi(a+b+1)$ equals $\rho(a+1)$ (resp. $\varphi(b+1)+n$) if that location of $\sigma$ is l (resp. r), and if we define $c>a+b+1$ be the minimal location of an r (resp. l) in $\sigma$ then $\psi(c)$ is $\varphi(b+1)+n$ (resp. $\rho(a+1)$). Since $\theta_{\psi(a+b+1)}\leq\theta_{\psi(c)}$, we deduce in the l case that $\eta_{\rho(a+1)}\leq\kappa_{\varphi(b+1)}$ and indeed the location number $a+b+1$ in $\sigma(\eta,\kappa)$ is also l. In the r case we get $a+b+1 \leq c-1 \in R$ and thus $\kappa_{\varphi(b+1)}=\theta_{\psi(a+b+1)}<\theta_{\psi(c)}=\eta_{\rho(a+1)}$, so that this location is r also in $\sigma(\eta,\kappa)$.

Therefore all the locations in $\sigma(\eta,\kappa)$ coincide with those in $\sigma$ until we cover either all the l's or all the r's. But then the rest is also the same since both of these sequences are in $\mathcal{S}_{n,m}$. We have thus established this direction.

Conversely, assume that $\eta$ and $\kappa$ are in the corresponding sets and produce $\sigma$ via Definition \ref{defsigma}. We need to show that $\theta_{\psi(t)}\leq\theta_{\psi(t+1)}$ for any $1 \leq t<n+m$, with the inequality being strict when $t \in R$. There are the four cases for the locations $t$ and $t+1$ of $\sigma$, namely ll, lr, rl, and rr as in Definition \ref{shuffle}. We can write $t=a+b$, where the first $t$ entries of $\sigma$ consist of $a$ l's and $b$ r's.

But Definition \ref{shuffle} implies that $\theta_{\psi(t)}$ is $\eta_{\rho(a)}$ in the former two cases and $\kappa_{\varphi(b)}$ in the latter two cases, while the value of $\theta_{\psi(t+1)}$ is $\eta_{\rho(a+1)}$ in the first and third cases, and it is $\kappa_{\varphi(b+1)}$ in the second and fourth ones. Moreover, we deduce from Definition \ref{defsigma} that the first and third case correspond to the inequality $\eta_{\rho(a+1)}\leq\kappa_{\varphi(b+1)}$, while the second and fourth one are associated with $\eta_{\rho(a+1)}>\kappa_{\varphi(b+1)}$. Applying these considerations to the $t$th location implies that $\eta_{\rho(a)}\leq\kappa_{\varphi(b+1)}$ in the first two cases, and produces the inequality $\kappa_{\varphi(b)}<\eta_{\rho(a+1)}$ in last two.

Now, our assumption on $\eta$ and $\kappa$ produce $\eta_{\rho(a)}\leq\eta_{\rho(a+1)}$ and $\kappa_{\varphi(b)}\leq\kappa_{\varphi(b+1)}$, which yields $\theta_{\psi(t)}\leq\theta_{\psi(t+1)}$ in all four cases. As $t \not\in R$ in the lr case, and we have a strict inequality in rl case, these two cases are covered. As in the first (resp. fourth) case we have $t \in R$ if and only if $a \in T$ (resp. $b \in S$), which then implies $\eta_{\rho(a)}<\eta_{\rho(a+1)}$ (resp. $\kappa_{\varphi(b)}<\kappa_{\varphi(b+1)}$), we deduce the required strict inequality also in these cases. Hence $\theta$ satisfies all the conditions for being in $\mathcal{M}_{n+m,R}^{\psi}$, as desired. This completes the proof of the lemma.
\end{proof}
The extension of Definition \ref{shuffle} to the case of vanishing $n$ or $m$ implies that the assertion of Lemma \ref{matchsigma} is also valid in these cases.

\medskip

We can now establish our formula.
\begin{prop}
Take compositions $\alpha \vDash n$ and $\beta \vDash m$, as well as permutations $\rho \in S_{n}$ and $\varphi \in S_{m}$. Then $\mathrm{F}_{\alpha}^{\rho}\mathrm{F}_{\beta}^{\varphi}=\sum_{\sigma\in\mathcal{S}_{n,m}}\mathrm{F}_{\gamma^{\sigma}(\alpha,\beta)}^{\psi^{\sigma}(\rho,\varphi)}$ in $\mathtt{NCQSym}$. \label{prodF}
\end{prop}

\begin{proof}
We would like to employ Lemma \ref{sumFnTNC}, which in our formulation applies only when both $n$ and $m$ are positive. But when $n$ or $m$ vanishes, the corresponding multiplier is 1, the set $\mathcal{S}_{n,m}$ consists of a single expression, and the resulting parameters $\gamma^{\sigma}(\alpha,\beta)$ and $\psi^{\sigma}(\rho,\varphi)$ produce those of the other multiplier. So the equality holds in this case.

We thus assume that $n$ and $m$ are positive, so that $\alpha:=\operatorname{comp}_{n}T$ for some $T\subseteq\mathbb{N}_{n-1}$ and $\beta=\operatorname{comp}_{m}S$ with $S\subseteq\mathbb{N}_{m-1}$, and the product in question is $\mathrm{F}_{n,T}^{\rho}\mathrm{F}_{m,S}^{\varphi}$. We express both multipliers via Lemma \ref{sumFnTNC}, so that our product becomes $\sum_{\eta\in\mathcal{M}_{n,T}^{\rho}}\sum_{\kappa\in\mathcal{M}_{m,S}^{\varphi}}x^{\eta}x^{\kappa}$, and each summand is the monomial associated with the concatenation $\eta\kappa$.

Recalling the map from Definition \ref{defsigma}, we can separate this expression for $\mathrm{F}_{n,T}^{\rho}\mathrm{F}_{m,S}^{\varphi}$ as $\sum_{\sigma\in\mathcal{S}_{n,m}}\sum_{(\eta,\kappa)\in\mathcal{M}_{n,T}^{\rho}\times\mathcal{M}_{m,S}^{\varphi},\ \sigma(\eta,\kappa)=\sigma}x^{\eta\kappa}$. But then Lemma \ref{matchsigma} expresses the sum associated with each $\sigma\in\mathcal{S}_{n,m}$ as $\sum_{\theta\in\mathcal{M}_{n+m,R}^{\psi}}x^{\theta}$, with the parameters $\psi:=\psi^{\sigma}(\rho,\varphi)$ and $R:=R^{\sigma}(T,S)\subseteq\mathbb{N}_{n+m-1}$ as in Definition \ref{shuffle}. Another application of Lemma \ref{sumFnTNC} reduces the latter sum to $\mathrm{F}_{n+m,R}^{\psi}$, and as this is the same as $\mathrm{F}_{\gamma}^{\psi}$ for $\gamma:=\gamma^{\sigma}(\alpha,\beta)$ (Definition \ref{shuffle} again), the desired expression is established. This proves the proposition.
\end{proof}
In fact, using Remark \ref{gencomp} and the resulting extension of Lemmas \ref{sumFnTNC} and \ref{matchsigma} and Definition \ref{defsigma} to the case where $n$ or $m$ vanishes, we could have avoided the separation into cases in Proposition \ref{prodF}. The formula in $\mathtt{QSym}$ follows from Proposition \ref{prodF} after applying the projection from Proposition \ref{projQSym} (which ignores the permutations part).

The cases where $n=0$ or $m=0$ are given by the fact that the corresponding multiplier there is 1. Here are the smallest non-trivial examples.
\begin{ex}
When $n=m=1$, Examples \ref{exshuff} and \ref{permshuff} show that $(\mathrm{F}_{1}^{1})^{2}$ is the sum of $\mathrm{F}_{2}^{12}$ from lr in $\mathcal{S}_{1,1}$, and $\mathrm{F}_{11}^{21}$ arising from rl. Following the terms from $\mathcal{S}_{2,1}$ as they appear in these examples, $\mathrm{F}_{2}^{12}\mathrm{F}_{1}^{1}$ is the sum of $\mathrm{F}_{21}^{132}$, $\mathrm{F}_{3}^{123}$, and $\mathrm{F}_{12}^{312}$, while for $\mathrm{F}_{11}^{12}\mathrm{F}_{1}^{1}$ we get $\mathrm{F}_{21}^{132}$, $\mathrm{F}_{12}^{123}$, and $\mathrm{F}_{111}^{312}$. Replacing the superscript 12 by 21 would replace 132, 123, and 312 by 231, 213, and 321 respectively. Similarly, with $\mathcal{S}_{1,2}$ we get $\mathrm{F}_{1}^{1}\mathrm{F}_{2}^{12}=\mathrm{F}_{12}^{213}+\mathrm{F}_{3}^{123}+\mathrm{F}_{21}^{231}$, $\mathrm{F}_{1}^{1}\mathrm{F}_{11}^{12}=\mathrm{F}_{12}^{213}+\mathrm{F}_{21}^{123}+\mathrm{F}_{111}^{231}$, and with the superscript 21 we change 213, 123, and 231 to 312, 132, and 321 respectively. \label{exprod}
\end{ex}

\begin{rmk}
The two summands in the expression for $(\mathrm{F}_{1}^{1})^{2}$ in Example \ref{exprod} correspond to the fact that when multiplying $\mathrm{F}_{1}^{1}=\sum_{i=1}^{\infty}x_{i}$ with itself, we will have products $x_{i}x_{j}$ with $i \leq j$ and those with $j<i$, yielding the two terms. Similarly, when we multiply $\mathrm{F}_{2}^{12}=\sum_{i \leq j}x_{i}x_{j}$ (resp. $\mathrm{F}_{11}^{12}=\sum_{i<j}x_{i}x_{j}$) with $\mathrm{F}_{1}^{1}$ from the right, we get, in the respective orders, the summands where $i \leq k<j$, those where $i \leq j \leq k$ (resp. $i<j \leq k$), and those in which $k<i \leq j$ (resp. $k<i<j$). The superscript 21 corresponds to taking $j \leq i$ (resp. $j<i$) but still of the products $x_{i}x_{j}$ in that order. When we multiply in the opposite order, namely our $\mathrm{F}_{1}^{1}$ times $\sum_{j \leq k}x_{j}x_{k}$ (resp. $\sum_{j<k}x_{j}x_{k}$), the respective summands are with $j<i \leq k$, then $i \leq j \leq k$ (resp. $i \leq j<k$), and $k \leq j<i$ (resp. $k<j<i$), and again the order on $j$ and $k$ is inverted for the superscript 21. \label{expsumind}
\end{rmk}

\begin{ex}
Take $n=m=3$, with $\rho=312 \in S_{3}$ and $\alpha=21\vDash3$ (so that $T:=\{2\}$) and the other multiplier involving $\varphi=231$ and $\beta=12\vDash3$ (and hence $S:=\{1\}$). Then if we write the monomial $x^{\eta}$ for $\eta\in\mathcal{M}_{3,\{2\}}^{312}$ as $x_{i_{1}}x_{i_{2}}x_{i_{3}}$ and $x^{\kappa}$ with $\kappa\in\mathcal{M}_{3,\{1\}}^{231}$ as $x_{j_{1}}x_{j_{2}}x_{j_{3}}$, then we have $i_{3} \leq i_{1}<i_{2}$ and $j_{2}<j_{3} \leq j_{1}$. We look for the conditions for $\eta$ and $\kappa$ to be such that the element $\sigma(\eta,\kappa)\in\mathcal{S}_{3,3}$ from Definition \ref{defsigma} would be lrrllr. The first l implies that $i_{3} \leq j_{2}$, in the comparison of the two smallest entries. The following r now means that $j_{2}<i_{1}$, and the next r is $j_{3}<i_{1}$ as well. Then an l standing for $i_{1} \leq j_{1}$, another l meaning that $i_{2} \leq j_{1}$ as well, and the final is because we exhausted our l's. Combining these inequalities yields $i_{3} \leq j_{2}<j_{3}<i_{1}<i_{2} \leq j_{1}$, and the product lies in $\mathcal{M}_{6,\{2,3,4\}}^{356124}$. Going over Definition \ref{shuffle} shows that the superscript is $\psi^{\sigma}(\rho,\varphi)$, the second subscript is the set $R^{\sigma}(T,S)$, and the corresponding composition $2112\vDash6$ is $\gamma^{\sigma}(\alpha,\beta)$, as Lemma \ref{matchsigma} predicts. \label{exsigma}
\end{ex}

\begin{rmk}
One also easily verifies that for every monomial in $\mathcal{M}_{6,R}^{\psi}$ from Example \ref{exsigma}, the first three indices satisfy $i_{3}<i_{1}<i_{2}$ and the for last three ones we have $j_{2}<j_{3}<j_{1}$. The inequalities that we got are stricter than the requirement that the left half $x^{\eta}$ of the monomial is with $\eta\in\mathcal{M}_{3,\{2\}}^{312}$ and the right half $x^{\kappa}$ corresponds to $\kappa\in\mathcal{M}_{3,\{1\}}^{231}$, because the determination of $\sigma(\eta,\kappa)$ can only arise when those stronger inequalities hold. For other choices of that element of $\mathcal{S}_{3,3}$, terms with $i_{3}=i_{1}<i_{2}$ as well as those in which $j_{2}<j_{3}=j_{1}$, can show up. \label{strongineq}
\end{rmk}
One can check the smaller cases in the expressions from Example \ref{exprod}, expressed as in Remark \ref{expsumind}, and see in which of them stronger inequalities than the original show up (as in Remark \ref{strongineq}), and in which equalities are allowed.

\medskip

Note that the convention in Definition \ref{defsigma} is that when equalities between the indices in the left multiplier $x^{\eta}$ and those in the right multiplier $x^{\kappa}$ appear, we consider the index from $\eta$ to come before that from $\sigma$. This corresponds to the fact that elements appearing in the case lr of Definition \ref{shuffle} do not go into $R^{\sigma}(S,T)$, and those with rl do. This is visible in Example \ref{expsumind} in the fact that the term $\sum_{i=1}^{\infty}x_{i}^{2}=\mathrm{M}_{12}$ for $(12)\vDash\mathbb{N}_{2}$ is considered, in the product $(\mathrm{F}_{1}^{1})^{2}$, as part of the summand $\mathrm{F}_{2}^{12}$, and not $\mathrm{F}_{11}^{21}$.

The convention we applied is, however, not the only one. Indeed, in the terminology defining shuffles through descents, the way to obtain $R^{\sigma}(T,S)$ is by taking permutations $\rho \in S_{n}$ and $\varphi \in S_{m}$ such that $\operatorname{Des}(\rho)=T$ and $\operatorname{Des}(\varphi)=S$, and then $R^{\sigma}(T,S)$ is the descent set of $\psi^{\sigma}(\rho,\varphi)$. But the latter permutation is based on $\rho$ and $\varphi+n$, namely we chose the convention of making the entries of $\varphi$ (which appear in the r locations of $\sigma$) larger than those of $\rho$ (from the l locations), which refers to and lr never being a descent but an rl always being one. We now wish to generalize this.

In the spirit of the previous paragraph, we recall from Corollary 5.12 of \cite{[AsS]} that the product formula in $\mathtt{QSym}$, in the fundamental basis, can be obtained by substituting any lifts of $\rho$ and $\varphi$ to words on disjoint sets of numbers inside $\mathbb{N}_{n+m}$, not necessarily where those of $\varphi$ are larger than those of $\rho$. We therefore extend Definition \ref{shuffle} as follows.
\begin{defn}
Consider two positive integers $n$ and $m$.
\begin{enumerate}[$(i)$]
\item Consider an element $\sigma\in\mathcal{S}_{n,m}$ and a permutation $\tau \in S_{n+m}$, written in one-line notation. We define the word $w^{\sigma}(\tau) \in S_{n+m}$ by putting the first $n$ entries of $\tau$ in the locations of l in $\sigma$ while preserving their order, and doing the same with the last $m$ entries of $\tau$ and the locations of r in $\sigma$.
\item Given such $\sigma$ as well as subsets $T\subseteq\mathbb{N}_{n-1}$ and $S\subseteq\mathbb{N}_{n-1}$, take $\tau \in S_{n+m}$ such that $\operatorname{Des}(\tau_{|n})=T$ and $\operatorname{Des}(\tau_{n|})=S$. We then define the subset $R^{\sigma}_{\tau}(T,S)\subseteq\mathbb{N}_{n+m-1}$ to be the descent set of $w^{\sigma}(\tau) \in S_{n+m}$ from part $(i)$.
\item For $\alpha \vDash n$ and $\beta \vDash m$, set $T:=\operatorname{comp}_{n}^{-1}\alpha$ and $S:=\operatorname{comp}_{n}^{-1}\alpha$, and then for $\tau$ as in part $(ii)$ and $\sigma\in\mathcal{S}_{n,m}$, with the resulting subset $R^{\sigma}_{\tau}(T,S)\subseteq\mathbb{N}_{n-1}$, we define $\gamma^{\sigma}_{\tau}(\alpha,\beta):=\operatorname{comp}_{n+m}R^{\sigma}_{\tau}(T,S) \vDash n+m$.
\end{enumerate} \label{altshuff}
\end{defn}
Like Definition \ref{shuffle}, Definition \ref{altshuff} for $\gamma^{\sigma}_{\tau}(\alpha,\beta)$ extends to the case where $n$ or $m$ vanishes, where both $\sigma$ and $\tau$ are uniquely determined, and it produces the other composition.

Note that for $\tau \in S_{n+m}$ we indeed have $\tau_{|n} \in S_{n}$ and $\tau_{n|} \in S_{m}$ via Definition \ref{setwithr}, yielding indeed descent sets in $\mathbb{N}_{n-1}$ and in $\mathbb{N}_{m-1}$ respectively in Definition \ref{altshuff}. These descent sets are the same as those of the two sequences obtained by separating $\tau$ in the corresponding location, without the standardization required for yielding $\tau_{|n}$ and $\tau_{n|}$ in the former definition. Moreover, the rule for determining whether an element $t=a+b$ lies in $R^{\sigma}_{\tau}(T,S)$ is the same for every $\tau$ in the ll and rr cases from Definition \ref{shuffle}.

Here are some relations between the notions from Definitions \ref{shuffle} and \ref{altshuff}.
\begin{rmk}
If $\sigma\in\mathcal{S}_{n,m}$ is the sequence beginning with $n$ l's and ending with $m$ r's (we call this $\sigma$ the \emph{trivial sequence}), then $w^{\sigma}(\tau)=\tau$ by definition. We also note that from $\rho \in S_{n}$ and $\varphi \in S_{m}$ as in Definition \ref{shuffle}, the element $\psi^{\sigma}(\rho,\varphi)$ is obtained as $w^{\sigma}(\tau)$ where $\tau$ is the word obtained by concatenating $\rho$ with $\varphi+n$. Conversely, wherever the first $n$ entries of $\tau$ are some element of $S_{n}$ with descent set $T$, and the last $m$ entries are obtained by adding $n$ to an element of $S_{m}$ with descent set $S$, the set $R^{\sigma}_{\tau}(T,S)$ from Definitions \ref{altshuff} reduces to $R^{\sigma}(T,S)$, because all the rl cases must be descents while none of the lr cases are (and hence $\gamma^{\sigma}_{\tau}(\alpha,\beta)=\gamma^{\sigma}(\alpha,\beta)$ for every $\alpha \vDash n$ and $\beta \vDash m$). \label{trivtau}
\end{rmk}
We call the case of $\tau$ as in Remark \ref{trivtau} the \emph{trivial} choice (the identity element of $S_{n+m}$ is thus always trivial, and associated with empty $T$ and $S$ and with $\alpha$ and $\beta$ of length 1), and we now turn to the cases from Example \ref{exshuff} with non-trivial choices of $\tau$.
\begin{ex}
When $n+m=2$, the unique non-trivial element $\tau \in S_{2}$ is 21. With this $\tau$, we get $R^{\sigma}_{\tau}(\emptyset,\emptyset)=\{1\}$ and $\gamma^{\sigma}_{\tau}(1,1)=11\vDash2$ when $\sigma$ is lr, and $R^{\sigma}_{\tau}(\emptyset,\emptyset)=\emptyset$ and $\gamma^{\sigma}_{\tau}(1,1)=2\vDash2$ if $\sigma$ is rl (opposite to Example \ref{exshuff}). For $n=2$ and $m=1$, we consider the case where $\tau$ is 231 when $T=\emptyset\subseteq\mathbb{N}_{1}$, and 321 for the non-empty $T$. Then when $\sigma$ is lrl we get $R^{\sigma}_{\tau}(T,\emptyset)=\{1\}$ and $\gamma^{\sigma}_{\tau}(\alpha,1)=12\vDash3$ for both choices of $\alpha$ and $T$ (with the corresponding $\tau$). If $\sigma$ is llr, then $R^{\sigma}_{\tau}(\emptyset,\emptyset)=\{2\}$ and $\gamma^{\sigma}_{\tau}(2,1)=3\vDash3$ while $R^{\sigma}_{\tau}(\{1\},\emptyset)=\{1,2\}$ and $\gamma^{\sigma}_{\tau}(11,1)=111\vDash3$. Taking $\sigma$ to be rll, we get $R^{\sigma}_{\tau}(\emptyset,\emptyset)=\emptyset$ and $\gamma^{\sigma}_{\tau}(2,1)=3\vDash3$ as well as $R^{\sigma}_{\tau}(\{1\},\emptyset)=\{2\}$ and $\gamma^{\sigma}_{\tau}(11,1)=21\vDash3$. The other choice of $\tau$ is to be 132 for the empty $T$ and 312 when it is non-empty, with which the value lrl of $\sigma$ now produces results that depend on $T$ and $\alpha$, namely $R^{\sigma}_{\tau}(\emptyset,\emptyset)=\emptyset$ and $\gamma^{\sigma}_{\tau}(2,1)=3\vDash3$ while $R^{\sigma}_{\tau}(\{1\},\emptyset)=\{1,2\}$ and $\gamma^{\sigma}_{\tau}(11,1)=111\vDash3$. When $\sigma$ is llr, we again have $R^{\sigma}_{\tau}(\emptyset,\emptyset)=\{2\}$ and $\gamma^{\sigma}_{\tau}(2,1)=21\vDash3$ but $R^{\sigma}_{\tau}(\{1\},\emptyset)=\{1\}$, and $\gamma^{\sigma}_{\tau}(11,1)=12\vDash3$, while if $\sigma$ is rll, then the values are $R^{\sigma}_{\tau}(\emptyset,\emptyset)=\{1\}$, $\gamma^{\sigma}_{\tau}(2,1)=12\vDash3$, $R^{\sigma}_{\tau}(\{1\},\emptyset)=\{2\}$ and $\gamma^{\sigma}_{\tau}(11,1)=21\vDash3$. \label{shufftau}
\end{ex}
One can complement Example \ref{shufftau} by checking the values with $n=1$ and $m=2$ for non-trivial $\tau$. One such choice of $\tau$ is 213 when $S$ is empty and 231 in case it is not, and the other choice is 312 for empty $S$ and 321 for non-empty $S$. We will write only those in which $\sigma$ is rlr, where the results for the latter pair are independent of $S$ and $\beta$, and produce $R^{\sigma}_{\tau}(\emptyset,S)=\{2\}$ and $\gamma^{\sigma}_{\tau}(1,\beta)=21$. With the former pair we get $R^{\sigma}_{\tau}(\emptyset,\emptyset)=\emptyset$ and $\gamma^{\sigma}_{\tau}(1,2)=3$, while $R^{\sigma}_{\tau}(\emptyset,\{1\})=\{1,2\}$ and $\gamma^{\sigma}_{\tau}(1,11)=111$.

\begin{rmk}
When $\tau$ is the trivial choice from Remark \ref{trivtau}, some choices of $\sigma$ produce the notions from Definition \ref{defconc}. Indeed, the trivial sequence $\sigma$ produces the value $\gamma^{\sigma}_{\tau}(\alpha,\beta)=\gamma^{\sigma}(\alpha,\beta)=\alpha\odot\beta$ wherever $n$ and $m$ are positive, while if $\sigma$ is the opposite one, namely consists of $m$ r's followed by $n$ l's, then $\gamma^{\sigma}_{\tau}(\alpha,\beta)=\gamma^{\sigma}(\alpha,\beta)=\beta\alpha$. We now assume that $\tau$ is opposite to the trivial case, namely its last $m$ entries form an element of $S_{m}$ with descent set $S$, and its first $n$ entries are the result of adding $m$ to the entries of an element of $S_{n}$ with descent set $T$. In this case again the values of $R^{\sigma}_{\tau}(T,S)$ and $\gamma^{\sigma}_{\tau}(\alpha,\beta)$ are independent of the choices of permutations with the desired descent sets, and the latter composition reduces to $\alpha\beta$ when $\sigma$ is the trivial sequence and to $\beta\odot\alpha$ when it is the opposite one (when $n$ and $m$ are positive). \label{gammaconc}
\end{rmk}
We note that the only way to get $\gamma^{\sigma}(\alpha,\beta)=n+m$ of length 1 (or equivalently an empty set $R^{\sigma}(T,S)$) is where $\alpha$ and $\beta$ are also compositions of length one (namely $T=S=\emptyset$) and $\sigma$ is the trivial sequence. Similarly, in order to get $\gamma^{\sigma}(\alpha,\beta)=1^{n+m}$ (i.e., $R^{\sigma}(T,S)=\mathbb{N}_{n+m-1}$) we must have $\alpha=1^{n}$ and $\beta=1^{m}$ (namely $T=\mathbb{N}_{n-1}$ and $S=\mathbb{N}_{m-1}$) and $\sigma$ has to be the opposite sequence in $\mathcal{S}_{n,m}$. Here are these extreme cases with Definition \ref{altshuff}.
\begin{rmk}
The case $\gamma^{\sigma}_{\tau}(\alpha,\beta)=n+m$ and $R^{\sigma}(T,S)=\emptyset$ corresponds to $w^{\sigma}(\tau)=\operatorname{Id}_{n+m}$. For this to happen we must have $\tau_{|n}=\operatorname{Id}_{n}$ and $\tau_{n|}=\operatorname{Id}_{m}$, i.e., $T=S=\emptyset$ and $\ell(\alpha)=\ell(\beta)=1$. Then there is a correspondence between the elements $\sigma\in\mathcal{S}_{n,m}$ and elements $\tau \in S_{n+m}$ with these properties, in which $\sigma$ orders $\tau$ to be $\operatorname{Id}_{n+m}$, yielding the desired result. Similarly, $\gamma_{\tau}^{\sigma}(\alpha,\beta)=1^{n+m}$ and $R^{\sigma}_{\tau}(T,S)=\mathbb{N}_{n+m-1}$ are associated with $w^{\sigma}(\tau)$ being the longest element $w^{0}_{n+m}$, which implies that $T$, $S$, $\alpha$, and $\beta$ are the full sets and longest compositions as above, and we once again get a similar bijection for this setting. Note that in these cases, as well as in the Remarks \ref{trivtau} and \ref{gammaconc} and Example \ref{shufftau}, the value of $\tau_{|n}$ and $\tau_{n|}$ was either determined by its descent set, or the resulting $R^{\sigma}_{\tau}(T,S)$ and $\gamma_{\tau}^{\sigma}(\alpha,\beta)$ were independent of $\tau_{|n}$ and $\tau_{n|}$ (and only depended on the set of entries they contained before the standardization). We will soon see that this is not always the case. \label{extreme}
\end{rmk}

\begin{ex}
When $n=1$ (so that $T$ and $\alpha$ are fixed) and $m=3$, if $S=\emptyset$ and $\beta=3\vDash3$, the choices for $\tau$ are 1234, 2134, 3124, and 4123. These are associated, via the bijection from Remark \ref{extreme}, with the elements lrrr, rlrr, rrlr, and rrrl of $\mathcal{S}_{1,3}$ in that order. The elements $\tau$ for which $S=\mathbb{N}_{2}$ and $\beta=111$, which match our ordering of $\mathcal{S}_{1,3}$ in the second bijection from that remark, are 4321, 3421, 2431, and 1432 respectively. For the other choices of $\beta$, the elements $\tau$ that begin with 2 or 3 (so they are neither trivial nor opposite from it) are 2143, 2341, 3142, and 3241 where $\beta=21$, as well as 2314, 2413, 3214, and 3412 yielding $\beta=12$. When we match these elements $\tau$ with $\sigma$ which is rrlr, the composition $\gamma_{\tau}^{\sigma}(\alpha,\beta)\vDash4$ is 22 twice, 211 twice, 13 twice, and then 13 and 121. Similarly, when $\sigma$ is rlrr these are 31 and 121, followed by 31 twice and 112 twice, and ended with 22 and 112. \label{n1m3}
\end{ex}
Note that for $\beta=21\vDash3$ the value of $\tau_{1|}$ can be either 132 or 231, and where $\beta=12$ this standardized part is either 213 or 312. Then $\tau$ itself is determined by its first entry. With rrlr for $\sigma$ in Example \ref{n1m3}, the fact that the compositions 22 and 211 appear twice means that the choice of the permutation $\tau_{1|} \in S_{3}$ with descent set $\{2\}$ (to match with $\beta=21$) does not affect $\gamma_{\tau}^{\sigma}(\alpha,\beta)$, as in the cases from Remarks \ref{trivtau}, \ref{gammaconc}, \ref{extreme}. But for $\beta=12$, with the first entry 3, the element 213 of $S_{3}$ produced $\tau=3214$ and $\gamma_{\tau}^{\sigma}(\alpha,\beta)=13$, while with 312 we got $\tau=3412$ and $\gamma_{\tau}^{\sigma}(\alpha,\beta)=121$, despite both having descent set $\{1\}$ and the same first entry, exemplifying the last sentence in Remark \ref{extreme}.

\medskip

Using Definition \ref{altshuff}, we make the following generalization of Definition \ref{defsigma}.
\begin{defn}
Let $n$, $m$, $\rho$, $\varphi$, $T$, $S$, $\eta$, and $\kappa$ as in Definition \ref{defsigma}, and fix $\tau \in S_{n+m}$ with $\operatorname{Des}(\tau_{|n})=T$ and $\operatorname{Des}(\tau_{n|})=S$ as in Definition \ref{altshuff}, in the one-line notation $\tau_{1}\ldots\tau_{n+m}$. We then define $\sigma_{\tau}(\eta,\kappa)\in\mathcal{S}_{n,m}$ by a similar induction on $a+b$, where the case with $a=n$ or $b=m$ being determined by the $\mathcal{S}_{n,m}$ condition as above. For $a<n$ and $b<m$, we again consider $\eta_{\rho(a+1)}$ and $\kappa_{\varphi(b+1)}$, and put, in the location $a+b+1$ of our element of $\mathcal{S}_{n,m}$, the symbol l when $\eta_{\rho(a+1)}<\kappa_{\varphi(b+1)}$, and the symbol r in case $\eta_{\rho(a+1)}>\kappa_{\varphi(b+1)}$. When $\eta_{\rho(a+1)}=\kappa_{\varphi(b+1)}$, the symbol that we put is l when $\tau_{a+1}<\tau_{n+b+1}$, and it is r in case $\tau_{a+1}>\tau_{n+b+1}$. \label{tausigma}
\end{defn}
As when $\tau$ is trivial as in Remark \ref{trivtau}, we always have $\eta_{\rho(a+1)}<\kappa_{\varphi(b+1)}$, the sequence $\sigma_{\tau}(\eta,\kappa)$ from Definition \ref{tausigma} reduces, with such $\tau$, to $\sigma(\eta,\kappa)$ from Definition \ref{defsigma}. Like with the latter definition, Definition \ref{tausigma} produces, when $n$ or $m$ vanishes, the unique element of $\mathcal{S}_{n,m}$ for every $\eta$ and $\kappa$ (one of which is always $\emptyset$).

We have the following extension of Lemma \ref{matchsigma} to this setting.
\begin{lem}
Take $n$, $m$, $T$, $S$, and $\tau$ as in Definition \ref{tausigma}, and fix $\sigma\in\mathcal{S}_{n,m}$. Set $R:=R^{\sigma}_{\tau}(T,S)\subseteq\mathbb{N}_{n+m-1}$ via Definition \ref{altshuff}, and write $\psi$ for the element $\psi^{\sigma}(\rho,\varphi) \in S_{n+m}$ as in Lemma \ref{matchsigma}. Then $(\eta,\kappa)\mapsto\eta\kappa$ yields a bijection from $\{(\eta,\kappa)\in\mathcal{M}_{n,T}^{\rho}\times\mathcal{M}_{m,S}^{\varphi}\;|\;\sigma_{\tau}(\eta,\kappa)=\sigma\}$ and $\mathcal{M}_{n+m,R}^{\psi}$ from Lemma \ref{sumFnTNC}. \label{sigmamatch}
\end{lem}

\begin{proof}
As in the proof of Lemma \ref{matchsigma}, we need to show that when $\theta=\eta\kappa$ the condition $\theta\in\mathcal{M}_{n+m,R}^{\psi}$ is equivalent to $\eta\in\mathcal{M}_{n,T}^{\rho}$, $\kappa\in\mathcal{M}_{m,S}^{\varphi}$, and $\sigma_{\tau}(\eta,\kappa)=\sigma$. When the former holds, and we write $\rho(j)=\psi(k)$ and $\rho(j+1)=\psi(h)$ for $h>k$, and $\eta_{\rho(j)}=\theta_{\psi(k)}\leq\theta_{\psi(h)}=\eta_{\rho(j+1)}$. Equality can occur if the entries between $k$ and $h-1$ are not in $R^{\sigma}_{\tau}(T,S)$, namely they are not descents of $w^{\sigma}(\tau)$ via Definition \ref{altshuff}.

But this implies that the sequence starting with $w_{k}=\tau_{j}$ and ending with $w_{h}=\tau_{j+1}$, and containing in the middle $h-k-1$ entries of $\tau$ with indices that are larger than $n$ (arising from the r's in $\sigma$ between the $j$th l and the next one), is strictly increasing. But this means that $\tau_{j}<\tau_{j+1}$, namely $j$ is not a descent of $\tau_{|n}$, hence does not lie in $T$. This establishes that $\eta\in\mathcal{M}_{n,T}^{\rho}$, with an analogous argument (with $\varphi+n$) yielding that $\kappa\in\mathcal{M}_{m,S}^{\varphi}$.

For proving that $\sigma_{\tau}(\eta,\kappa)$ from Definition \ref{tausigma} is the desired element $\sigma$, we work by the usual induction, assuming that we verified it for the first $t=a+b$ entries (with $a<n$ and $b<m$), and consider the next one. We again let $c>a+b+1$ be the minimal index for a location of $\sigma$ containing the opposite symbol from that in the location $t+1$. We then have $\theta_{\psi(a+b+1)}\leq\theta_{\psi(c)}$, with the indices being $\rho(a+1)$ and $\varphi(b+1)$, so that where the values compared are $\eta_{\rho(a+1)}$ and $\kappa_{\varphi(b+1)}$.

Assuming that the inequality is strict, when the location $t+1$ of $\sigma$ is an l it becomes $\eta_{\rho(a+1)}=\theta_{\psi(a+b+1)}<\theta_{\psi(c)}=\kappa_{\varphi(b+1)}$ and we have an l in that location of $\sigma_{\tau}(\eta,\kappa)$, and $\kappa_{\varphi(b+1)}=\theta_{\psi(a+b+1)}<\theta_{\psi(c)}=\eta_{\rho(a+1)}$ yields the same with an r as in the proof of Lemma \ref{matchsigma}. An equality implies that the numbers between $a+b+1$ and $c+1$ are not in $R$ hence not descents of $w^{\sigma}(\tau)$, and in particular its $c$th entry is larger than its $(a+b+1)$st entry.

But Definition \ref{altshuff} and the choices of our indices shows that these entries of $w^{\sigma}(\tau)$ are $\tau_{a+1}$ and $\tau_{b+1+n}$. If the location in question of $\sigma$ is an l (resp. r), then the inequality in question is $\tau_{a+1}<\tau_{n+b+1}$ (resp. $\tau_{a+1}>\tau_{n+b+1}$), and then, since we work under the assumption that $\eta_{\rho(a+1)}=\kappa_{\varphi(b+1)}$, Definition \ref{tausigma} shows that our location in $\sigma_{\tau}(\eta,\kappa)$ is also an l (resp. r). This induction, and the usual argument for $a=n$ or $b=m$ to land in $\mathcal{S}_{n,m}$, shows that $\sigma_{\tau}(\eta,\kappa)=\sigma$ as desired.

Conversely, the argument proving the other direction in Lemma \ref{matchsigma} still holds, with non-strict inequalities, yielding the inequality $\theta_{\psi(t)}\leq\theta_{\psi(t+1)}$ for every $t$. The strict inequality in case $t \in R$ in the ll and rr cases also extends verbatim to our case. It thus remains to show that in the remaining two cases, lr and rl, the situation in which $t \in R$ also implies a strict inequality.

Now, the former case occurs when $\theta_{\psi(t)}=\eta_{\rho(a)}\leq\kappa_{\varphi(b+1)}=\theta_{\psi(t+1)}$, and note that a strict inequality indeed yields, via Definition \ref{tausigma}, that the $t$th location of $\sigma=\sigma_{\tau}(\eta,\kappa)$ has to be an l. However, in the case of an equality this is an l if $\tau_{a}<\tau_{n+b+1}$, while with the opposite inequality $\tau_{a}<\tau_{n+b+1}$, which corresponds to the fact that $t \in R$, that location would have been an r, and it is not. The latter case works the same with $\theta_{\psi(t)}=\kappa_{\varphi(b)}\leq\eta_{\rho(a+1)}=\theta_{\psi(t+1)}$, the r in the $t$th location, and the comparisons between $\tau_{n+b}$ and $\tau_{a}$.

We have thus established that $\theta_{\psi(t)}<\theta_{\psi(t+1)}$ wherever $t \in R$, so that $\mathcal{M}_{n+m,R}^{\psi}$, as required. This completes the proof of the lemma.
\end{proof}

\medskip

Lemma \ref{sigmamatch} (which again has the natural extension to the case where $n$ or $m$ vanish, by the extension of Definition \ref{altshuff}) allows us to obtain the following generalization of Proposition \ref{prodF}.
\begin{prop}
Fix $\tau\in\mathcal{S}_{n,m}$ and take $\alpha \vDash n$, $\beta \vDash m$, $\rho \in S_{n}$, and $\varphi \in S_{m}$ as in Proposition \ref{prodF}. Then in $\mathtt{NCQSym}$ we have $\mathrm{F}_{\alpha}^{\rho}\mathrm{F}_{\beta}^{\varphi}=\sum_{\sigma\in\mathcal{S}_{n,m}}\mathrm{F}_{\gamma^{\sigma}_{\tau}(\alpha,\beta)}^{\psi^{\sigma}(\rho,\varphi)}$. \label{prodFtau}
\end{prop}
Indeed, we consider the proof of Proposition \ref{prodF}, and note the extension of Definition \ref{altshuff} to the case of vanishing $n$ or $m$. Otherwise, the same argument from that proof, but in which we decompose the pairs $(\eta,\kappa)\in\mathcal{M}_{n,T}^{\rho}\times\mathcal{M}_{m,S}^{\varphi}$ according to the value of $\sigma_{\tau}(\eta,\kappa)$ from Definition \ref{tausigma} instead of $\sigma(\eta,\kappa)$ from Definition \ref{defsigma}, and apply Lemma \ref{sigmamatch} in the place of Lemma \ref{matchsigma}, establishes Proposition \ref{prodFtau}. Note that in the latter two results, the choice of $\tau$ affects only the composition $\gamma^{\sigma}_{\tau}(\alpha,\beta) \vDash n+m$, and not the permutation $\psi^{\sigma}(\rho,\varphi) \in S_{n+m}$.

\begin{ex}
If we take the product $(\mathrm{F}_{1}^{1})^{2}$ from Example \ref{exprod} and Remark \ref{expsumind}, but now with $\tau=21 \in S_{2}$, then the expression from Proposition \ref{prodFtau} would be, via Example \ref{shufftau}, the sum of $\mathrm{F}_{11}^{12}$ from the element lr of $\mathcal{S}_{1,1}$, and $\mathrm{F}_{2}^{21}$ obtained via rl. Hence the monomials $x_{i}x_{j}$ in which $i=j$ are now considered as part of $\mathrm{F}_{2}^{21}$ and rl with this $\tau$, rather than $\mathrm{F}_{2}^{12}$ in that example (with the trivial $\tau$). With the non-trivial values 231 and 321 of $\tau$ from that example, we get $\mathrm{F}_{2}^{12}\mathrm{F}_{1}^{1}=\mathrm{F}_{12}^{132}+\mathrm{F}_{21}^{123}+\mathrm{F}_{3}^{312}$ and $\mathrm{F}_{11}^{12}\mathrm{F}_{1}^{1}=\mathrm{F}_{12}^{132}+\mathrm{F}_{111}^{123}+\mathrm{F}_{21}^{312}$, while letting $\tau$ take the values 132 and 312 produces $\mathrm{F}_{2}^{12}\mathrm{F}_{1}^{1}=\mathrm{F}_{3}^{132}+\mathrm{F}_{21}^{123}+\mathrm{F}_{12}^{312}$ and $\mathrm{F}_{11}^{12}\mathrm{F}_{1}^{1}=\mathrm{F}_{111}^{132}+\mathrm{F}_{12}^{123}+\mathrm{F}_{21}^{312}$ respectively. As in Example \ref{exprod}, taking the former multiplier to be with the superscript 21 instead of 12 will give the same expressions, but in which the respective superscipts will be 231, 213, and 321. \label{altprod}
\end{ex}
The expressions for $n=1$ and $m=2$ mentioned after Example \ref{shufftau} can complement Example \ref{altprod} to give alternative formulae for the products $\mathrm{F}_{1}^{1}\mathrm{F}_{2}^{12}$ and $\mathrm{F}_{1}^{1}\mathrm{F}_{11}^{12}$ (as well as $\mathrm{F}_{1}^{1}\mathrm{F}_{2}^{21}$ and $\mathrm{F}_{1}^{1}\mathrm{F}_{11}^{21}$) considered in Example \ref{exprod}.
\begin{ex}
Let $n=m=3$, $\rho=312 \in S_{3}$, $\alpha=21\vDash3$, $T:=\{2\}$, $\varphi=231$, $\beta=12\vDash3$, and $S:=\{1\}$ be as in Example \ref{exsigma}, so that for $\eta\in\mathcal{M}_{3,\{2\}}^{312}$ and $\kappa\in\mathcal{M}_{3,\{1\}}^{231}$ we have $x^{\eta}=x_{i_{1}}x_{i_{2}}x_{i_{3}}$ and $x^{\kappa}=x_{j_{1}}x_{j_{2}}x_{j_{3}}$ with the inequalities $i_{3} \leq i_{1}<i_{2}$ and $j_{2}<j_{3} \leq j_{1}$. When $\sigma\in\mathcal{S}_{3,3}$ is lrrllr, the trivial values for $\tau \in S_{6}$, which are 132546, 231546, 132645, or 231645 via Remark \ref{trivtau}, we saw the conditions for $\sigma_{\tau}(\eta,\kappa)=\sigma(\eta,\kappa)$ to be $\sigma$ in that example. Definition \ref{altshuff} implies that 2 and 4 are in $R^{\sigma}_{\tau}(T,S)$ with that $\sigma$ for every choice of $\tau$ (for which the descent sets of $\tau_{|3}$ and $\tau_{3|}$ are $T$ and $S$ respectively), and with choices like 152436, 241635, 351426, and others (44 options for $\tau$), this set contains only those elements and Lemma \ref{sigmamatch} compares the concatenations $\eta\kappa$ for which $\sigma_{\tau}(\eta,\kappa)=\sigma$ as $\mathcal{M}_{6,\{2,4\}}^{356124}$. The set $\{2,3,4\}$ is obtained only from the 4 trivial options of $\tau$, while $\{1,2,3,4\}$, $\{2,3,4,5\}$, and $\{1,2,3,4,5\}$ are not attained at all. The 6 choices 365214, 563412, 465213, 564213, 465312, and 564312 yield $\{1,2,4,5\}$, while there are 13 options (like 451362 and 561423) for which the concatenations produce $\mathcal{M}_{6,\{1,2,4\}}^{356124}$, and 13 permutations (including 154623 and 165324) yielding $\mathcal{M}_{6,\{1,2,4\}}^{356124}$. \label{extausigma}
\end{ex}
Note that with the parameters from Example \ref{extausigma}, even the minimal set $\{2,4\}$ forces, with that value of $\sigma$, the inequalities $i_{3}<i_{1}<i_{2}$ for $\eta$ and $j_{2}<j_{3}<j_{1}$ for $\kappa$, as in Remark \ref{strongineq}. For different values of $\sigma$ one or both of these can still be weakened. Note that the value $\psi^{\sigma}(\rho,\varphi)=356124$ in that example is independent of $\tau$, as the formula from Lemma \ref{sigmamatch} asserts.

\subsection{Matching Expressions and the Proofs}

We saw that the parameter $\tau \in S_{m+n}$ from Definition \ref{altshuff} encodes, via the descent sets of its two parts, the subsets $T\subseteq\mathbb{N}_{n}$ and $S\subseteq\mathbb{N}_{m}$ (or the compositions $\alpha \vDash n$ and $\beta \vDash m$). In order to match this property with the fact that in our main theorems we relate a composition $\alpha$ with its transpositions $\alpha^{t}$ from Definition \ref{invdef}, we make the following observation, already noted in \cite{[Z1]} and others.
\begin{lem}
For $w \in S_{n}$ we have the three equalities $\operatorname{Des}(w^{0}_{n}w)=\operatorname{Des}(w)^{c}$, $\operatorname{Des}(ww^{0}_{n})=\operatorname{Des}(w)^{t}$, and $\operatorname{Des}(w^{0}_{n}ww^{0}_{n})=\operatorname{Des}(w)^{r}$. \label{opersDes}
\end{lem}
Note that the word $ww^{0}_{n}$ is typically denoted by $w^{r}$, as it is the word obtained by writing the entries of $w$ in the opposite order. However, based on Lemma \ref{opersDes}, we write $w^{t}$ for this word. Note that this expression extends for every word, not necessarily arising from a one-line notation from $S_{n}$, and indeed we will use it for parts of words from $S_{n}$, without the standardization (which does not affect the validity of this part of Lemma \ref{opersDes}).

\medskip

For proving our main theorems, we will apply Corollary \ref{detS}, or more precisely the second convention of it which appears in Remark \ref{firstS}. To do this, we will apply the co-multiplication $\Delta$ to an appropriate fundamental quasi-symmetric function, and then we need to establish the vanishing of a linear combination of products of two fundamental quasi-symmetric functions. We will do so by canceling the summands obtained from Proposition \ref{prodFtau} in pairs, for which we will set some notation and establish a matching.
\begin{defn}
Take $\tau \in S_{n}$, with $n>0$, and a number $0 \leq r \leq n$.
\begin{enumerate}[$(i)$]
\item Let $\tau_{\leq r}$ and $\tau_{>r}$ be the words consisting of the first $r$ entries and the last $n-r$ entries of $\tau$ respectively. Thus $\tau$ is the concatenation $\tau_{\leq r}\tau_{>r}$.
\item Write $\tau^{(r)}$ for the concatenation $\tau_{\leq r}\tau_{>r}^{t}$, after transposing the second word.
\end{enumerate} \label{taurdef}
\end{defn}
The cases $r=n$ and $r=0$ in Definition \ref{taurdef} produce $\tau_{\leq n}=\tau_{>0}=\tau$ while $\tau_{\leq 0}$ and $\tau_{>n}$ are empty, so that $\tau^{(n)}=\tau$ and $\tau^{(0)}=\tau^{t}$. Note that $\tau_{|r}$ and $\tau_{r|}$ from Definition \ref{setwithr} are the standardizations of our $\tau_{\leq r}$ and $\tau_{>r}$ respectively. The expressions $\tau_{\leq r}$ and $\tau_{>r}$ are denoted in \cite{[LM]} by $\tau_{[r]}$ and $\tau^{[r]}$ respectively.

We deduce from that definition via Lemma \ref{opersDes} the following consequence.
\begin{cor}
If $\alpha:=\operatorname{comp}_{n}\operatorname{Des}(\tau)$ for $\tau \in S_{n}$, then for every $0 \leq r \leq n$, the sequence $\tau^{(r)}$ from Definition \ref{taurdef} is an element of $S_{n}$ which produces an expression for the product $\mathrm{F}_{\alpha_{|r}}^{\rho_{|r}}\mathrm{F}_{\alpha_{r|}^{t}}^{\rho_{r|}}$ for any $\rho \in S_{n}$. \label{tauforr}
\end{cor}

\begin{proof}
Write $T:=\operatorname{Des}(\tau)$, so that $\alpha=\operatorname{comp}_{n}T$. Then if $0<r<n$ then it is clear that the descent sets of $\tau_{\leq r}$ and $\tau_{>r}$, which are those of $\tau_{|r}$ and $\tau_{r|}$ respectively, are $T_{|r}$ and $T_{r|}$ respectively, whose associated compositions of $r$ and $n-r$ are $\alpha_{|r}$ and $\alpha_{r|}$ respectively.

After applying transposition to $\tau_{>r}$, as we keep the set of entries, the result is again a permutation in $S_{n}$ (in one-line notation as usual). As the standardization of the resulting word is clearly $\tau_{r|}^{t}$, whose descent set is $T_{r|}^{t}=\operatorname{comp}_{n-r}\alpha_{r|}^{t}$ by Lemma \ref{opersDes}, the result follows.

As the cases $r=n$, which involves just $\mathrm{F}_{\rho}^{\alpha}\cdot1$ with $\tau^{(n)}=\tau$, and $r=0$, with which the product is $1\cdot\mathrm{F}_{\rho}^{\alpha^{t}}$ and the permutation $\tau^{(0)}$ is $\tau^{t}$, the formula holds in these cases as well. This proves the corollary.
\end{proof}

\begin{ex}
We take $n=6$, $T=\{2,3\}$, and $\alpha=213\vDash6$ as in Example \ref{excomp}, and then the element $\tau=264135 \in S_{6}$ has $T$ as its descent set. The reverse word $\tau^{t}=531462$ has descent set $T^{t}=\{1,2,5\}$ (as Lemma \ref{opersDes} asserts), and $\alpha^{t}=1131$. For the values $1 \leq r\leq5$, we have $\tau_{\leq5}=26413$, $\tau_{>5}=5$, $\tau_{\leq4}=2641$, $\tau_{>4}=35$, $\tau_{\leq3}=264$, $\tau_{>3}=135$, $\tau_{\leq2}=26$, $\tau_{>2}=4135$, $\tau_{\leq1}=2$, and $\tau_{>1}=64135$ for the partial words from Definition \ref{taurdef}, whose descent sets (or those of their standardizations) are those appearing in that example. Transposing the words $\tau_{>r}$ and concatenating them with $\tau_{\leq r}$ yields the words $\tau^{(5)}=264135$, $\tau^{(4)}=264153$, $\tau^{(3)}=264531$, $\tau^{(2)}=265314$,  and $\tau^{(1)}=253146$, and one can verify Corollary \ref{tauforr} for these words. \label{extaur}
\end{ex}
In fact, Definition \ref{taurdef} always yields $\tau^{(n-1)}=\tau$, as seen in Example \ref{extaur}.

\medskip

The desired cancelation will be obtained using the following observation.
\begin{lem}
Consider $\tau \in S_{n}$, $0 \leq r \leq n-1$, and $\sigma\in\mathcal{S}_{r,n-r}$ that ends with an r, and let $\tilde{\sigma}$ be the element of $\mathcal{S}_{r+1,n-r-1}$ that is obtained by replacing the last r in $\sigma$ by an l. Then we have the equality $w^{\sigma}(\tau^{(r)})=w^{\tilde{\sigma}}(\tau^{(r+1)})$. \label{forcanc}
\end{lem}

\begin{proof}
Definitions \ref{altshuff} and \ref{taurdef} imply that $w^{\sigma}(\tau^{(r)})$ is obtained by putting the entries of $\tau_{\leq r}$ in the $r$ locations marked with l in $\sigma$, and those of $\tau_{>r}^{t}$ in the $n-r$ locations containing r there. Similarly, for $w^{\tilde{\sigma}}(\tau^{(r+1)})$ we do the same with $\tau_{\leq r+1}$ and the $r+1$ l's of $\tilde{\sigma}$ and using $\tau_{>r+1}^{t}$ for the $n-r-1$ places of r in that sequence.

But the latter definition implies that $\tau_{\leq r+1}$ is the concatenation of $\tau_{\leq r}$ and the single letter $\tau_{r+1}$, while $\tau_{>r}$ is the concatenation of the letter $\tau_{r+1}$ and $\tau_{>r+1}$. From the latter it follows that $\tau_{>r}^{t}$ is the concatenation of $\tau_{>r+1}^{t}$ and $\tau_{r+1}$.

Thus the $r$ locations of l in $\sigma$, which are the first $r$ locations of l in $\tilde{\sigma}$, contain the entries of $\tau_{\leq r}$ in $w^{\sigma}(\tau^{(r)})$ (according to their order), and also in $w^{\tilde{\sigma}}(\tau^{(r+1)})$ because the first $r$ entries of $\tau_{\leq r+1}$ is $\tau_{\leq r}$. Moreover, the first $n-r-1$ locations of r in $\sigma$ are all the locations of r in $\tilde{\sigma}$, and in $w^{\tilde{\sigma}}(\tau^{(r+1)})$ we put the entries of $\tau_{>r+1}^{t}$ in them (with their order), as we do in $w^{\sigma}(\tau^{(r)})$ because the first $n-r-1$ letters of $\tau_{>r}^{t}$ produce $\tau_{>r+1}^{t}$.

This shows that all the letters of $w^{\sigma}(\tau^{(r)})$ and of $w^{\tilde{\sigma}}(\tau^{(r+1)})$ but the last coincide. There are two ways to see that the last letter in both has to be $\tau_{r+1}$, thus completing the equality. First, the last r in $\sigma$ shows up at the end and this letter is the last one in $\tau_{>r}^{t}$, and last l in $\tilde{\sigma}$ is also at the end and our letter is the last one of $\tau_{\leq r+1}$. Alternatively, both $w^{\sigma}(\tau^{(r)})$ and of $w^{\tilde{\sigma}}(\tau^{(r+1)})$ are permutations in $S_{n}$ (in one-line notation), so that the last entry is determined by the other ones to complete the content of both words to be $\mathbb{N}_{n}$. This completes the proof of the lemma.
\end{proof}

\begin{ex}
Take $\tau$ (and $n$) as in Example \ref{extaur}, and consider the case $r=2$. Then there are 10 choices of $\sigma\in\mathcal{S}_{2,4}$ that end with an r, which correspond to the 10 elements $\tilde{\sigma}\in\mathcal{S}_{3,3}$ ending with an l. The trivial one llrrrr for the former produces $w^{\sigma}(\tau^{(2)})=\tau^{(2)}=265314$, which corresponds to llrrrl for $\tilde{\sigma}$, with which $w^{\tilde{\sigma}}(\tau^{(3)})$ with $\tau^{(3)}=264531$ produces the same word. As another choice, take rlrlrr for $\sigma$, associated with which is rlrlrl for $\tilde{\sigma}$, and then both words from Lemma \ref{forcanc} are 523614. \label{cancex}
\end{ex}

\begin{rmk}
The number 10 of choices in Example \ref{cancex} is the size $\frac{5!}{2!3!}$ of $\mathcal{S}_{2,3}$, adding an r at the end of which produces all relevant elements $\sigma\in\mathcal{S}_{2,4}$, and when we add l at the end, we obtain the corresponding $\tilde{\sigma}$. \label{numofpairs}
\end{rmk}
The consequences of Lemma \ref{forcanc} that will be used are the following ones, in which expressions like $\alpha_{r|}^{t}$ are to be understood as $(\alpha_{r|})^{t}$ (first apply the operation from Definition \ref{cutatr}, and then reverse), and not as $(\alpha^{t})_{r|}$ (which is different in general).
\begin{cor}
Take $\tau$, $r$, $\sigma$, and $\tilde{\sigma}$ as in Lemma \ref{forcanc}, and set $T:=\operatorname{Des}(\tau)$ and $\alpha:=\operatorname{comp}_{n}T \vDash n$. Then both $\gamma^{\sigma}_{\tau^{(r)}}(\alpha_{|r},\alpha_{r|}^{t})$ and $\gamma^{\tilde{\sigma}}_{\tau^{(r+1)}}(\alpha_{|r+1},\alpha_{r+1|}^{t})$ are defined and they are equal, and we have $\psi^{\sigma}(\operatorname{Id}_{r},w^{0}_{n-r})=\psi^{\tilde{\sigma}}(\operatorname{Id}_{r+1},w^{0}_{n-r-1})$. \label{comppars}
\end{cor}

\begin{proof}
Corollary \ref{tauforr} implies that the equalities $\operatorname{comp}_{r}\operatorname{Des}(\tau^{(r)}_{r|})=\alpha_{|r}$ and $\operatorname{comp}_{n-r}\operatorname{Des}(\tau^{(r)}_{r|})=\alpha_{|r}^{t}$ hold, and the same for $r+1$, which makes both sides defined (with the usual completions in case $r=0$ or $r+1=n$). As we get $w^{\sigma}(\tau^{(r)})=w^{\tilde{\sigma}}(\tau^{(r+1)})$ from Lemma \ref{forcanc}, applying Definition \ref{altshuff} shows that $R^{\sigma}_{\tau^{(r)}}(T_{|r},T_{r|})=R^{\tilde{\sigma}}_{\tau^{(r+1)}}(T_{|r+1},T_{r+1|})$ as these are the descent sets of both sides (at least when $0<r<n-1$), and we then get the desired equality from that definition (also in the external cases of $r$). This establishes the first assertion.

We now consider the case where $\tau=\operatorname{Id}_{n}$. Then $\tau^{(r)}$ is the concatenation of $\operatorname{Id}_{r}$ and $w^{0}_{n-r}+r$, and similarly for $\tau^{(r+1)}$. But then Remark \ref{trivtau} implies that $\psi^{\sigma}(\operatorname{Id}_{r},w^{0}_{n-r})$ is the word $w^{\sigma}(\tau^{(r)})$, and similarly $\psi^{\tilde{\sigma}}(\operatorname{Id}_{r+1},w^{0}_{n-r-1})$ is the same as $w^{\tilde{\sigma}}(\tau^{(r+1)})$. Applying Lemma \ref{forcanc} once again, the second assertion follows as well. This proves the corollary.
\end{proof}

\medskip

We can now prove all of our main results together.
\begin{proof}[Proof of Theorems \ref{SQSymF}, \ref{SQSymqF}, and \ref{SNCQSymF}]
For a unified proof, we write $\mathrm{F}_{\alpha}$ to denote either $F_{\alpha}\in\mathtt{QSym}$, $F_{\alpha}^{(q)}\in\mathtt{QSym}^{(q)}$, or $\mathrm{F}_{\alpha}^{\operatorname{Id}_{n}}\in\mathtt{NCQSym}$. Propositions \ref{coprodQSym}, \ref{coprodQSymq}, and \ref{coprodNCQSym} show that we have the equality $\Delta(\mathrm{F}_{\alpha})=\sum_{r=0}^{n}\mathrm{F}_{\alpha_{|r}}\otimes\mathrm{F}_{\alpha_{r|}}$ in all three cases (in the latter one we also use the fact that when $\rho=\operatorname{Id}_{n}$ we get $\rho_{|r}=\operatorname{Id}_{r}$ and $\rho_{r|}=\operatorname{Id}_{n-r}$). Thus $\{\mathrm{F}_{\alpha}\}_{\alpha}$ satisfy the condition for $\{g_{\alpha}\}_{\alpha}$ in Corollary \ref{detS}.

The element playing the role of $\tilde{g}_{\alpha}$ from this corollary is, when $\alpha \vDash n$ (which is the same as $n(\alpha)=n$ in the notation there), is $(-1)^{n}F_{\alpha^{t}}$ when we work with $\mathtt{QSym}$, $(-1)^{n}q^{\operatorname{inv}(\alpha^{t})}\widetilde{F}_{\alpha^{t}}^{(q)}$ for $\mathtt{QSym}^{(q)}$, and $(-1)^{n}\mathrm{F}_{\alpha^{t}}^{w^{0}_{n}}$ when the algebra is $\mathtt{NCQSym}$. We write this expression as $(-1)^{n}\widetilde{\mathrm{F}}_{\alpha^{t}}$ for our unified notation.

Working with the convention from Remark \ref{firstS}, and recalling that $\tilde{g}_{\alpha_{r|}}$ is $(-1)^{n-r}\widetilde{\mathrm{F}}_{\alpha_{r|}^{t}}$ because $\alpha_{r|}$ and $\alpha_{r|}^{t}$ are compositions of $n-r$, the result will follow from Corollary \ref{detS} once we establish the equality $\sum_{r=0}^{n}(-1)^{n-r}\mathrm{F}_{\alpha_{|r}}\widetilde{\mathrm{F}}_{\alpha_{r|}^{t}}=0$. Noting that the maps of algebras that are given in Propositions \ref{projQSym} and \ref{projQSymq} send $\mathrm{F}_{\beta}^{\operatorname{Id}_{m}}$ to $F_{\beta}\in\mathtt{QSym}$ and to $F_{\beta}^{(q)}\in\mathtt{QSym}^{(q)}$ for $\beta \vDash m$, as well as $\mathrm{F}_{\beta}^{w^{0}_{m}}$ to $F_{\beta}$ inside $\mathtt{QSym}$ and to $q^{\operatorname{inv}(\beta)}\widetilde{F}_{\beta}^{(q)}$ in the $q$-Hopf algebra $\mathtt{QSym}^{(q)}$, the vanishing result for $\mathtt{NCQSym}$ will imply those for the other two cases.

We thus need to prove the equality $\sum_{r=0}^{n}(-1)^{n-r}\mathrm{F}_{\alpha_{|r}}^{\operatorname{Id}_{r}}\mathrm{F}_{\alpha_{r|}^{t}}^{w^{0}_{n-r}}=0$ for every $\alpha \vDash n>0$ (when $n=0$ and $\alpha=\emptyset$, the equality from Corollary \ref{detS} is the trivial one $1\cdot1=1$). For this we write $\alpha$ as $\operatorname{comp}_{n}T$ for the appropriate $T\subseteq\mathbb{N}_{n-1}$, fix some $\tau \in S_{n}$ with descent set $T$, and consider $\rho:=\operatorname{Id}_{n}$. Recalling that $\rho_{|r}=\operatorname{Id}_{r}$ and $\rho_{r|}=\operatorname{Id}_{n-r}$ in this case, so that $\rho_{r|}^{t}=w^{0}_{n-r}$, we can apply Corollary \ref{tauforr}, and take $\tau^{(r)}$ from Definition \ref{taurdef} for evaluating the $r$th summand using Proposition \ref{prodFtau}.

The resulting expansion yields $\sum_{r=0}^{n}(-1)^{n-r}\sum_{\sigma\in\mathcal{S}_{r,n-r}}\mathrm{F}_{\gamma^{\sigma}_{\tau^{(r)}}(\alpha_{|r},\alpha_{r|})}^{\psi^{\sigma}(\operatorname{Id}_{r},w^{0}_{n-r})}$. For every $0 \leq r<n$, we write $\mathcal{S}_{r,n-r}^{+}$ for the set of elements of $\mathcal{S}_{r,n-r}$ that end with an r, and let $\mathcal{S}_{r,n-r}^{-}$ denote the complement. As an element of $\mathcal{S}_{r,n-r}^{-}$ ends with an l, we can replace this last l by an r, and get an element $\sigma\in\mathcal{S}_{r-1,n-r+1}^{+}$, and then our element of $\mathcal{S}_{r,n-r}^{-}$ is $\tilde{\sigma}$ in the notation from Lemma \ref{forcanc} (with the value of $r$ moved by 1). Since the signs for $r$ and for $r-1$ are inverses, our expression becomes $\sum_{r=0}^{n-1}(-1)^{n-r}\sum_{\sigma\in\mathcal{S}_{r,n-r}^{+}}[\mathrm{F}_{\gamma^{\sigma}_{\tau^{(r)}}(\alpha_{|r},\alpha_{r|}^{t})}^{\psi^{\sigma}(\operatorname{Id}_{r},w^{0}_{n-r})}- \mathrm{F}_{\gamma^{\tilde{\sigma}}_{\tau^{(r+1)}}(\alpha_{|r+1},\alpha_{r+1|}^{t})}^{\psi^{\tilde{\sigma}}(\operatorname{Id}_{r+1},w^{0}_{n-r-1})}]$.

So fix $0 \leq r \leq n-1$ and $\sigma\in\mathcal{S}_{r,n-r}^{+}$, and we invoke Corollary \ref{comppars}, which tells us that in the two expressions in the corresponding difference, the compositions in the two subscripts coincide, as do the permutations in the two superscripts. Therefore the terms cancel out in pairs in our expression, which establishes its vanishing, as desired. This completes the proofs of all three theorems.
\end{proof}
Given $0 \leq r \leq n-1$, the number of pairs of $\sigma\in\mathcal{S}_{r,n-r}^{+}$ and $\tilde{\sigma}\in\mathcal{S}_{r+1,n-r-1}^{-}$ in the proofs of Theorems \ref{SQSymF}, \ref{SQSymqF}, and \ref{SNCQSymF} is $\frac{(n-1)!}{r!(n-1-r)!}$. Indeed, this is the size of the set $\mathcal{S}_{r,n-r-1}$, elements of which produce $\mathcal{S}_{r,n-r}^{+}$ by adding an r at the end, and the corresponding counterparts from $\mathcal{S}_{r+1,n-r-1}^{-}$ by adding an l there (as is visible in Example \ref{cancex} via Remark \ref{numofpairs}).

We remark that for Theorem \ref{SQSymF} and $\mathtt{QSym}$, one can describe an algorithm for matching the pairs where all the products are taken with trivial $\tau$'s and Proposition \ref{prodF} rather than \ref{prodFtau}, but the current proof works for the more complicated Hopf algebras.

\medskip

We now present some examples for how these proofs work.
\begin{ex}
When $n=1$, the primitive element $\mathrm{F}_{1}^{1}$ produces $\mathrm{F}_{1}^{1}-\mathrm{F}_{1}^{1}=0$. If $n=2$, then there are the compositions 2 and 11, for which we must take $\tau$ to be 12 and 21 respectively, for both of which the term associated with $r=1$ is $(\mathrm{F}_{1}^{1})^{2}$. Expanding it using $\tau^{(1)}=\tau$, this gives $\mathrm{F}_{2}^{12}-(\mathrm{F}_{2}^{12}+\mathrm{F}_{11}^{21})+\mathrm{F}_{11}^{21}=0$ in the former case and $\mathrm{F}_{11}^{12}-(\mathrm{F}_{11}^{12}+\mathrm{F}_{2}^{21})+\mathrm{F}_{2}^{21}=0$ in the latter, with the pair consisting of the first two expressions cancel, as does the pair of the last two. Turning to $n=3$, the compositions 3 and 111 require the choices 123 and 321 for $\tau$ respectively, with which we get $\mathrm{F}_{3}^{123}-(\mathrm{F}_{3}^{123}+\mathrm{F}_{21}^{132}+\mathrm{F}_{12}^{312})+(\mathrm{F}_{21}^{132}+\mathrm{F}_{12}^{312}+\mathrm{F}_{111}^{321})-\mathrm{F}_{111}^{321}=0$ and $\mathrm{F}_{111}^{123}-(\mathrm{F}_{111}^{123}+\mathrm{F}_{12}^{132}+\mathrm{F}_{21}^{312})+(\mathrm{F}_{12}^{132}+\mathrm{F}_{21}^{312}+\mathrm{F}_{3}^{321})-\mathrm{F}_{3}^{321}=0$ when the elements of $\mathcal{S}_{2,1}$ are always ordered as llr, lrl, rll, and those of $\mathcal{S}_{1,2}$ are lrr, rlr, rrl. For the other compositions there is more than one choice of $\tau$, as both 132 and 231 have descent set $\{2\}$, while the descent set of both 213 and 312 is $\{1\}$. The first two produce the equalities $\mathrm{F}_{21}^{123}-(\mathrm{F}_{21}^{123}+\mathrm{F}_{3}^{132}+\mathrm{F}_{12}^{312})+(\mathrm{F}_{3}^{132}+\mathrm{F}_{12}^{312}+\mathrm{F}_{21}^{321})-\mathrm{F}_{21}^{321}=0$ and $\mathrm{F}_{21}^{123}-(\mathrm{F}_{21}^{123}+\mathrm{F}_{12}^{132}+\mathrm{F}_{3}^{312})+(\mathrm{F}_{12}^{132}+\mathrm{F}_{3}^{312}+\mathrm{F}_{21}^{321})-\mathrm{F}_{21}^{321}=0$ respectively, and we get $\mathrm{F}_{12}^{123}-(\mathrm{F}_{12}^{123}+\mathrm{F}_{21}^{132}+\mathrm{F}_{111}^{312})+(\mathrm{F}_{21}^{132}+\mathrm{F}_{111}^{312}+\mathrm{F}_{12}^{321})-\mathrm{F}_{12}^{321}=0$ and $\mathrm{F}_{12}^{123}-(\mathrm{F}_{12}^{123}+\mathrm{F}_{111}^{132}+\mathrm{F}_{21}^{312})+(\mathrm{F}_{111}^{132}+\mathrm{F}_{21}^{312}+\mathrm{F}_{12}^{321})-\mathrm{F}_{12}^{321}=0$ from the latter two (all in these orders on $\mathcal{S}_{2,1}$ and $\mathcal{S}_{1,2}$). \label{cancsmall}
\end{ex}

\medskip

Two special cases reduce to nicer formulae.
\begin{rmk}
As we saw for the small values of $n$ in Example \ref{cancsmall}, the composition $\alpha=n \vDash n$ is associated only with $\tau=\operatorname{Id}_{n}$, with which $\tau^{(r)}$ is always the trivial word (in the sense of Remark \ref{trivtau}) that is associated with $\operatorname{Id}_{r}$ and $w^{0}_{n-r}$. We saw this in the proof of Corollary \ref{comppars}, which also shows that $\gamma^{\sigma}_{\tau^{(r)}}(r,1^{n-r})=\gamma^{\sigma}(r,1^{n-r})=\psi^{\sigma}(\operatorname{Id}_{r},w^{0}_{n-r})$ for every $\sigma\in\mathcal{S}_{r,n-r}$ (note that for our $\alpha$ we have $\alpha_{|r}=r \vDash r$, $\alpha_{r|}=n-r \vDash n-r$, and thus $\alpha_{r|}^{t}=1^{n-r}$). Similarly, for $\alpha=1^{n} \vDash n$ (with $\alpha_{|r}=1^{r}$ and $\alpha_{r|}^{t}$ of length 1) we must have $\tau=w^{0}_{n}$, so that $\tau^{(r)}$ is always the opposite to the trivial case (as in Remark \ref{gammaconc}) for every $r$. In this case every $\sigma\in\mathcal{S}_{r,n-r}$ produces the equality $\gamma^{\sigma}_{\tau^{(r)}}(1^{r},n-r)=\psi^{\overline{\sigma}}(\operatorname{Id}_{n-r},w^{0}_{r})$ (note the of the indices inversion), where $\overline{\sigma}$ is the element of $\mathcal{S}_{n-r,r}$ that is obtained by inverting all the symbols of $\sigma$. \label{cancext}
\end{rmk}
We can now complete the proof of the remaining theorem.
\begin{proof}[Proof of Theorem \ref{SFalphaw0}]
We consider Proposition \ref{coprodNCQSym} with $\rho=w^{0}_{n}$, so that $\rho_{|r}=w^{0}_{r}$ and $\rho_{r|}=w^{0}_{n-r}$. We can thus again employ Corollary \ref{detS} (as it is stated, not in the convention from Remark \ref{firstS}), which would produce our result once we show that wherever $n>0$, the equalities $\sum_{r=0}^{n}(-1)^{r}\mathrm{F}_{r}^{\operatorname{Id}_{r}}\mathrm{F}_{1^{n-r}}^{w^{0}_{n-r}}=0$ and $\sum_{r=0}^{n}(-1)^{r}\mathrm{F}_{1^{r}}^{\operatorname{Id}_{r}}\mathrm{F}_{n-r}^{w^{0}_{n-r}}=0$ hold.

But Remark \ref{cancext} shows that when we consider the proof of Theorem \ref{SNCQSymF}, when $\ell(\alpha)=1$ and when $\alpha=1^{n}$, then we get precisely the desired equalities (up to the global sign $(-1)^{n}$). Hence the required equalities hold, and the desired antipode formulae follow. This proves the theorem.
\end{proof}
Indeed, in the smallest values of $n$, we need to show that the antipode of $\mathrm{F}_{1}^{1}$, $\mathrm{F}_{11}^{21}$, $\mathrm{F}_{2}^{21}$, $\mathrm{F}_{111}^{321}$, and $\mathrm{F}_{3}^{321}$, are $-\mathrm{F}_{1}^{1}$, $\mathrm{F}_{2}^{12}$, $\mathrm{F}_{11}^{12}$, $-\mathrm{F}_{3}^{123}$, and $-\mathrm{F}_{111}^{123}$ respectively. The equalities required for producing them are those appearing in Example \ref{cancsmall} for $\mathrm{F}_{1}^{1}$ (which is already covered in Theorem \ref{SNCQSymF}, and is easily verified by primitivity), $\mathrm{F}_{2}^{12}$, $\mathrm{F}_{11}^{12}$, $\mathrm{F}_{3}^{123}$, and $\mathrm{F}_{111}^{123}$ respectively (up to the sign $(-1)^{n}$).

\medskip

There are some symmetries in the vanishing equation arising from $\tau \in S_{n}$ in the proof of our main theorems, as hinted by the fact that we used the same equation for the proof of Theorem \ref{SFalphaw0}. To see them in full, we indicate how the proof using the convention from Corollary \ref{detS} (without Remark \ref{firstS}) works.
\begin{rmk}
For the vanishing of the expression $\sum_{r=0}^{n}S(\mathrm{F}_{\alpha_{|r}}^{\operatorname{Id}_{r}})\mathrm{F}_{\alpha_{r|}}^{\operatorname{Id}_{n-r}}$, we use the following analogues of Definition \ref{taurdef}, Corollary \ref{tauforr}, Lemma \ref{forcanc}, and Corollary \ref{comppars}. We write $\tau^{[r]}$ for the concatenation $\tau_{\leq r}^{t}\tau_{>r}$ (the transposition now on the first part), which will now be applicable for evaluating $\mathrm{F}_{\alpha_{|r}^{t}}^{\rho_{|r}}\mathrm{F}_{\alpha_{r|}}^{\rho_{r|}}$ in case $\alpha$ is $\operatorname{comp}_{n}\operatorname{Des}(\tau)$. If $1 \leq r \leq n$ and $\sigma\in\mathcal{S}_{r,n-r}$ \emph{begins} with an r, we denote by $\hat{\sigma}$ the element of $\sigma\in\mathcal{S}_{r-1,n-r+1}$ resulting from changing that first r to an l, and a similar argument proves that $w^{\sigma}(\tau^{[r]})=w^{\hat{\sigma}}(\tau^{[r-1]})$. We deduce, in this setting, that $\gamma^{\sigma}_{\tau^{[r]}}(\alpha_{|r}^{t},\alpha_{r|})$ and $\gamma^{\hat{\sigma}}_{\tau^{[r-1]}}(\alpha_{|r-1}^{t},\alpha_{r+1|}^{t})$ are both defined and equal, and that $\psi^{\sigma}(w^{0}_{r},\operatorname{Id}_{n-r})=\psi^{\hat{\sigma}}(w^{0}_{r-1},\operatorname{Id}_{n-r+1})$. By setting $\hat{\mathcal{S}}_{r,n-r}^{+}$ for $r\geq1$ the set elements of $\mathcal{S}_{r,n-r}$ beginning with an r, the expression $\sum_{r=0}^{n}(-1)^{r}\mathrm{F}_{\alpha_{|r}^{t}}^{w^{0}_{r}}\mathrm{F}_{\alpha_{r|}}^{\operatorname{Id}_{n-r}}$ becomes $\sum_{r=1}^{n}(-1)^{r} \sum_{\sigma\in\hat{\mathcal{S}}_{r,n-r}^{+}}[\mathrm{F}_{\gamma^{\sigma}_{\tau^{(r)}}(\alpha_{|r}^{t},\alpha_{r|})}^{\psi^{\sigma}(w^{0}_{r},\operatorname{Id}_{n-r})}- \mathrm{F}_{\gamma^{\hat{\sigma}}_{\tau^{(r+1)}}(\alpha_{|r+1},\alpha_{r+1|})}^{\psi^{\hat{\sigma}}(w^{0}_{r-1},\operatorname{Id}_{n-r+1})}]$, which for every $r$ produces the sum of $|\mathcal{S}_{r-1,n-r}|=\frac{(n-1)!}{(r-1)!(n-r)!}$ canceling pairs. \label{altproof}
\end{rmk}

Recalling Lemma \ref{opersDes} and the notation $w^{t}$, we will also write $w^{c}$ and $w^{r}$ for $w^{0}_{n}w$ and $w^{0}_{n}ww^{0}_{n}$ respectively. Combining that lemma with Lemma \ref{relsinv}, we obtain the following symmetry properties.
\begin{prop}
If $\tau \in S_{n}$ is such that $\operatorname{comp}_{n}\operatorname{Des}(\tau)=\alpha \vDash n$, $0 \leq r \leq n$, and $\sigma\in\mathcal{S}_{r,n-r}$, then we have $w^{\sigma}(\tau^{(r)})^{c}=w^{\sigma}\big((\tau^{c})^{(r)})\big)$. Recalling the element $\overline{\sigma}\in\mathcal{S}_{n-r,r}$ from Remark \ref{cancext}, we also get $w^{\sigma}(\tau^{(r)})^{t}=w^{\overline{\sigma}}\big((\tau^{t})^{[n-r]})\big)$ and $w^{\sigma}(\tau^{(r)})^{r}=w^{\overline{\sigma}}\big((\tau^{r})^{[n-r]})\big)$. In particular, the compositions appearing as subscripts in the equations yielding the formulae the antipode of $\mathrm{F}_{\alpha}^{\operatorname{Id}_{n}}$ (via $\tau$), of $\mathrm{F}_{\alpha^{c}}^{\operatorname{Id}_{n}}$ (using $\tau^{c}$), of $\mathrm{F}_{\alpha^{t}}^{\operatorname{Id}_{n}}$ (through $\tau^{t}$), and of $\mathrm{F}_{\alpha^{r}}^{\operatorname{Id}_{n}}$ (with the parameter $\tau^{r}$) are the same, provided that the first two are carried out in the convention from Remark \ref{firstS} and the latter two use Remark \ref{altproof}. \label{symeq}
\end{prop}
Note that the permutations in the superscripts of the elements containing the words from Proposition \ref{symeq} are $\psi^{\sigma}(\operatorname{Id}_{r},w^{0}_{n-r})$ for the former two and $\psi^{\sigma}(w^{0}_{r},\operatorname{Id}_{n-r})$ for the latter two. Similarly, using the formulae from Remark \ref{altproof} with $\alpha=1^{n}$ or with $\ell(\alpha)=1$ yield the formulae $\sum_{r=0}^{n}(-1)^{n-r}\mathrm{F}_{r}^{w^{0}_{r}}\mathrm{F}_{1^{n-r}}^{\operatorname{Id}_{n-r}}=0$ and $\sum_{r=0}^{n}(-1)^{n-r}\mathrm{F}_{1^{r}}^{w^{0}_{r}}\mathrm{F}_{n-r}^{\operatorname{Id}_{n-r}}=0$, producing the alternative proof of Theorem \ref{SFalphaw0}, via Remark \ref{firstS}.

\medskip

We have already seen that Theorems \ref{SNCQSymF} and \ref{SFalphaw0} cover all the possible cases when $n\geq2$. For larger values of $n$, it seems that the only cases where the antipode takes a fundamental non-commutative quasi-symmetric function to another single one (up to sign) are those covers in these theorems, and no other case. We exemplify this in the case $n=3$.
\begin{ex}
When we consider $\mathrm{F}_{21}^{321}$, Proposition \ref{coprodNCQSym} the next terms in the vanishing formula from the proof of Theorem \ref{SNCQSymF} is $\mathrm{F}_{1}^{1}\mathrm{F}_{11}^{21}$ (or $\mathrm{F}_{2}^{21}\mathrm{F}_{1}^{1}$ via Remark \ref{altproof}). The three options of $\tau$ for this product produce $\mathrm{F}_{21}^{132}+\mathrm{F}_{12}^{312}+\mathrm{F}_{111}^{321}$ from Example \ref{exprod}, as well as $\mathrm{F}_{21}^{132}+\mathrm{F}_{111}^{312}+\mathrm{F}_{12}^{321}$ and $\mathrm{F}_{111}^{132}+\mathrm{F}_{21}^{312}+\mathrm{F}_{12}^{321}$, none of which cancel with the original term $\mathrm{F}_{21}^{321}$ (this term does not appear in any of the three expressions for the other product, appearing in Examples \ref{exprod} and \ref{altprod}). When we take $\mathrm{F}_{12}^{213}$, the next product is $\mathrm{F}_{1}^{1}\mathrm{F}_{2}^{12}$, where the trivial case from Example \ref{exprod} is the only one with which the initial term cancels. But then none of the expressions for the product $\mathrm{F}_{2}^{12}\mathrm{F}_{1}^{1}$ cancels both of the terms $\mathrm{F}_{3}^{123}$ and $\mathrm{F}_{21}^{231}$ simultaneously. Similarly, in two of the forms for $\mathrm{F}_{11}^{21}\mathrm{F}_{1}^{1}$ that show up in Examples \ref{exprod} and \ref{altprod} a term cancels with the initial one $\mathrm{F}_{12}^{213}$, but no expression for $\mathrm{F}_{1}^{1}\mathrm{F}_{11}^{21}$ contains $\mathrm{F}_{21}^{231}$ or $\mathrm{F}_{111}^{231}$. \label{SFn=3}
\end{ex}
For evaluating the antipode on all fundamental non-commutative quasi-symmetric functions of degree 3, we note the sign $(-1)^{n}=(-1)$ appearing in Theorems \ref{SNCQSymF} and \ref{SFalphaw0}, and recall that $\alpha_{|1}=\alpha_{n-1|}=1$ and $\rho_{|1}=\rho_{n-1|}=1$. We can thus apply Lemma \ref{grconS} and write $-S(\mathrm{F}_{\alpha}^{\rho})$ with $\alpha\vDash3$ and $\rho \in S_{3}$ as either $\mathrm{F}_{\alpha}^{\rho}-\mathrm{F}_{\alpha_{|2}}^{\rho_{|2}}\mathrm{F}_{1}^{1}+\mathrm{F}_{1}^{1}S(\mathrm{F}_{\alpha_{1|}}^{\rho_{1|}})$ or $\mathrm{F}_{\alpha}^{\rho}-\mathrm{F}_{1}^{1}\mathrm{F}_{\alpha_{1|}}^{\rho_{1|}}+S(\mathrm{F}_{\alpha_{|2}}^{\rho_{|2}})\mathrm{F}_{1}^{1}$, and evaluate the antipodes in degree 2 by our theorems. These calculations show that the value of $-S(\mathrm{F}_{\alpha}^{\rho})$ when $n=3$ are as follows.
\begin{prop}
The expression $-S(\mathrm{F}_{\alpha}^{\rho})$, which is determined when $\rho=123$ or if $\rho=321$ and $\alpha$ is 3 or 111 by Theorems \ref{SNCQSymF} and \ref{SFalphaw0}, attains, for the latter value of $\rho$, the values $\mathrm{F}_{12}^{123}-\mathrm{F}_{12}^{321}+\mathrm{F}_{21}^{321}$ and $\mathrm{F}_{21}^{123}-\mathrm{F}_{21}^{321}+\mathrm{F}_{12}^{321}$ when $\alpha$ is 21 and 12 respectively. For the other choices of $\rho$, the rows of the following table are ordered by $\alpha$ being 3, 21, 12, and 111 respectively. \[\begin{tabular}{|c|c|c|c|} \hline $132$ & $213$ & $231$ & $312$ \\ \hline $\!\mathrm{F}_{111}^{231}\!+\!\mathrm{F}_{12}^{213}\!-\!\mathrm{F}_{12}^{312}\!$ & $\!\mathrm{F}_{111}^{312}\!+\!\mathrm{F}_{21}^{132}\!-\!\mathrm{F}_{21}^{231}\!$ & $\!\mathrm{F}_{12}^{213}\!-\!\mathrm{F}_{12}^{312}\!-\!\mathrm{F}_{3}^{132}\!+\!\mathrm{F}_{3}^{231}\!+\!\mathrm{F}_{111}^{231}\!$ & $\!\mathrm{F}_{21}^{132}\!-\!\mathrm{F}_{21}^{231}\!-\!\mathrm{F}_{3}^{213}\!+\!\mathrm{F}_{3}^{312}\!+\!\mathrm{F}_{111}^{312}\!$ \\ \hline $\!\mathrm{F}_{21}^{231}\!+\!\mathrm{F}_{12}^{213}\!-\!\mathrm{F}_{12}^{312}\!$ & $\!\mathrm{F}_{21}^{312}\!+\!\mathrm{F}_{12}^{132}\!-\!\mathrm{F}_{12}^{231}\!$ & $\!\mathrm{F}_{12}^{213}\!-\!\mathrm{F}_{12}^{312}\!-\!\mathrm{F}_{21}^{132}\!+\!2\mathrm{F}_{21}^{231}\!$ & $\!\mathrm{F}_{12}^{132}\!-\!\mathrm{F}_{12}^{231}\!-\!\mathrm{F}_{21}^{213}\!+\!2\mathrm{F}_{21}^{312}\!$ \\ \hline $\!\mathrm{F}_{12}^{231}\!+\!\mathrm{F}_{21}^{213}\!-\!\mathrm{F}_{21}^{312}\!$ & $\!\mathrm{F}_{12}^{312}\!+\!\mathrm{F}_{21}^{132}\!-\!\mathrm{F}_{21}^{231}\!$ & $\!\mathrm{F}_{21}^{213}\!-\!\mathrm{F}_{21}^{312}\!-\!\mathrm{F}_{12}^{132}\!+\!2\mathrm{F}_{12}^{231}\!$ & $\!\mathrm{F}_{21}^{132}\!-\!\mathrm{F}_{21}^{231}\!-\!\mathrm{F}_{12}^{213}\!+\!2\mathrm{F}_{12}^{312}\!$ \\ \hline $\!\mathrm{F}_{3}^{231}\!+\!\mathrm{F}_{21}^{213}\!-\!\mathrm{F}_{21}^{312}\!$ & $\!\mathrm{F}_{3}^{312}\!+\!\mathrm{F}_{12}^{132}\!-\!\mathrm{F}_{12}^{231}\!$ & $\!\mathrm{F}_{21}^{213}\!-\!\mathrm{F}_{21}^{312}\!-\!\mathrm{F}_{111}^{132}\!+\!\mathrm{F}_{111}^{231}\!+\!\mathrm{F}_{3}^{231}\!$ & $\!\mathrm{F}_{12}^{132}\!-\!\mathrm{F}_{12}^{231}\!-\!\mathrm{F}_{111}^{213}\!+\!\mathrm{F}_{111}^{312}\!+\!\mathrm{F}_{3}^{312}\!$ \\ \hline \end{tabular}\]
\label{antipode3}
\end{prop}
Since each $\mathrm{F}_{\beta}^{\varphi}$ contains the monomial non-commutative quasi-symmetric function, and none of the canceling correction terms in Proposition \ref{antipode3} have the superscript $\rho^{t}$, we deduce that with this value of $n$, any expression that is not covered by our theorems produces an antipode that cannot be expressed as single non-commutative quasi-symmetric function of any of the sorts appearing in Definition \ref{fundNCQSym}.

Recall that Proposition \ref{SNCQSymM} produces the respective values $\mathrm{E}_{3}^{321}=\mathrm{M}_{(123)}$, $-\mathrm{E}_{12}^{321}=-\mathrm{M}_{(123)}-\mathrm{M}_{(3,12)}$, $-\mathrm{E}_{21}^{321}=-\mathrm{M}_{(123)}-\mathrm{M}_{(23,1)}$, and $\mathrm{E}_{111}^{321}=\mathrm{F}_{3}^{321}$ for $-S(\mathrm{M}_{\alpha}^{123})$ where $\alpha$ is 3, 21, 12, and 111 respectively (note the minus sign), namely for $-S(\mathrm{M}_{(123)})$, $-S(\mathrm{M}_{(12,3)})$, $-S(\mathrm{M}_{(12,3)})$, and $-S(\mathrm{M}_{(1,2,3)})$ respectively. The two intermediate values can also be obtained from Proposition \ref{antipode3}, as these are the $-S$-image of $\mathrm{F}_{21}^{213}-\mathrm{F}_{111}^{213}$ and $\mathrm{F}_{12}^{132}-\mathrm{F}_{111}^{132}$ respectively as well. The value of $-S(\mathrm{M}_{(3,2,1)})=-S(\mathrm{F}_{111}^{321})$ is $\mathrm{F}_{3}^{123}$ via Theorem \ref{SFalphaw0}, and for the other elements of the form $\mathrm{M}_{\mathbf{B}_{\rho}}$ with $\rho \in S_{3}$, they are given via the lower row in the table from Proposition \ref{antipode3}.

Similar consideration produce the following consequence.
\begin{cor}
We have $-S(\mathrm{M}_{(13,2)})=\mathrm{M}_{(13,2)}-\mathrm{M}_{(123)}-\mathrm{M}_{(3,12)}-\mathrm{M}_{(12,3)}$ and  $-S(\mathrm{M}_{(2,13)})=\mathrm{M}_{(2,13)}-\mathrm{M}_{(123)}-\mathrm{M}_{(23,1)}-\mathrm{M}_{(1,23)}$. In addition, we have the equality $-S(\mathrm{M}_{(23,1)})=\mathrm{M}_{(23,1)}-\mathrm{M}_{(123)}-\mathrm{M}_{(12,3)}-\mathrm{M}_{(3,12)}$, and finally $-S(\mathrm{M}_{(3,12)})=\mathrm{M}_{(3,12)}-\mathrm{M}_{(123)}-\mathrm{M}_{(1,23)}-\mathrm{M}_{(23,1)}$. \label{antiM3}
\end{cor}

\begin{proof}
The parameters of $-S$ can be expressed as $\mathrm{F}_{21}^{312}-\mathrm{F}_{111}^{312}=\mathrm{F}_{21}^{132}-\mathrm{F}_{111}^{132}$, $\mathrm{F}_{12}^{231}-\mathrm{F}_{111}^{231}=\mathrm{F}_{12}^{213}-\mathrm{F}_{111}^{213}$, $\mathrm{F}_{21}^{321}-\mathrm{F}_{111}^{321}=\mathrm{F}_{21}^{231}-\mathrm{F}_{111}^{231}$, and  $\mathrm{F}_{12}^{321}-\mathrm{F}_{111}^{321}=\mathrm{F}_{12}^{132}-\mathrm{F}_{111}^{132}$ respectively, and we substitute the expressions arising from Proposition \ref{antipode3} to get the desired results. This proves the corollary.
\end{proof}
Note that the proof of Corollary \ref{antiM3} also verifies some of the linear relations in the degree 3 part of $\mathtt{NCQSym}$, like those appearing in Example \ref{FAex}, and their behavior with respect to (minus) the antipode. Other relations, like $\mathrm{F}_{3}^{213}-\mathrm{F}_{12}^{213}=\mathrm{F}_{3}^{123}-\mathrm{F}_{12}^{123}$ or $\mathrm{F}_{3}^{321}-\mathrm{F}_{21}^{321}=\mathrm{F}_{3}^{312}-\mathrm{F}_{21}^{312}$, produce the equalities resulting from Proposition \ref{SNCQSymM} and Corollary \ref{antiM3} combined with the fact that $-S(\mathrm{M}_{(123)})=\mathrm{M}_{(123)}$ by the primitivity of this element.

\medskip

We recall from Corollary \ref{deg2QSymq} that the square of the antipode on $\mathtt{QSym}^{(q)}$ has infinite order, by checking its action on the part of degree 2. For $\mathtt{NCQSym}$ the square is trivial on degree 2, but we obtain the following result using degree 3.
\begin{cor}
The operator $S^{2}-\operatorname{Id}$, when restricted to the degree 3 part of $\mathtt{NCQSym}$, is nilpotent of degree 2. In particular, when $\mathbb{Z}$ embeds into $R$, the order of $S^{2}$ on $\mathtt{NCQSym}$ is infinite. \label{deg3NCQSym}
\end{cor}

\begin{proof}
Theorems \ref{SNCQSymF} and \ref{SFalphaw0}, Proposition \ref{antipode3}, and Corollary \ref{antiM3} show that the operator $-S$ leaves $\mathrm{M}_{(123)}$, $\mathrm{M}_{(2,3,1)}-\mathrm{M}_{(1,3,2)}$, and $\mathrm{M}_{(3,1,2)}-\mathrm{M}_{(2,1,3)}$ invariant, and also interchanges $\mathrm{M}_{(1,2,3)}$ with $\mathrm{F}_{3}^{321}$ and $\mathrm{M}_{(3,2,1)}$ with $\mathrm{F}_{3}^{123}$. By setting $X:=\mathrm{M}_{(1,23)}+\mathrm{M}_{(23,1)}+\mathrm{M}_{(123)}$ and $Y:=\mathrm{M}_{(12,3)}+\mathrm{M}_{(3,12)}+\mathrm{M}_{(123)}$ and we get that $-S$ subtracts $X$ from $\mathrm{M}_{(1,23)}$, $\mathrm{M}_{(2,13)}$, and $\mathrm{M}_{(3,12)}$ and subtracts $Y$ from $\mathrm{M}_{(12,3)}$, $\mathrm{M}_{(13,2)}$, and $\mathrm{M}_{(23,1)}$, from which we deduce that it sends $X$ to $-Y$ and $Y$ to $-X$, and it also adds $X+Y$ to $\mathrm{M}_{(1,3,2)}+\mathrm{M}_{(2,1,3)}$.

It follows that $S^{2}=(-S)^{2}$ fixes $\mathrm{M}_{(123)}$, $\mathrm{M}_{(2,3,1)}-\mathrm{M}_{(1,3,2)}$, $\mathrm{M}_{(3,1,2)}-\mathrm{M}_{(2,1,3)}$, $\mathrm{M}_{(1,3,2)}+\mathrm{M}_{(2,1,3)}$, $\mathrm{M}_{(1,2,3)}$, $\mathrm{M}_{(3,2,1)}$, and it adds to each of $\mathrm{M}_{(1,23)}$, $\mathrm{M}_{(2,13)}$, $\mathrm{M}_{(3,12)}$, $\mathrm{M}_{(12,3)}$, $\mathrm{M}_{(13,2)}$, and $\mathrm{M}_{(23,1)}$ either $X-Y$ or $Y-X$. Therefore this submodule, of rank 12 inside our part of degree 3 (which has rank 13), is sent by the operator $S^{2}-\operatorname{Id}$ to its submodule that is generated by $X-Y$, and this element is invariant under $-S$ hence is annihilated by $S^{2}-\operatorname{Id}$.

It remains to consider one more generator, say $\mathrm{M}_{(1,3,2)}$ (or $\mathrm{M}_{(2,1,3)}$), to which the operator $-S$ adds the sum of the $-S$-invariant elements $\mathrm{M}_{(2,3,1)}-\mathrm{M}_{(1,3,2)}$, $-\mathrm{M}_{(3,1,2)}+\mathrm{M}_{(2,1,3)}$, and $\mathrm{M}_{(123)}$ (resp. $-\mathrm{M}_{(2,3,1)}+\mathrm{M}_{(1,3,2)}$, $\mathrm{M}_{(3,1,2)}-\mathrm{M}_{(2,1,3)}$, and $\mathrm{M}_{(123)}$), and the combination $\mathrm{M}_{(2,13)}+\mathrm{M}_{(12,3)}-\mathrm{M}_{(13,2)}+\mathrm{M}_{(23,1)}$ (resp. $\mathrm{M}_{(13,2)}+\mathrm{M}_{(1,23)}-\mathrm{M}_{(2,13)}+\mathrm{M}_{(3,12)}$). As another application of $-S$ subtracts $X+Y$ from that combination, the image of $\mathrm{M}_{(1,3,2)}$ (resp. $\mathrm{M}_{(2,1,3)}$) under $S^{2}-\operatorname{Id}$ is the sum of $2\mathrm{M}_{(2,3,1)}-2\mathrm{M}_{(1,3,2)}-2\mathrm{M}_{(3,1,2)}+2\mathrm{M}_{(2,1,3)}+2\mathrm{M}_{(2,13)}-2\mathrm{M}_{(13,2)}$ and $\mathrm{M}_{(12,3)}-\mathrm{M}_{(1,23)}+\mathrm{M}_{(23,1)}-\mathrm{M}_{(3,12)}$ (resp. their additive inverses), which altogether is $S^{2}$-invariant, the nilpotency of the asserted order follows.

As simple matrix power calculation shows that a non-trivial Jordan block of eigenvalue 1 has infinite order when $\mathbb{Z}$ embeds into $R$, the assertion about the order is also established. This proves the corollary.
\end{proof}
Note that the elements that is added to $\mathrm{M}_{(1,3,2)}$ (or $\mathrm{M}_{(2,1,3)}$) in the proof of Corollary \ref{deg3NCQSym} is linearly independent from $X-Y$ there. Hence the rank of the eigenspace of $S^{2}$ on this module of rank 13 is 11, and the matrix for $S^{2}$ in the basis resulting from the proof has two Jordan blocks of size 2 plus 9 trivial blocks of size 1 (all with the eigenvalue 1).

\begin{rmk}
Recall from Remark \ref{NCSym} that $\mathtt{NCQSym}$ contains $\mathtt{NCSym}$ as a Hopf subalgebra. We have $\mathrm{m}_{(1)}=\mathrm{M}_{(1)}$, $\mathrm{m}_{(12)}=\mathrm{M}_{(12)}$, and $\mathrm{m}_{(1,2)}=\mathrm{M}_{(1,2)}+\mathrm{M}_{(2,1)}$ in degrees $n\leq2$, the first two of which are primitive, and the antipode takes the latter to $\mathrm{m}_{(1,2)}+2\mathrm{m}_{(12)}$, as follows from Theorems \ref{SNCQSymF} and \ref{SFalphaw0}. When $n=3$ we have the primitive element $\mathrm{m}_{(123)}=\mathrm{M}_{(123)}$, the element $\mathrm{m}_{(1,2,3)}=\sum_{\rho \in S_{3}}\mathrm{M}_{\mathbf{B}_{\rho}}$, and the three sums $\mathrm{m}_{(12,3)}=\mathrm{M}_{(12,3)}+\mathrm{M}_{(3,12)}$, $\mathrm{m}_{(13,2)}=\mathrm{M}_{(13,2)}+\mathrm{M}_{(2,13)}$, and $\mathrm{m}_{(1,23)}=\mathrm{M}_{(1,23)}+\mathrm{M}_{(23,1)}$. In the proof of Corollary \ref{deg3NCQSym} we saw that $-S$ subtracts $X+Y=\mathrm{m}_{(1,23)}+\mathrm{m}_{(12,3)}+2\mathrm{m}_{(123)}$ from any of the latter three sums. We write the remaining expression, as the sum of $\mathrm{M}_{(2,3,1)}-\mathrm{M}_{(1,3,2)}$ and $\mathrm{M}_{(3,1,2)}-\mathrm{M}_{(2,1,3)}$ (both of which are $-S$-invariant), $2\mathrm{M}_{(1,3,2)}+2\mathrm{M}_{(2,1,3)}$ (to which $-S$ adds $2X+2Y$), and $\mathrm{M}_{(1,2,3)}+\mathrm{M}_{(3,2,1)}$. As $\mathrm{F}_{3}^{321}+\mathrm{F}_{3}^{123}$ is the latter plus $X+Y$, we get $-S(\mathrm{m}_{(1,2,3)})=\mathrm{m}_{(1,2,3)}+3X+3Y$. \label{NCSymS}
\end{rmk}
In particular, Remark \ref{NCSymS} and Corollary \ref{deg3NCQSym} imply that $S^{2}$ is the identity on the relevant parts of the subalgebra $\mathtt{NCSym}$, which should indeed be the case, via Proposition \ref{Sprop}, as the latter is co-commutative.

\noindent\textsc{Einstein Institute of Mathematics, the Hebrew University of Jerusalem, Edmund Safra Campus, Jerusalem 91904, Israel}

\noindent E-mail address: zemels@math.huji.ac.il

\end{document}